\documentclass[preprint,12pt]{elsarticle}

\usepackage{amsmath,amssymb,dsfont,amscd,amsthm,amsfonts}
\usepackage{mathrsfs}  
\DeclareMathAlphabet\mathbfcal{OMS}{cmsy}{b}{n}
\usepackage [ruled] {algorithm2e} 
\usepackage{mathtools}
\usepackage{mathbbol}
\usepackage{todonotes}
\usepackage{centernot}
\usepackage{stmaryrd}
\usepackage{comment}
\usepackage{graphicx}
\usepackage{caption}
\usepackage{subcaption}
\usepackage{float}
\usepackage{soul}
\usepackage{tikz}
\usepackage{bm}
\usepackage{xcolor}
\usepackage{xcolor,colortbl}

\usepackage{hyperref}

\hypersetup{
    colorlinks=true,
    linkcolor=red,
    filecolor=magenta,      
    urlcolor=cyan,
    citecolor=green
}
\usepackage[most]{tcolorbox}

\newtcolorbox{shadedcvbox}[1][]{enhanced jigsaw,
  colback=white!80!blue,
  coltext={black},
  boxrule=0.5pt,
  arc=3mm,
  auto outer arc,
  boxsep=3pt,
  left=4pt,
  right=2pt,
  bottom=2pt,
  top=2pt,
  #1}

\newtheorem{thm}{Theorem}[section]

\newtheorem{lem}[thm]{Lemma}

\theoremstyle{definition}
\newtheorem{defn}[thm]{Definition}
\theoremstyle{remark}
\newtheorem{rem}[thm]{Remark}

\theoremstyle{definition}
\newtheorem{algo}[thm]{Algorithm}

\graphicspath{{figs/}}

\newcommand{\RR}{\mathbb{R}}
\newcommand{\XX}{\mathbf{X}} 
\newcommand{\uu}{\mathbf{u}} 
\newcommand{\vv}{\mathbf{v}} 
\newcommand{\by}{\mathbf{y}} 
\newcommand{\bx}{\mathbf{x}} 
\newcommand{\bom}{\mathbf{m}} 
\newcommand{\bS}{\mathbf{S}} 
\newcommand{\bI}{\mathbf{I}} 
\newcommand{\bP}{\mathbf{P}} 
\newcommand{\bU}{\mathbf{U}} 
\newcommand{\bV}{\mathbf{V}} 
\newcommand{\bY}{\mathbf{Y}} 
\newcommand{\bZ}{\mathbf{Z}} 
\newcommand{\bSig}{\boldsymbol{\Sigma}} 
\newcommand{\bThe}{\boldsymbol{\Theta}} 
\newcommand{\spn}{\text{span}}
\newcommand{\Inf}{\text{inf}}

\newcommand{\VecMat}[2]{\mathrm{Mat}_{#1,#2}(\RR)}

\DeclareMathOperator{\tr}{tr}
\DeclareMathOperator{\Exp}{Exp}
\DeclareMathOperator{\Log}{Log}

\newcommand{\red}[1]{{\textcolor{red}{#1}}}

\journal{Computers and Structures}

\begin{document}

\begin{frontmatter}

\title{On the stability of POD Basis Interpolation via Grassmann Manifolds for Parametric Model Order Reduction in Hyperelasticity}

\author{Orestis Friderikos\corref{cor1}}
\ead{friderikos@lmt.ens-cachan.fr}

\author{Emmanuel Baranger}
\ead{emmanuel.baranger@ens-paris-saclay.fr}

\author{Marc Olive\corref{cor1}}
\ead{marc.olive@math.cnrs.fr}

\author{David N\'eron}
\ead{david.neron@ens-cachan.fr}
\cortext[cor1]{Corresponding author.}
\address{Université Paris-Saclay, ENS Paris-Saclay, CNRS,  LMT - Laboratoire de Mécanique et Technologie, 91190, Gif-sur-Yvette, France.}

\begin{abstract}

This work considers the stability of Proper Orthogonal Decomposition (POD) basis interpolation on Grassmann manifolds for parametric Model Order Reduction (pMOR) in hyperelasticity. The article contribution is mainly about stability conditions, all defined from strong mathematical background. We show how the stability of  interpolation can be lost if certain geometrical requirements are not satisfied by making a concrete elucidation of the local character of linearization. 
To this effect, we draw special attention to the Grassmannian Exponential map and optimal injectivity condition of this map, related to the cut--locus of Grassmann manifolds. From this, explicit stability conditions are established and can be directly used to determine the loss of injectivity in practical pMOR applications. Another stability condition is formulated when increasing the number $p$ of mode, deduced from principal angles of subspaces of different dimensions $p$. 
This stability condition helps to explain the non-monotonic oscillatory behavior of the error-norm with respect to the number of POD modes, and on the contrary, the monotonic decrease of the error-norm in the two benchmark numerical examples considered herein. Under this study, pMOR is applied in hyperelastic structures using a non-intrusive approach for inserting the interpolated spatial POD ROM basis in a commercial FEM code. The accuracy is assessed by \emph{a posteriori} error norms defined using the ROM FEM solution and its high fidelity counterpart simulation. Numerical studies successfully ascertained and highlighted the implication of stability conditions. The various stability conditions can be applied to a variety of other relevant  problems involving parametrized ROMs generation based on POD basis interpolation via Grassmann manifolds. 
\end{abstract}


\end{frontmatter}


\tableofcontents


\section{Introduction}
\label{}

This work considers  the notion of stability conditions of POD basis interpolation on Grassmann manifolds. This interpolation method is used to adapt Reduced-order models (ROMs) to parameter changes in various scientific fields, among others, design, optimization, control, uncertainty quantification,  data-driven systems, etc. Here we introduce three important stability conditions that are quite essential to the interpolation method. The interesting thing about them is that they do not seem special to problems in hyperelasticity. It may be an illustration of general stability conditions applicable to a variety of problems of other scientific fields as well.

ROMs aim to decrease the computational burden of large-scale systems and solve parametrized problems by generating models with lower complexity, but accurately enough to represent the high fidelity counterpart simulations. One popular method is the Proper Orthogonal Decomposition (POD)~\cite{holmes2012turbulence,Henri2003,Mosquera2019b}, also known as Kharhunen-Lo\`eve Decomposition (KLD)~\cite{Kar1946,loeveprobability}, Singular Value Decomposition (SVD)~\cite{golub1996} or Principal Component Analysis (PCA)~\cite{Jolliffe2002,Abdi2010}. We need to emphasize that all these POD techniques are referred as \emph{a posteriori} as they require some knowledge (at least partial) on the solution of the problem.

Parametric Model Order Reduction  (pMOR) is used to generate a ROM  that approximates a full-order system with high accuracy over a range of parameters. In case of solving a parametric problem using the POD, the method starts by a sampling stage during which the full-order system is solved for some rather small set of \emph{training} points. The state variable field \emph{`snapshots'} are then compressed using the POD to generate a ROM basis that is expected to reproduce the most characteristic dynamics of its high-fidelity counterpart. Nevertheless, since the POD bases are generated for a set of training points, they are optimal only to these parameter values.
Thus, a main drawback of POD is the sensitivity to parameter changes and the lack of robustness over the entire parameter space. Consequently, any ROM basis generated by the approach outlined above cannot be expected to give a good approximation away from the training point. In pMOR, the question we have to address is how to compute a good approximation of the POD basis related to a \emph{new parameter} value. Multiple methods have been proposed for adapting POD basis to address parameter variation as thoroughly documented in  related review articles~\cite{Benner2015,Zimmermann2019,Cueto2014}.

For nonlinear systems, even though a Galerkin projection reduces the number of unknowns, the computational burden for obtaining the solution could still be high due to the prohibitive computational costs involved in the evaluation of nonlinear terms. Hence, the nonlinear Galerkin projection in principle leads to a ROM but its evaluation could be more expensive than the evaluation of the original problem. To this effect, to make the resulting ROMs computationally efficient, POD is typically used together with a sparse sampling method, also called \emph{hyper reduction}, such as the missing point estimation (MPE)~\cite{Astrid_2008}, the empirical interpolation method (EIM)~\cite{Radermacher2016}, the discrete empirical interpolation method (DEIM)~\cite{Chaturantabut2010}, the Gappy POD method~\cite{Everson1995}, and the Gauss-Newton with approximated tensors (GNAT) method~\cite{Carlberg2013}.

Parametric Model Order Reduction using \emph{POD basis interpolation} is done initially in the field of computational fluid dynamics which was proposed for parametrized systems that are linear in state~\cite{Mosquera2019b,Farhat2008,Amsallem2009,Mosquera2018}. Similar approach has been scarcely applied in hyperelasticity, like in~\cite{Niroomandi2012}, where they propose real time simulations of hypepelastic structures using POD basis interpolation, in combination with an asymptotic numerical method. Here, pMOR is used to hyperelastic structures by adapting pre-computed POD basis.

When addressing the question of POD basis interpolation, the main point is that \emph{interpolation cannot be done in a linear space}. Indeed, any mode $p$ POD basis performed on some matrix $\bS\in \VecMat{n}{N_t}$ give rise to a truncated matrix $\bS_{p}\in \VecMat{n}{p}$ (where $n=3N_s$ and $N_s$, $N_t$ respectively correspond to the number of spatial points and time points). Now, despite the appearances, computation can not be done in the linear space $\VecMat{n}{p}$ of matrices, as the matrix $\bS_{p}$ encodes a $p$ dimensional vector subspace. The goal is thus to make interpolation on the set of $p$ dimensional subspaces of $\RR^{n}$, which defines exactly the Grassmann manifold $\mathcal{G}(p,n)$. Such Grassmann manifold interpolation is well documented~\cite{Amsallem2009,Mosquera2018,Mosquera2019b,Edelman1998g,Absil2004}, all coming from the fluid mechanics community, and computation can be done explicitly.

Thus, we might have been satisfied with a simple application of the existing and now well-known formulas, using the \emph{logarithm map} to linearize, and then the \emph{exponential map} to return back to the manifold. Such maps are issued from the riemannian structure of Grassmann $\mathcal{G}(p,n)$ and its associated geodesics~\cite{Gallot1990}. A first condition appears, as the logarithm map is only defined on some subset $\mathrm{U}\subset \mathcal{G}(p,n)$ explicitly defined as a subset of non singular matrices. So linearization can only be done once we have checked that all training points are contained in $\mathrm{U}$. In fact, such a condition is usually checked, as square matrices are generically non--singular.

A second condition concerns the use of the exponential map, which is defined on all the vector space $\RR^{d}$ (with $d=p(n-p)$ the dimension of $\mathcal{G}(p,n)$). Nevertheless, it is only \emph{injective} inside a subset $\mathrm{V}\subset \RR^{d}$ deduced from the \emph{cut--locus}~\cite{Gallot1990} of the Riemannian manifold $\mathcal{G}(p,n)$. Considering all geodesics with the same starting point, such a cut--locus is in fact the set of points where such geodesics are no longer minimal, and thus the exponential map is no more injective. Without any control of such an injectivity condition, the return back of the interpolated curve \emph{via} the exponential map can lead to some disconnected curve on the manifold, which should be avoided.

An explicit determination of such a cut--locus was already mentioned in~\cite{Wong1967}, without any proof, and a result by Kozlov~\cite[Theorem 12.5]{Kozlov2000} make a clear understanding of such a cut-locus using singular values of matrix representation of a velocity vector. We thus write an explicit way to compute such a cut--locus, with clear proof. As this result is not a classical one, and to be self contained, we had to develop the necessary mathematics to obtain such cut--locus of the Grassmann manifold $\mathcal{G}(p,n)$, as well as the open subset $\mathrm{V}$.

In fact, from this cut--locus and its associated subset $\mathrm{V}$, it was possible to improve the already known exponential injectivity condition, obtained from the \emph{injectivity radius} of Grassmann manifolds~\cite{Kozlov2000}, and used in~\cite{Mosquera2019b} to control computations. In most of our cases, indeed, the injectivity condition issued from the cut--locus is better than the one obtained from injectivity radius.

A third stability condition considered here is related to the intrinsic non-inclusion defect of the interpolated subspaces of different dimensions. Numerical results showed indeed that the accuracy of interpolation may not improve by increasing the POD modes.
A consequence is that it is not possible to control or predict the interpolation behavior. At first glance, this fact seems inconsistent with the expected improvement of the solution by increasing the number of modes. We indicate that the non-connectivity of the solutions is inherited from the construction of the  interpolation formulae using the Logarithm and the Exponential maps. To prove the fact, our basic tool is the computation of the principal angles of two POD basis of different mode $p$. This enables us to compute the geometric distance between subspaces of different dimension~\cite{Ye2016}. To this end, a new stability condition will be tied with the geometric distance which measures the non-inclusion defect between these subspaces. To the best of the author's knowledge, this finding has never been reported in the variety of ROM problems involving POD basis  interpolation on Grassmann manifolds.

From all this, we finally get three kinds of stability condition, each clearly established: $(1)$ a first one about the logarithm map domain of definition, $(2)$ a second one on the loss of injectivity of the exponential map, \emph{via} the cut--locus of Grassmann manifolds and $(3)$ a third one about the increasing mode, controlled from a well-defined geometric distance between subspaces of different dimensions.

Considering the \emph{mechanical part}, the overall procedure comprises an off-line and an on-line stage. The off-line stage characterizes the potentially costly procedure of solving FEM problems associated with different values of the physical or modeling parameter (training points). The on-line stage consists of the POD basis interpolation on Grassmann manifolds to determine a ROM basis for an unseen target parameter. Then, a non-intrusive approach is introduced for the obtained spatial POD basis. Note, that this approach deviates from the POD methods that relying on a Galerkin/Petrov Galerkin projection on the governing equations. Instead, the ROM-FEM models are implemented by inserting the interpolated spatial POD basis using linear constraint equations in Abaqus.
It is evident that, by constraining the degrees of freedom, the reduced model still embeds the high dimension. We remark that we followed this approach using a commercial code only for evaluating the stability and accuracy of the adaption of POD basis via interpolation on Grassmann manifolds. This is because it is not our objective to implement a method of nonlinear model reduction for the effective evaluation of the nonlinear terms, although it is a quite challenging task to be realized inside a commercial FEM code.

In our applications we employed benchmark hyperelastic structures to elaborate the stability loss of POD basis interpolation even at low complexity models. We expect that the stability issues discussed herein will be also inherent and  critical for more demanding problems in hyperleasticity, to mention among others computation of soft tissues, blood vessels, human skin inflation, human Mitral valve, etc. For the pMOR, two hyperelastic structures modeled with isotropic and anisotropic constitutive laws are studied. Specifically, for the anisotropic model, a subclass of transversely isotropic materials is considered. In this subclass, the strain energy function is assumed to depend only on two invariant measures of finite deformation~\cite{Holzapfel2007,bonet1998simple,almeida1998finite,itskov2001generalized}. At the numerical examples, the decision made is to enter the parameters in two ways considering a) the model anisotropy defined by the fiber orientation angle, and b) the material coefficients of the hyperelastic constitutive equations.     

\textbf{Organization of the article} \\
The present paper is organized as follows. In~\autoref{sec:Problem_Formulation} and~\autoref{sec:POD} we recall the theoretical background so to understand the way to make interpolation of POD bases using the corresponding points on a Grassmann manifold. Then~\autoref{sec:CF_Interpolation_GM} produces all explicit algorithm to obtain interpolation on Grassmann manifolds, and we also define three stability conditions: one from the logarithm map, a second one from the exponential map, and a third one from increasing POD modes. The mechanical part starts with~\autoref{sec:App_HyperElast}, which covers the framework of hyperelasticity theory in continuum mechanics for an incompressible transverse isotropic material. In ~\autoref{sec:Numerical_Investigations},  the interpolation performance using two hyperelastic structures is shown, and  further important computational aspects are discussed. Finally, ~\autoref{sec:Conclusions} highlights the main results and some important outcomes. The~\ref{sec:Grassmann_Manifolds} is devoted to the mathematical proofs needed to have well-defined stability conditions, as for instance an explicit determination of the cut--locus of Grassmann manifolds.

\section{Problem Formulation}
\label{sec:Problem_Formulation}

We consider some mechanical problem governed by a specific parameter $\lambda \in [\lambda_{min},\lambda_{max}]\subset \RR$, which comes from hyperelasticity in our situation (see \autoref{sec:App_HyperElast}). For each parameter $\lambda$, the solution is given by a space-time smooth field
\begin{equation*}
(t,\XX)\in [0;T]\times \Omega_{0}\mapsto u^{\lambda}(\XX,t)\in \RR^3
\end{equation*}
where $\Omega_0$ is a closed convex subset of $\RR^3$ and $T>0$.

To avoid costly computations for all values $\lambda\in [\lambda_{min},\lambda_{max}]$, we would like to interpolate between a finite number of FEM solutions $u_i:=u^{\lambda_i}$, associated to $N$ training points $\lambda_1,\dotsc,\lambda_N$. In fact, it is at the level of the POD performed on the snapshot matrices $\bS(\lambda_i)$ (defined in the next section) associated to the solutions $u_i$ that this interpolation will be considered.

But one of the essential points of this POD is that it associates to each snapshot matrix $\bS(\lambda_i)$ a certain point $\bom_i$ of a Grassmann manifold $\mathcal{G}$, and it is therefore needed at this stage to interpolate between points $\bom_1,\dotsc,\bom_N$ on $\mathcal{G}$. It is now proposed to detail the link between a POD reduction and the construction of a point on a Grassmann manifold.

\section{Proper Orthogonal Decomposition and Grassmann manifolds}\label{sec:POD}

The POD method can be applied to curves defined in Hilbert spaces of infinite dimension. The initial idea is to determine a subspace of a given dimension $p$ (which is the fixed number of modes of the POD),  reflecting ``as well as possible" this curve, as it is very well explained in~\cite{Henri2003,Mosquera2018}. In most cases, however, we do not consider the entire curve, but only a finite number of points of a Hilbert space $\mathcal{H}_{\text{spatial}}=\RR^{N_s}$ of finite dimension $N_s$ (the number of space points). More precisely any FEM solution $u$ of our problem under consideration produces a \emph{snapshot matrix}
\begin{equation*}
\bS_{jk},\quad 1\leq j\leq 3N_s,\quad 1\leq k\leq N_t
\end{equation*}
with $N_t$ the number of time steps. Such matrix encodes in fact $N_t$ vectors $\uu_{k}:=u(\cdot,t_k)\in \mathcal{H}_{\text{spatial}}$, and we write
\begin{equation*}
\bS:=[\uu_1,\dotsc,\uu_{N_t}]
\end{equation*}

Take now $\langle\cdot,\cdot\rangle$ to be the standard inner product of the Hilbert space $\mathcal{H}_{\text{spatial}}$. To any $p$ dimensional vector subspace $\mathcal{V}_p$ of $\mathcal{H}_{\text{spatial}}$, there is an associated orthogonal projection
\begin{equation*}
\boldsymbol{\pi}_{p} \: : \: \mathcal{H}_{\text{spatial}}\longrightarrow \mathcal{V}_{p}
\end{equation*}
and the POD method address the question to minimize the distance function
\begin{equation*}
\mathcal{J}(\mathcal{V}_p):=\sum_{k=1}^{N_t} \| \uu_{k} -\boldsymbol{\pi}_{p}(\uu_{k})\|^2,\quad \|\cdot\|:=\sqrt{\langle \cdot,\cdot\rangle}
\end{equation*}
over all $p$ dimensional subspaces $\mathcal{V}_p$. It then appears that the set of all such subspaces define a \emph{smooth compact Riemannian manifold}~\cite{Boothby1986,Gallot1990}
\begin{equation*}
\mathcal{G}(p,n):=\left\{\mathcal{V}_p\subset \mathcal{H}_{\text{spatial}},\quad \dim(\mathcal{V}_p)=p\right\},\quad n:=3N_s
\end{equation*}
so that any $p$ dimensional vector subspace $\mathcal{V}_p$ can be considered as some \emph{point} $\bom\in \mathcal{G}(p,n)$, and the question is finally to minimize $\mathcal{J}(\bom)$ over all $\bom \in \mathcal{G}(p,n)$.

In practice, let consider an orthonormal basis $\phi_{1},\dotsc,\phi_{p}$ of $\mathcal{V}_p$ so that the matrix form of $\boldsymbol{\pi}_{p}$ is given by
\begin{equation*}
\boldsymbol{\Phi}_{p}\boldsymbol{\Phi}_{p}^{T},\quad \boldsymbol{\Phi}_p:=[\phi_{1},\dots,\phi_{p}]\in \VecMat{n}{p}
\end{equation*}
where $\VecMat{n}{p}$ is the vector space of $n\times p$ matrices, and (right) superscript ${(\cdot)}^T$ denotes the transposition operation. By direct computation, the distance function $\mathcal{J}$ is then rewritten
\begin{equation*}
\mathcal{J}(\bom)=\|\bS-\boldsymbol{\Phi}_{p}\boldsymbol{\Phi}_{p}^{T}\bS\|_{\text{F}}^{2}
\end{equation*}
where $\|\mathbf{A}\|_{\text{F}}:=\sqrt{\tr(\mathbf{A}\mathbf{A}^T)}$ is the Frobenius norm on $\VecMat{n}{p}$.

Now it is classically known that minimization of $\mathcal{J}$ is given by Eckart--Young Theorem~\cite{Eckart193,Golub1987,golub1996,Stewart1973} and can be obtained via a \emph{singular value decomposition} of $\bS$. Indeed, take this SVD to be
\begin{equation*}
\bS=\bU \bSig \bV^{T},\quad \bU:=[\phi_{1},\dotsc,\phi_{N_t}]
\end{equation*}
with singular values $\sigma_1\geq \sigma_{2}\geq \dotsc \geq \sigma_{N_t}$. Then one solution of minimizing $\mathcal{J}$ is given by
\begin{equation*}
\bom_{0}:=\spn(\phi_{1},\dotsc,\phi_{p})
\end{equation*}
which is unique whenever $\sigma_{p}>\sigma_{p+1}$~\cite{Henri2003}.
Let also define the \emph{reduced model} $\bS_p$ of our snapshot matrix by
\begin{equation*}
\bS_p:=\boldsymbol{\Phi}_{p}\boldsymbol{\Phi}_{p}^{T}\bS,\quad
\boldsymbol{\Phi}_{p}:=[\phi_{1},\dotsc,\phi_{p}].
\end{equation*}
For each snapshot matrix $\bS(\lambda_{i})$ associated to training points $\lambda_{i}$ ($i=1,\dots,N$), we thus obtain a point
\begin{equation*}
\bom_{i}:=\spn(\phi_{1}^{(i)},\dotsc,\phi_{p}^{(i)})\in \mathcal{G}(p,n)
\end{equation*}
once chosen a fix mode $p$ for the POD. For a new target parameter $\widetilde{\lambda}$, interpolation has to be done on the Grassmann manifold $\mathcal{G}(p,n)$, which is now detailed.

\section{ROM Adaptation Based on Interpolation in Grassmann Manifolds}
\label{sec:CF_Interpolation_GM}

Computation on manifold, such as the one of Lagrange interpolation, can only be done using \emph{local coordinates}. Such local coordinates are obtained via bijective maps, which are defined, in general, on subsets $\mathrm{U}$ of the manifold (called the \emph{local charts}). In the case of a Riemannian manifold, one can use the \emph{normal coordinates} directly deduced from the geodesics of the manifold.

In our case, local charts will be given by \emph{logarithm maps}, so we obtain smooth diffeomorphisms
\begin{equation*}
\Log \: : \: \mathrm{U}\longrightarrow \mathrm{V}:=\Log(\mathrm{U})\subset \RR^{d}
\end{equation*}
where $d$ is the dimension of the manifold, and the reverse operation is given by the exponential map. Nevertheless, such operation has to be well-defined, which is achieved when the exponential map is injective.   

Such an injectivity condition was already addressed in the work of Mosquera et al~\cite{Mosquera2019b}, using the injectivity radius of Grassmann manifolds (see ~\eqref{eq:Disk_Inj}). Other injectivity conditions are presented here, less restrictive than the one issued from the injectivity radius (see Remark~\ref{rem:Inj_Radius_wrt_Cut_locus}).

Another issue is the one of increasing the number $p$ of mode. Indeed, one should expect that the interpolation is sharpened by increasing $p$, which can be controlled by using the \emph{geometric distance} computed for subspaces of different dimensions, as defined in~\cite{Ye2016}.

Let us know present in the next~\autoref{subsec:Assumption_Local_Chart} the necessary assumptions to have a well-defined interpolation, while~\autoref{subsec:Interpol_Lagrangian} produce the algorithm to compute an interpolation, taking into account all stability conditions. Finally~\autoref{subec:Geometric_Distances} focus on the explicit formulae to compare two subspaces of different dimensions.

\subsection{Interpolation from logarithm and exponential map: necessary assumptions}\label{subsec:Assumption_Local_Chart}

Let us consider back the points $N$ points $\{\bom_{i}\}^{N}_{i=1}$ in the Grassmann manifold $\mathcal{G}(p,n)$, all obtained from the ROMs of the snapshot matrices (as detailed in \autoref{sec:POD}). The goal here is to obtain a \emph{well-defined} interpolation of a spatial POD basis associated with a new target point $\tilde{\lambda}$. This is detailed in \autoref{subsec:Interpol_Lagrangian}, and we just focus here on the main ideas issued from the seminal work of Amsallem~\cite{Amsallem2009}:
\begin{enumerate}
	\item Choose a base point $\bom_0$ in the family $\bom_1,\dots,\bom_N$, altogether with its associated logarithm map $\Log_{\bom_0}$ (from Definition~\ref{def:Log_map_Grass}).
	\item Compute the velocity vectors $v_i:=\Log_{\bom_0}(\bom_i)$ all lying in a tangent plane, which a vector space $\RR^{d}$ (with $d=p(n-p)$ the dimension of $\mathcal{G}(p,n)$).
	\item Compute a new velocity vector $\widetilde{v}$ associated to a target point $\widetilde{\lambda}$.
	\item Obtain an interpolated point $\widetilde{\bom}:=\Exp_{\bom_0}(\widetilde{v})\in \mathcal{G}(p,n)$ using the exponential map (from~\eqref{eq:Def_Exp_Map}) to return back to the Grassmann manifold $\mathcal{G}(p,n)$.
\end{enumerate}

As depicted in Figure~\ref{fig:Interpolation_Manifold}, it is nevertheless important not to forget that the logarithm map $\Log_{\bom_0}$ is only defined on some open set $\mathrm{U}_{\bom_0}$, taken from~\eqref{eq:Def_Open_Set_Log} and recalled below. So a first necessary condition is that
\begin{itemize}
	\item (C1): All points $\bom_1,\dotsc,\bom_N$ lie in $\mathrm{U}_{\bom_0}$.
\end{itemize}

To check such a condition, recall first that each point $\bom\in \mathcal{G}(p,n)$ correspond to an orthonormal basis stored in a $n\times p$ matrix
\begin{equation*}
\bY=[\by_1,\cdots,\by_p]\in \VecMat{n}{p},\quad \bY^T\bY=\bI_{p}.
\end{equation*}
Taking now matrices $\bY_i$ corresponding to $\bom_i$ ($i=0,\dots,N)$, such condition translate into
\begin{itemize}
	\item (C1)-matrix form: For all $i=1,\dots,N$, the matrix $\bY_0^{T}\bY_i$ is non singular.
\end{itemize}
From this and Theorem~\ref{thm:Exp_Diff_Cut_Locus}--\ref{thm:Cut_Locus_Grass} we deduce that the velocity vectors $v_i=\Log_{\bom_0}(\bom_i)$ all lie in the open set $\mathrm{V}_{\bom_{0}}=\Log_{\bom_0}\left(\mathrm{U}_{\bom_0}\right)$. Once computed the new velocity vector $\widetilde{v}\in \RR^d$, according to Theorem~\ref{thm:Exp_Diff_Cut_Locus}, a second necessary condition is then
\begin{itemize}
	\item (C2): $\widetilde{v}$ is inside the open set $\mathrm{V}_{\bom_{0}}$.
\end{itemize}
Such a condition seems to be more intricate than the previous one, but in fact it is simply related to singular values of a matrix. Indeed, in the case $2p\leq n$ (which will be our case), a velocity vector $\widetilde{v}$ is represented by a matrix $\widetilde{\bZ}\in \VecMat{n}{p}$ such that $\widetilde{\bZ}^T\bY_0=0$ (see~\eqref{eq:Def_Hor_Lift}). From Lemma~\ref{lem:rho_v_Grass} and Theorem~\ref{}, condition (C2) simply writes
\begin{itemize}
	\item (C2)-matrix form: Taking $\widetilde{\theta}_1$ to be the maximum singular value of $\widetilde{\bZ}$, we have $\widetilde{\theta}_1<\pi/2$.
\end{itemize}

The first condition (C1) is usually trivially satisfied, and the second one (C2) can be evaluated on a range of new parameters $\widetilde{\lambda}$, so to have an interval $[\widetilde{\lambda}_a,\widetilde{\lambda}_b]$ of well-defined interpolation. This was done on both benchmarks (see Figure~\ref{fig:Stability_condition_1_pbm_1} and~\ref{fig:Stability_condition_1_pbm_2}).

\begin{rem}\label{rem:Inj_Radius_wrt_Cut_locus}
	In the case of the compact manifold $\mathcal{G}(p,n)$, the exponential map is defined on all the vector space $\RR^{d}$, so it is always possible to compute a new point $\Exp_{\bom_0}(\widetilde{v})$ on the Grassmann manifold, so we obtain an interpolation which can be not well-defined. In the previous work of Mosquera et al.~\cite{Mosquera2018,Mosquera2019b}, an injectivity condition on the exponential map was defined using the \emph{injectivity radius} of $\mathcal{G}(p,n)$, given by\eqref{eq:Disk_Inj}, which translate into
	\begin{equation*}
	\|\widetilde{v}\|=\left(\widetilde{\bZ}^{T}\widetilde{\bZ}\right)^{1/2}=\left(\sum_{i=1}^{p}\widetilde{\theta}_{i}^2\right)^{1/2}<\frac{\pi}{2}
	\end{equation*}
	where $\widetilde{\theta}_{i}$ are the singular values of $\widetilde{\bZ}$, leading to a weaker condition than the (C2) one (see Lemma~\ref{lem:Cut_Better_than_Radius}).
\end{rem}

\begin{rem}[Violation of stability condition (C2) from an application point of view]
	Let us consider the case of the north hemisphere of the $2D$ sphere of radius $1$, with $\bom_{0}=N$ being the North Pole. The tangent plane is simply given by $\RR^{2}$, and to any velocity vector $v\in \RR^2$ corresponds a point on the north hemisphere, using the exponential map. Here, the exponential map is non injective for all $v\in \RR^2$ with length greater than $\pi/2$. If the interpolated curve inside $\RR^2$ is outside the disk of radius $\pi/2$ (see Figure~\ref{fig:Interpolation_Manifold}), then the corresponding interpolated curve on the north hemisphere is disconnected.
\end{rem}

\begin{figure}[H]
\begin{center}
\footnotesize 
\begin{tikzpicture}[scale=0.7]
\node [label={[xshift=0.8cm, yshift=-0.3cm]$N=\bom_0$}] (0) at (-7, 0) {$\bullet$};
\node [label=below:{$\bom_1$}] (1) at (-8, 1) {$\bullet$};
\node [label={[xshift=-0.4cm, yshift=-0.5cm]$\bom_2$}] (2) at (-6.5, -1.25) {$\bullet$};
\node [label=below:{$\bom_3$}] (3) at (-5.5, -2) {$\bullet$};
\node [label={[xshift=-0.4cm, yshift=-0.5cm]$\bom_4$}] (4) at (-4, 0.5) {$\bullet$};
\node [label=below:{$\bom_5$}] (5) at (-6, 3.25) {$\bullet$};
\node [label=below:{$\bom_6$}] (6) at (-7, 2.5) {$\bullet$};
\node [label=below:{$0$}] (7) at (4, 0) {$\bullet$};
\node [label=below:{$v_1$}] (8) at (2, 0.75) {$\bullet$};             
\node [label=below:{$v_2$}] (9) at (4.75, -1) {$\bullet$};
\node [label=below:{$v_3$}] (10) at (5.75, -1.5) {$\bullet$};
\node [label={[xshift=-0.4cm, yshift=-0.5cm]$v_4$}] (11) at (7.25, 0.75) {$\bullet$};
\node [label=below:{$v_5$}] (12) at (4, 3.5) {$\bullet$};
\node [label=below:{$v_6$}] (13) at (3, 2.5) {$\bullet$};
\node (14) at (-6.75, 0.5) {};
\node (15) at (-6.75, 0.5) {};
\node (16) at (-5.3, 3) {};
\node (17) at (2.4, 3) {};
\node (18) at (1.25, -3) {};
\node (19) at (-5.5, -3) {};
\node (20) at (-1.75, 4) {$\Log_{\bom_0}$};
\node (21) at (-1.95, -3.95) {$\Exp_{\bom_0}$};
\node [red,label={[xshift=0.3cm, yshift=-0.2cm]$\red{\widetilde{\bom}}$}] (23) at (-9.76, 1.26) {$\bullet$};
\node [red,label={[xshift=0.3cm, yshift=-0.2cm]$\red{\widetilde{v}}$}] (24) at (8.1, -0.84) {$\bullet$};
\node (25) at (6.5, 4.14) {\red{$\Log_{\bom_{0}}\circ \Exp_{\bom_{0}}(\widetilde{v})\neq \widetilde{v}$}};
\draw [->,>=latex,bend left=15] (16.center) to (17.center);
\draw [->,>=latex,bend left=15] (18.center) to (19.center);
\draw[->,>=latex] (-7,0)--(-9.7,-2.53) node[yshift=0.5cm,midway] {$\pi/2$};
\draw (-7,0) circle(3.7);
\draw (4,0) circle(3.7);
\draw [line width=1pt] plot [smooth] coordinates {(-8,1) (-7,0) (-6.5, -1.25) (-5.5, -2) (-3.88, -1.34) (-4, 0.5) (-6, 3.25) (-7, 2.5)};
\draw [red,line width=1pt] plot [smooth] coordinates {(-8,1) (-7.26, 0.72) (-7,0) (-7.02, -0.8) (-6.5, -1.25) (-5.5, -2) (-4.07,-2.26)};
\draw [red,line width=1pt] plot [smooth] coordinates {(-3.31,-0.28) (-4, 0.5) (-5.1,1.86) (-6, 3.25) (-6.74,3.3) (-7, 2.5)};
\draw [red,line width=1pt] plot [smooth] coordinates {(-9.93,2.26) (-9.76, 1.26) (-10.69, 0.28)};
\draw[->,>=latex] (4,0)--(0.99,-2.16) node[yshift=-0.7cm,midway] {$\theta_1=\pi/2$};
\draw (-1,4.5)--(-1,-4.5)--(9,-4.5)--(9,4.5)--cycle;
\draw (-7,-4.3) node {North hemisphere};
\draw (7,-4) node {Tangent plane $\RR^2$};
\draw [line width=1pt] plot [smooth] coordinates {(2, 0.75) (4, 0) (4.75, -1) (5.75, -1.5) (7.16,-0.92) (7.25, 0.75) (5.78, 2.82) (4, 3.5) (3, 2.5)};
\draw [red,line width=1pt] plot [smooth] coordinates {(2, 0.75) (2.82, 0.24) (4, 0) (4.75, -1) (5.75, -1.5) (7.37,-1.52)};
\draw [red,dashed,line width=1pt] plot [smooth] coordinates {(7.37,-1.52) (8.1, -0.84) (7.7,-0.06)};
\draw [red,line width=1pt] plot [smooth] coordinates {(7.7,-0.06) (7.25, 0.75) (5.56, 2.6) (4, 3.5) (3.2, 3.26) (3, 2.5)};
\draw [line width=1pt](-10,-6)--(-9,-6) node[right] {Real curve};
\draw [red,line width=1pt](-10,-6.8)--(-9,-6.8) node[right] {Interpolated curve (disconnected)};
\draw [line width=1pt](-1,-6)--(0,-6) node[right] {Image of real curve from $\Log_{\bom_{0}}$ map};
\draw [red,line width=1pt](-1,-6.8)--(0,-6.8) node[right] {Interpolated curve};
\draw [red,line width=1pt,dashed](-1,-7.6)--(0,-7.6) node[right] {Loose of injectivity};
\draw [red,->,>=latex,bend left=15] (25.south) to (24.north);
\end{tikzpicture}
\end{center}
\caption{Loss of injectivity of the exponential map}
\label{fig:Interpolation_Manifold}
\end{figure}
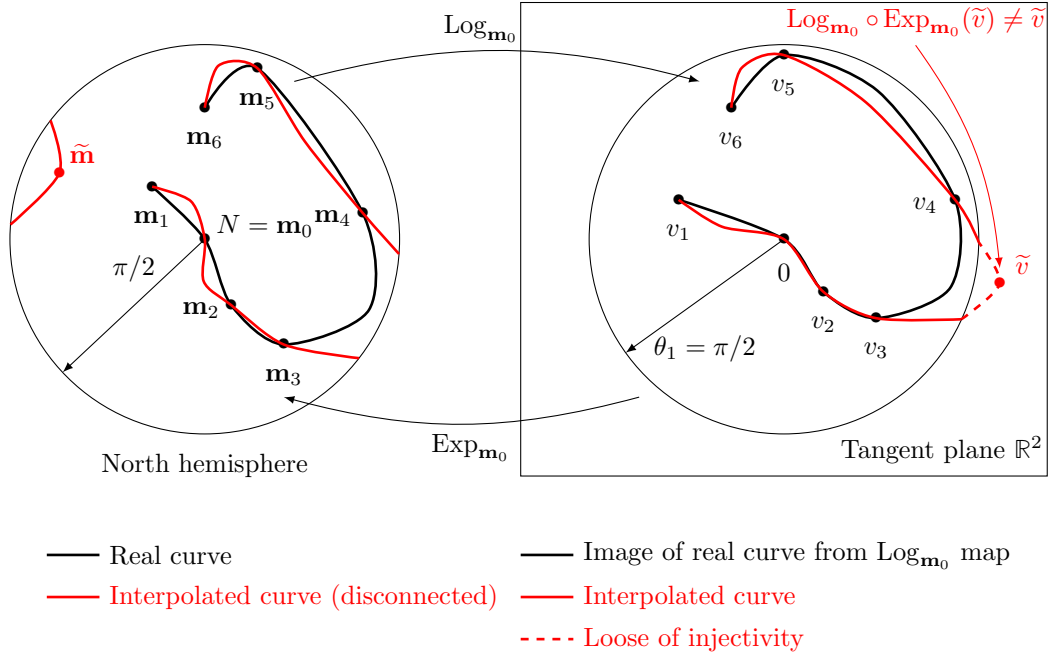

\subsection{Interpolation algorithm from Lagrange polynomials}\label{subsec:Interpol_Lagrangian}

Le us now produce the algorithm so to obtain an interpolated point $\widetilde{\bom}$ corresponding to a target parameter $\widetilde{\lambda}$. Such an algorithm is directly issued from the seminal work of Amsallem et al.\cite{Amsallem2009}, but it is modified so to obtain a well-defined interpolation, as we have to consider conditions (C1) and (C2) from the previous~\autoref{subsec:Assumption_Local_Chart}.

As detailed in~\autoref{sec:POD}, the POD of mode $p$ which was done on the snaphsot matrices $\bS_i$ (corresponding to the parameter $\lambda_{i}$ for $i=1,\dots,N$) define points $\bom_1,\dots,\bom_{N}$ on the Grassmann manifold $\mathcal{G}(p,n)$, and thus matrices in $\VecMat{n}{p}$ with orthonormal column vectors.

\begin{algo}[Interpolation on a Grassman manifold $\mathcal{G}(p,n)$]\label{alg:Interpol_Lagrang}\mbox{}\\
	\vspace*{-0.5cm}
	\begin{description}
		\item[\textbf{Input}]:
		\begin{itemize}
			\item Integers $p,n$ such that $2p\leq n$.
			\item Matrices $\bY_1,\dots,\bY_N$ in $\VecMat{n}{p}$ such that $\bY_{i}^{T}\bY_i=\bI_{p}$, respectively corresponding to parameters $\lambda_{1},\dotsc,\lambda_{N}$
			\item A target parameter $\widetilde{\lambda}$
		\end{itemize}
		\item[\textbf{Output}]: A new matrix $\widetilde{\bY}$ defining a new point $\widetilde{\bom}\in \mathcal{G}(p,n)$, corresponding to the target parameter $\widetilde{\lambda}$.
		\item[\textbf{Computations}]:
		\begin{enumerate}
			\item Choose a matrix $\bY_{0}\in \{\bY_1,\dots,\bY_N\}$ such that
			\begin{equation*}
			\text{(C1) stability}:\quad \bY_{0}^{T}\bY_i \text{ is non singular for all } i  
			\end{equation*}
			\item For each $i=1,\dots,N$, make a thin SVD and compute an $n\times p$ matrix $\bZ_i$:
			\begin{align*}
			\bY_{i}\left(\bY_{0}^{T}\bY_{i}\right)^{-1}-\bY_{0}&=\bU_{i}\bSig_{i} \bV_{i}^{T} \\
			\bZ_{i}:=\bU_{i}\arctan\left(\bSig_{i}\right) \bV_{i}^{T},
			\end{align*}
			all issued from the logarithm map (Definition~\ref{def:Log_map_Grass}).
			\item Compute an interpolated matrix and a thin SVD
			\begin{equation*}
			\widetilde{\bZ}:=\sum_{i=1}^{N} \prod_{i \neq j} \frac{\tilde{\lambda}-\lambda_j}{\lambda_i-\lambda_j}\bZ_i=\widetilde{\bU}\widetilde{\bThe}\widetilde{\bV}
			\end{equation*}
			\item (C2) stability: If $\widetilde{\theta}_1>\pi/2$, with $\widetilde{\theta}_1$ the largest singular value of $\widetilde{\bZ}$, then return an \emph{instability message}.
			\item Otherwise return the $n\times p$ matrix
			\begin{equation*}
			\widetilde{\bY}:=\bY_0\widetilde{\bV}\cos\widetilde{\bThe}+\widetilde{\bU}\sin\widetilde{\bThe}
			\end{equation*}
			issued from the exponential map~\ref{eq:Def_Exp_Map}.
		\end{enumerate}
	\end{description}
\end{algo}

\subsection{Instability problem due to increasing mode}\label{subec:Geometric_Distances}

As one should expect, the accuracy of the interpolation algorithm~\ref{alg:Interpol_Lagrang} should improve as the number $p$ of mode increase. In fact, when considering the snapshot matrices $\bS_1,\dots,\bS_N$ associated to the parameters $\lambda_{1},\dots,\lambda_{N}$, a POD of mode $p$ define subspaces $\mathcal{V}_1,\dots,\mathcal{V}_N$ of dimension $p$ (see~\autoref{sec:POD}). By construction, for another mode $p'>p$, the corresponding subspaces  $\mathcal{V}'_1,\dots,\mathcal{V}'_N$ are such that
\begin{equation*}
\mathcal{V}_i\subset \mathcal{V}_i'.
\end{equation*}
Take now a new parameter $\widetilde{\lambda}$ and suppose that algorithm~\ref{alg:Interpol_Lagrang} returns matrices $\widetilde{\bY}$ and $\widetilde{\bY}'$ which correspond respectively to mode $p$ and $p'>p$ interpolation. A stability condition should be
\begin{itemize}
	\item[$\bullet$] (C3) The subspaces $\widetilde{\mathcal{V}}$ and $\widetilde{\mathcal{V}}'$ respectively associated to the matrices $\widetilde{\bY}$ and $\widetilde{\bY}'$ are such that $\widetilde{\mathcal{V}}\subset \widetilde{\mathcal{V}}'$.
\end{itemize}
More generally, let us consider two subspaces $\mathcal{V}$ and $\mathcal{V}'$ of different dimensions $p<p'$, represented by matrices $\bY\in \VecMat{n}{p}$ and $\bY'\in \VecMat{n}{p'}$ such that
\begin{equation*}
\bY^{T}\bY=\mathbf{I}_p,\quad (\bY')^{T}\bY'=\mathbf{I}_{p'}.
\end{equation*}
One method to measure the non-inclusion defect between subspaces $\mathcal{V}$ and $\mathcal{V}'$ is to consider the \emph{geometric distance} $\delta(\mathcal{V},\mathcal{V}')$, issued from~\cite{Ye2016}, and defined using \emph{principal angles} as follows: taking singular values of $\bY^{T}\bY'\in \VecMat{p}{p'}$ to be $\sigma_{1}\geq \dots \geq \sigma_p\geq 0$, we have
\begin{equation}\label{eq:Geom_Distance}
\delta(\mathcal{V},\mathcal{V}')=\delta(\bY,\bY'):=\bigg(\sum_{i=1}^{\text{min}(p,p')} \arccos^{2}(\sigma_{i}) \bigg)^{1/2}.
\end{equation}

We are finally able to check stability condition (C3) using the following:
\begin{enumerate}
	\item  Assume  a set of POD modes $p \in \mathscr{P}_m$ and a threshold value $T_V$.
	\item For a given integer $p$ and a given target parameter $\widetilde{\lambda}$, compute matrix $\widetilde{\bY}$ issued from Algorithm~\ref{alg:Interpol_Lagrang}.
	\item For $p'>p$ compute matrix $\widetilde{\bY}'$ issued from the same algorithm Algorithm~\ref{alg:Interpol_Lagrang}.
	\item As we have $\widetilde{\bY}^{T}\widetilde{\bY}=\mathbf{I}_p$ and $(\widetilde{\bY}')^{T}\widetilde{\bY}'=\mathbf{I}_{p'}$ from Lemma~\ref{lem:Geod_Inside_Stiefel_Compact}, we deduce a geometric distance $\delta(\widetilde{\bY},\widetilde{\bY}')$ computed by~\eqref{eq:Geom_Distance}.
	\item Calculate 
	\begin{equation}\label{eq:Threshold_Stab_C3}
		\epsilon = (\delta_{\text{max}}(\widetilde{\bY},\widetilde{\bY}') - \delta_{\text{min}}(\widetilde{\bY},\widetilde{\bY}'))/(\delta_{\text{min}}(\widetilde{\bY},\widetilde{\bY}')),\quad p \in \mathscr{P}_m
	\end{equation}	
	\item If $\epsilon \geq T_V$ then return an instability message.
\end{enumerate}

Let us now describe the utilization of the (C3) stability condition from the application point of view. By computing the geometric distance $\delta(\widetilde{\bY},\widetilde{\bY}')$ we are able to explain the non-monotonic oscillatory behavior of the error norm due to increasing mode $p$. According to both situation under study, the first benchmark problem (see Figure~\ref{fig:Distance_subspaces_different_dimensions_pbl_1}) shows a clear oscillatory behavior, while the second one seems stable (see Figure~\ref{fig:Distance_subspaces_different_dimensions_pbl_2}): we thus compared the two $\epsilon$ values given by~\eqref{eq:Threshold_Stab_C3}, for each benchmark problem, and propose $T_V=100$ as a reference threshold.

\section{Application to Hyperelasticity}\label{sec:App_HyperElast}
\subsection{Kinematics of Continuum Mechanics Framework}

Let $\Omega_{0} \subset R^3$ and $\Omega \subset R^3$ represent the reference and the current configurations of a body, parameterized  in $\mathbf{X}$ and in $\mathbf{x}$, respectively. The non-linear deformation map $\varphi: \Omega_{0} \rightarrow \Omega$ at time $t$, transforms the referential (material) position $\mathbf{X}$ into the related current (spacial) position $\mathbf{x}=\varphi(\mathbf{X},t)$. The deformation gradient $\mathbf{F}$ is defined by

\begin{equation}
\mathbf{F} := \nabla \varphi(\mathbf{X}) = \frac{\partial \varphi( \mathbf{X})}{\partial \mathbf{X}} = \frac{\partial \mathbf{x}}{\partial \mathbf{X}}
\end{equation}  
with the Jacobian $J(\mathbf{X}) = \det(\mathbf{F}) > 0$ (volume ratio).  The right and left Cauchy-Green tensors are defined as $\mathbf{C} = \mathbf{F}^T \mathbf{F}$ and $\mathbf{B} = \mathbf{F} \mathbf{F}^T$, respectively. 

The three principal invariants of $\mathbf{C}$ which are identical to those of $\mathbf{B}$ are defined as
\begin{equation}
I_1 = \tr(\mathbf{C}), \quad I_2 = \frac{1}{2} [(\tr(\mathbf{C}))^2 - \tr \left(\mathbf{C^2}\right)], \quad I_3 = \text{det} (\mathbf{C}).
\label{eq:Invariants_1-3}
\end{equation}  

\subsection{Incompressible Transverse Isotropic Material}

A material with one family of fibers is considered where the stress at a material point depends not only on the deformation gradient $\mathbf{F}$ but also on the fiber direction. The fibers are modeled by a \emph{flow}~\cite{Gallot1990} obtained from some unit vector field  $\mathbf{a}_0$ on $\Omega_{0}$. The direction of a fiber at point $\mathbf{X} \in \Omega_{0}$ is thus obtained by the unit vector $\mathbf{a}_0 (\mathbf{X}), \, | \mathbf{a}_0 |=1$. 

Note that the unit vector field $\mathbf{a}_0$ induces a unit vector field $\mathbf{a}$ on current configuration $\Omega$ defined by 
\begin{equation*}
\mathbf{F}(\mathbf{X})\mathbf{a}_0(\mathbf{X})=\alpha \mathbf{a}(\mathbf{x})
\end{equation*} 
where the length changes of the fibers along its direction $\mathbf{a}_0$ is determined by the stretch $\alpha$ as the ratio between the current and the reference configuration.

Consequently, since $| \mathbf{a} |=1$, we can define the square of the stretch $\alpha$ following the symmetries of the deformation gradient
\begin{equation*}
\alpha^2 = \mathbf{a}_0\mathbf{F}^T \mathbf{F} \mathbf{a}_0 = \mathbf{a}_0\mathbf{C} \mathbf{a}_0. 
\end{equation*}

\subsection{Linearization of the principle of internal virtual work in the spatial description}

The linearization of the internal virtual work in the spatial description reads (see Section 8.4 in~\cite{Holzapfel2002})
\begin{equation}
D_{\Delta \mathbf{u}} \delta W_{int}(\mathbf{u}, \delta \mathbf{u}) =
\int_{\Omega}^{}  (\text{grad} \delta \mathbf{u} : \mathbb{c} : \text{grad} \Delta \mathbf{u} + 
\text{grad} \delta \mathbf{u} : \text{grad} \Delta \mathbf{u} \: \boldsymbol{\sigma}) dv
\label{eq:Linear_Int_work_spatial}
\end{equation} 
or in index notation (with Einstein convention on repeated indices),
\begin{equation}
D_{\Delta \mathbf{u}} \delta W_{int}(\mathbf{u}, \delta \mathbf{u}) =
\int_{\Omega}^{} \frac{\partial \delta u_a}{\partial x_b} 
(\delta_{ac} \sigma_{bd} + \mathbb{c}_{abcd}) \frac{\partial \Delta u_c}{\partial x_d} dv
\label{eq:Linear_Int_work_spatial_index}
\end{equation} 
where the term $\delta_{ac} \sigma_{bd} + \mathbb{c}_{abcd}$ is the effective elasticity tensor in the spatial description. The term $\delta_{ac} \sigma_{bd}$ corresponds to the geometrical stress contribution to  linearization (initial stress contribution at every increment) whereas $\mathbb{c}_{abcd}$ represents the material contribution to linearization. The elasticity tensor $\mathbb{c}_{abcd}$ in the spatial description is derived from the push-forward of the linearized second Piola-Kirchhoff stress tensor which yields the linearized Kirchhoff stress tensor $ \Delta \boldsymbol{\tau}$ from relation

\begin{equation}
\Delta \boldsymbol{\tau}= 
J  \mathbb{c} : \text{grad} \Delta \mathbf{u}
\label{eq:PF_Linear_S_index_2}
\end{equation}

Replacing the direction $\Delta \mathbf{u}$ of the directional derivative with the velocity vector $\mathbf{v}$, $\Delta \boldsymbol{\tau}$ and $\text{grad} \Delta  \mathbf{u}$ result in the Lie time derivative $\mathcal{L}_{v}(\boldsymbol{\tau})$ of $\boldsymbol{\tau}$ and the spatial velocity gradient $\mathbf{l}= \mathbf{\dot{F}} \mathbf{F}^{-1}$, respectively. Again, using the minor symmetries of $\mathbb{c}$, the following relation can be written

\begin{equation}
\mathcal{L}_{v}(\boldsymbol{\tau}) = 
\text{Oldr}(\boldsymbol{\tau}) = 
\boldsymbol{\dot{\tau}}- 
\mathbf{l} \boldsymbol{\tau} -
\boldsymbol{\tau} \mathbf{l}^T =
J \mathbb{c} : \mathbf{d}
\label{eq:Oldroyd_stress_rate}
\end{equation}
where $\text{Oldr}(\boldsymbol{\tau})$ denotes the objective Oldroyd stress rate (convected rate) of the contravariant Kirchhoff stress tensor $\boldsymbol{\tau}$ and $\mathbf{d}=\text{sym}(\mathbf{l})$ (symmetric part of $\mathbf{l}$) the rate of the deformation tensor. At this point we have to recall that for structural elements (shells, membranes, beams, trusses) Abaqus/Standard uses the elasticity tensor related to the Green-Naghdi objective rate. The detailed constitutive model used here is given in~\cite{Holzapfel2007}. 

\section{Numerical Investigations}
\label{sec:Numerical_Investigations}

The objective of this section is to assess the stability and accuracy of POD basis interpolation on Grassmann manifolds using two examples of hyperelastic structures. 

\subsection{Abaqus implementation of POD-ROM approximations}
\label{sec:multi_point_constraint_equations}

To implement a ROM for FEM analysis, a non-intrusive approach is utilized to insert the interpolated spatial POD basis into a commercial code. Specifically, a ROM is constructed using the multi-point constraint equations in Abaqus~\cite{ABAQUS2014}. A linear multi-point constraint requires that a linear combination of nodal variables is equal to zero:
	
\begin{equation}
A_1 u^{P}_{i} + A_2 u^{Q}_{j} + \dots + A_N u^{R}_{k} = 0
\label{eq:Linear constraint equations}
\end{equation} 
	where $u^{P}_{i}$ is the nodal variable at node $P$, degree of freedom $i$ and	$A_i, (i=1,\dots N)$ are coefficients that define the relative motion of the nodes. In Abaqus/Standard the first nodal variable specified ($u^{P}_{i}$ corresponding to $A_1$) will be eliminated to impose the constraint. In addition, the coefficient $A_1$ should not be set to zero. For the construction of a ROM, $p$ reference points  are created corresponding to the total number of POD modes (arbitrary positioned in space). These reference points are used to define the constraint equations for introducing the spatial POD modes and to assign the extra degrees of freedom corresponding to the unknown `time' variables. Thus, the interpolated spatial basis
	$\boldsymbol{\tilde{\Phi}}_p:=[\tilde{\phi}_{1},\dots,\tilde{\phi}_{p}]\in \VecMat{n}{p}$ representing the subspace 
	 $\tilde{\bom}:=\spn(\tilde{\phi}_{1},\dotsc,\tilde{\phi}_{p})$ on $\mathcal{G}(p,n)$ is imposed to the linear constraint equations as follows:
\begin{equation}
u(x_l,t,\tilde{\lambda}) - \sum_{h=1}^{p} \tilde{\phi}_h (x_l) \psi_h(t) = 0
\label{eq:Linear constraint equations_POD}
\end{equation} 
where $x_l,(l=1,\dots,N_s)$ is related to the nodal point positions, $\tilde{\phi}_h(x_l)$ represent the associated spatial POD $h$-mode for $x_l$, and  $\psi_h(t)$ is the `time' variable assigned to the reference point $h$ that has to be determined. Note also that the system of equations defined in~\eqref{eq:Linear constraint equations_POD} has to be generated for each degree of freedom. 

\rem{In fact this is not a standard POD-Galerkin approach since we are not projecting the linearized system of equations onto the interpolated spatial POD basis. But it serves us to assess the stability and accuracy of the ROM FEM model which is constructed by the interpolated POD basis. We mention that generating an efficient ROM model is not the
objective of this work.}

\subsection{Inflation of a spherical balloon}\label{subsec:Spherical_Balloon}

The first pMOR benchmark example concerns the inflation of a spherical balloon considering the material anisotropy defined by the fiber orientation angle as a parameter. The sphere has an initial radius of $R=10$, thickness $h=0.5$ and is loaded by an internal hydrostatic pressure of $P=40$ (no units). The FEM analysis is performed on an octant $\mathrm{S}_0$ of the sphere using plane symmetry boundary conditions, as depicted in Figure~\ref{fig:Geo_BC}, where three radial points $A(R,0,0)$, $B(0,R,0)$ and $C(0,0,R)$ are defined on each axis, respectively. Three-node shell elements (S3R) are used for the mesh~\cite{ABAQUS2014}. A total number of 514 elements are generated with 228 nodes. The hyperelastic constitutive behavior is implemented in Abaqus/Standard with a user-defined subroutine (UMAT)~\cite{ABAQUS2014}. 

\begin{figure}[http]
\centering
\includegraphics[width=.4\linewidth]{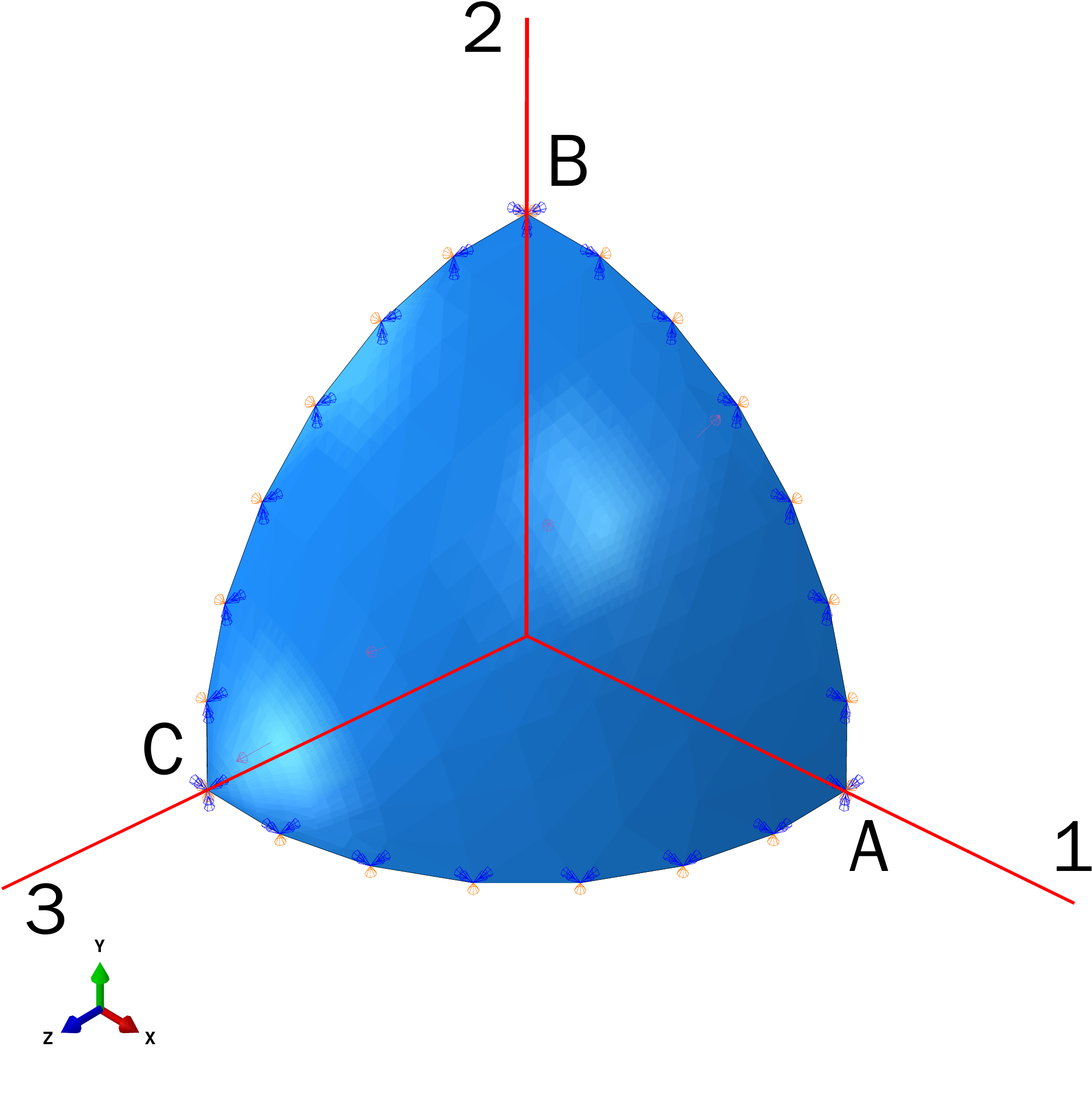}  
\caption{Geometry of an octant $\mathrm{S_0}$ of a spherical balloon made of transversely isotropic hyperelastic material. Three radial points A, B and C are defined on axis 1,2 and 3, respectively; plane symmetry boundary conditions are used.}
\label{fig:Geo_BC}
\end{figure}

\begin{rem}\label{rem:Local_Basis_Octant}
    The fiber orientation has to be defined on each point $M\in \mathrm{S}_0$ using an orthonormal basis of the tangent plane $T_M\mathrm{S}_0$, which has to be specified. 
    
    The choice made in Abaqus is to consider first an outward normal $\mathbf{n}(M)$ to this tangent plane and then a first vector $\mathbf{E}_{1}(M)$ as the orthogonal projection (normalized) of $\mathbf{e}_1:=(1,0,0)$ onto $T_M\mathrm{S}_0$. The second unit vector is the cross product $\mathbf{E}_{2}(M):=\mathbf{n}(M)\wedge \mathbf{E}_{1}(M)$.
\end{rem}

\subsubsection*{Explicit fiber orientations on the octant}
Let make now an explicit definition of the fiber orientations, with parameter some angle $\theta$ using local basis $\mathbf{E}_1(M),\mathbf{E}_2(M)$ of the tangent plane $T_M\mathrm{S}_0$ as explained in Remark~\ref{rem:Local_Basis_Octant}. More specifically, take
\begin{equation*}
    M=(\cos(u)\sin(v),\sin(u)\sin(v),\cos(v))\in \mathrm{S}_0,\quad (u,v)\in \left]0;\frac{\pi}{2}\right[\times \left]0;\frac{\pi}{2}\right[
\end{equation*}
and then define
\begin{align*}
    \mathbf{E}_1(M)&:=\frac{\mathbf{X}^h}{\|\mathbf{X}^h\| },\quad 
    \mathbf{X}^h:=\begin{pmatrix}
1-\cos^2(u)\sin^2(v) \\
-\sin(u)\cos(u)\sin^2(v) \\
-\cos(u)\sin(v)\cos(v)
\end{pmatrix}, \\
\mathbf{E}_2(M)&:=\mathbf{n}(M)\wedge \mathbf{E}_{1}(M).\quad  
\end{align*}
Note here that the vector $\mathbf{X}^h$ corresponds to the orthogonal projection of the vector $(1,0,0)$ onto the tangent plane $T_M\mathrm{S}_0$.

Finally, the unit vector defining the fiber orientation is given by (see Figure~\ref{fig:Fibers_Sphere_Abaqus} for some examples).
\begin{equation*}
    \mathbf{a}_0(\theta):=\cos(\theta)\mathbf{E}_1(M)+\sin(\theta)\mathbf{E}_{2}(M)
\end{equation*}

\begin{figure}[H]
\begin{subfigure}{.5\textwidth}
  \centering
  \includegraphics[scale=0.4]{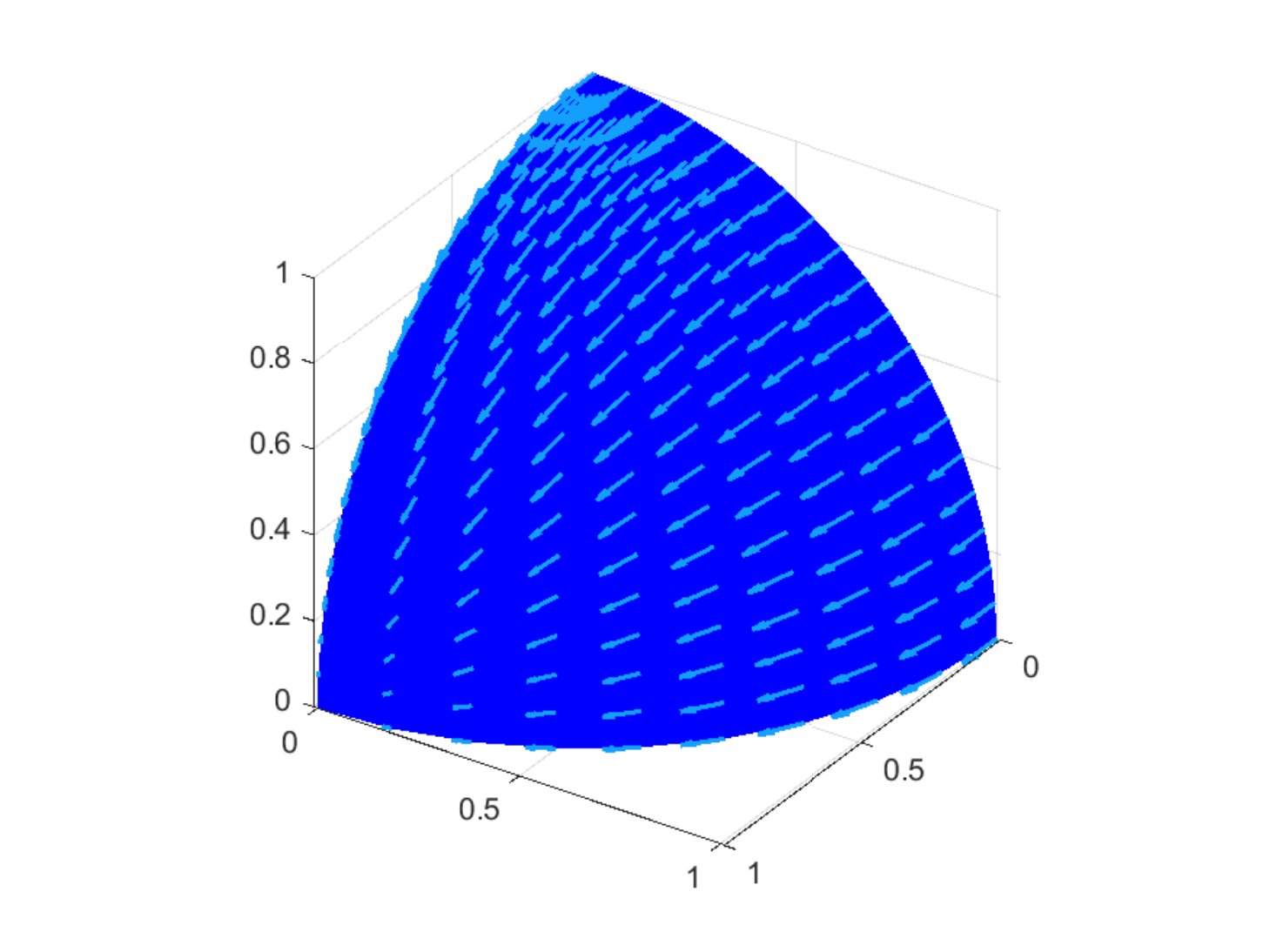}  
  \caption{Fibers orientation with $\theta=0$ degree}
  \label{fig:Fiber_0}
\end{subfigure}
\begin{subfigure}{.5\textwidth}
  \centering
  \includegraphics[scale=0.4]{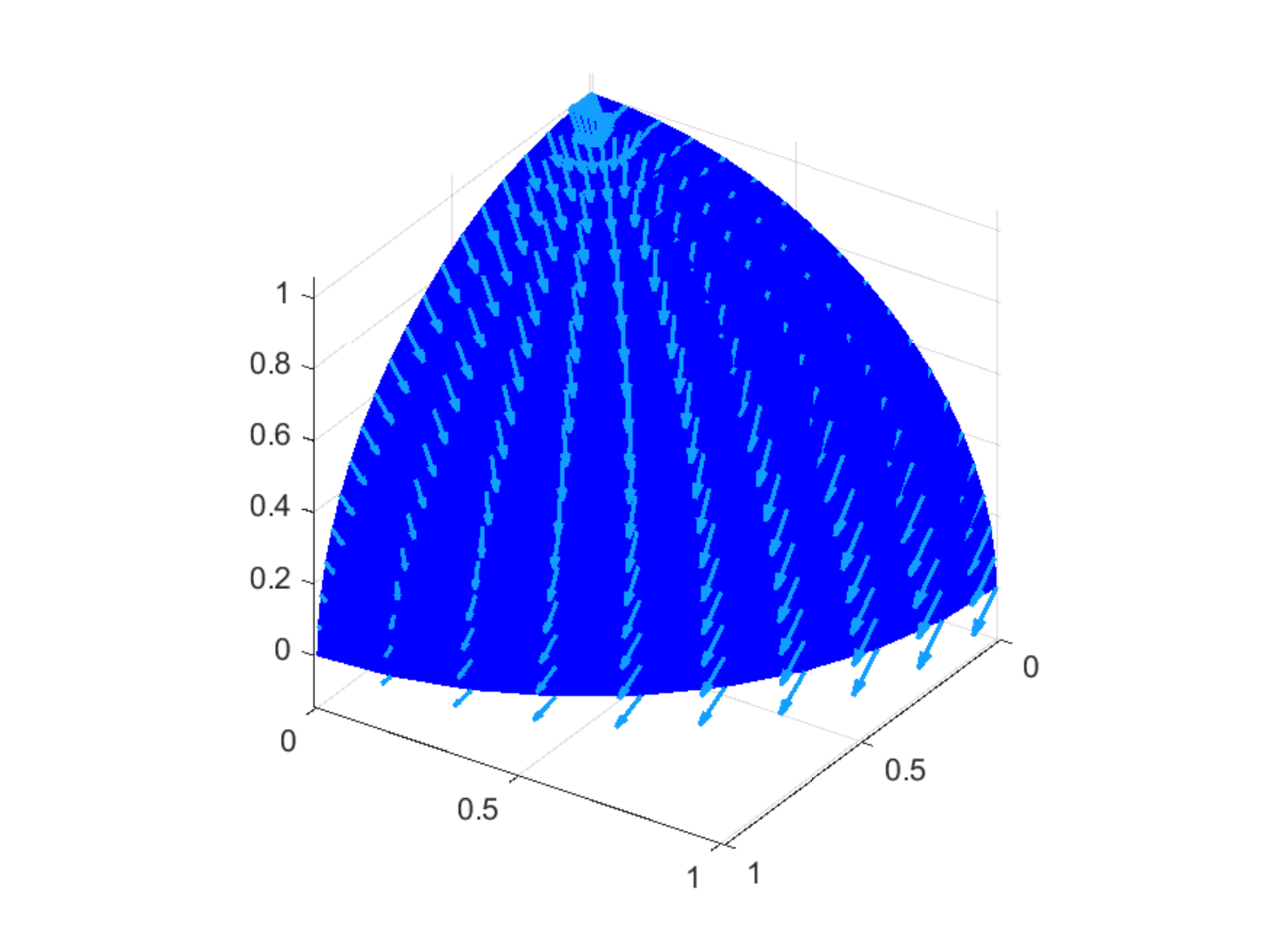}  
  \caption{Fibers orientation with $\theta=45$ degree}
  \label{fig:Fiber_45}
\end{subfigure}
\begin{subfigure}{.5\textwidth}
  \centering
  \includegraphics[scale=0.4]{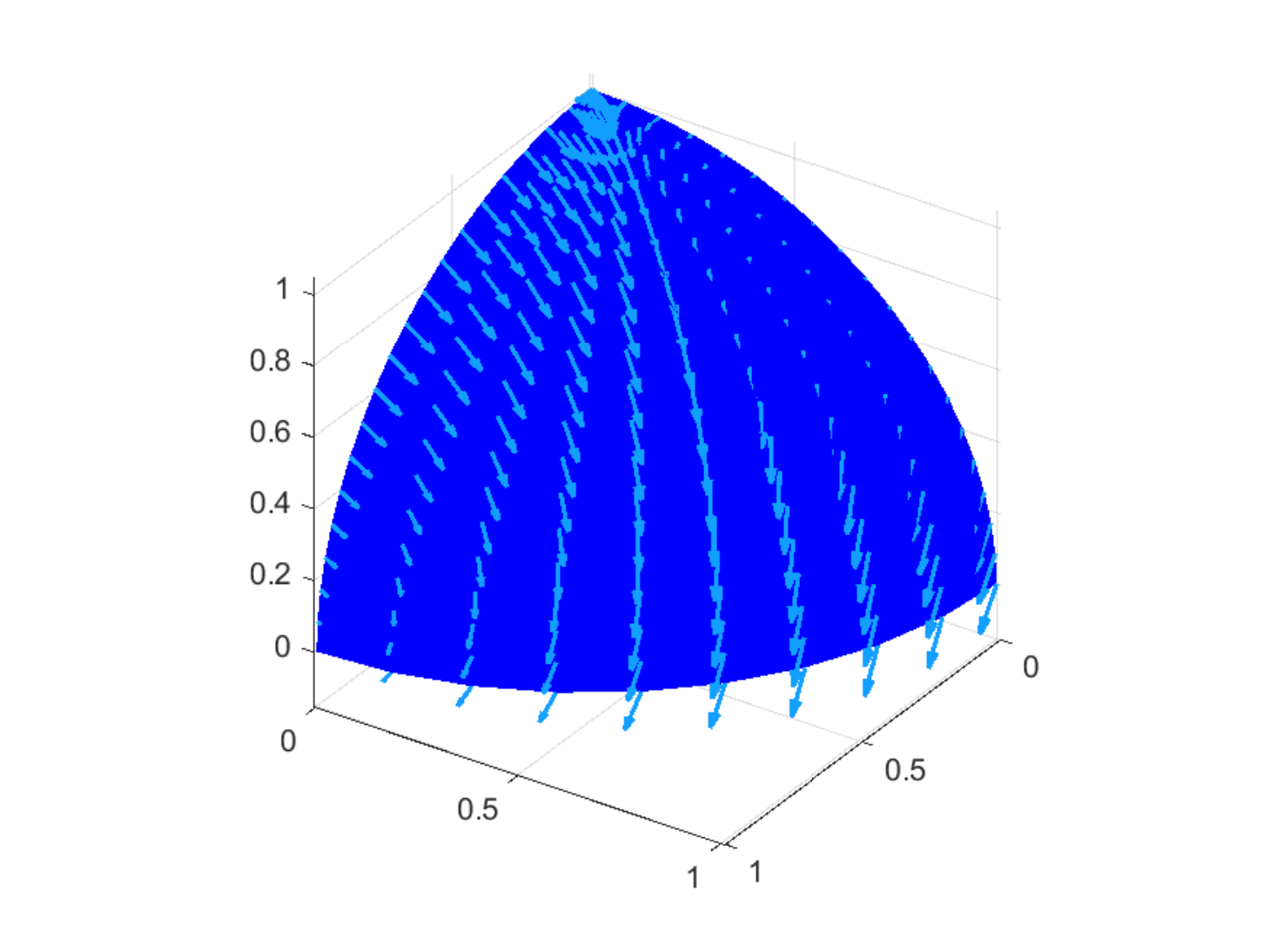}  
  \caption{Fibers orientation with $\theta=60$ degree}
  \label{fig:Fiber_60}
\end{subfigure}
\begin{subfigure}{.5\textwidth}
  \centering
  \includegraphics[scale=0.4]{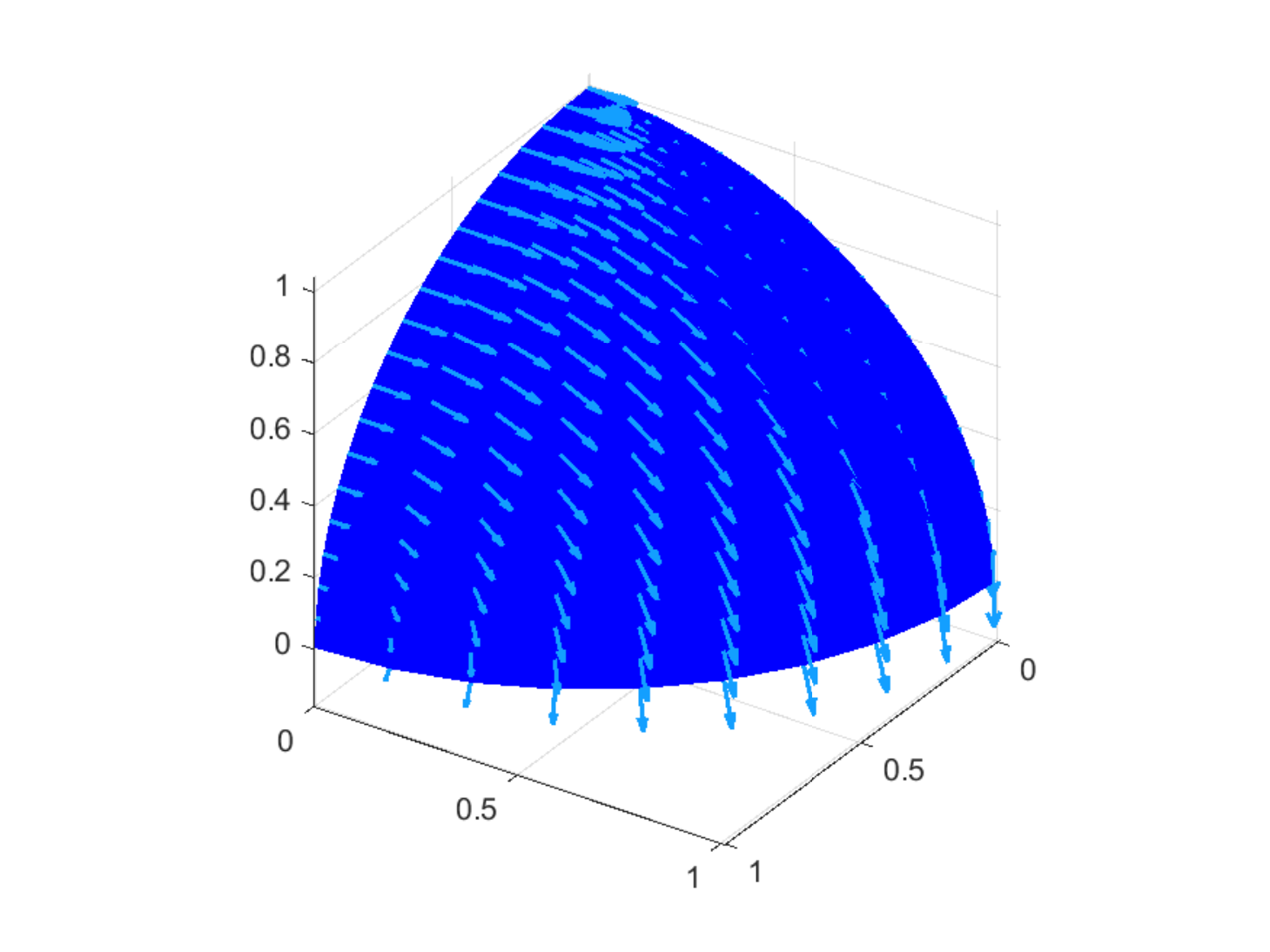}  
  \caption{Fibers orientation with $\theta=90$ degree}
  \label{fig:Fiber_90}
\end{subfigure}
\caption{Different fibers on the sphere.}
\label{fig:Fibers_Sphere_Abaqus}
\end{figure}

\subsubsection*{Model for strain energy function}
For a homogeneous transversely isotropic non-linear material, let consider a free energy function  that depends only on two invariants ($I_1,I_4$)   
\begin{equation*}
\Psi = \Psi\left(I_1(\mathbf{C}), I_4(\mathbf{C},\mathbf{a}_0)\right)
\end{equation*}
where $I_1=\tr(\mathbf{C})$, while
\begin{equation}
I_4 = \mathbf{a}_0\mathbf{C} \mathbf{a}_0,  
\label{eq:Invariants_4}
\end{equation}
is the invariant related to anisotropy. Since we assume incompressibility of the isotropic matrix material, i.e., $I_3=1$, the free energy is enhanced by an indeterminate Lagrange multiplier $p$ which is identified as a reaction pressure 
\begin{equation*}
\Psi = \Psi[I_1(\mathbf{C}),  I_4(\mathbf{C},\mathbf{a}_0)] + p (I_3-1).
\end{equation*}

The specific model used here is developed for membranous or thin shell-like sheets considering a plane stress state throughout the sheet~\cite{Holzapfel2007}. Following the method of Humphrey~\cite{Humphrey1990} which is based on a derivation by Spencer~\cite{Spencer1972}, the strain energy function is defined as
\begin{equation}
\Psi(I_1,I_4) := c_0(\text{exp}(Q)-1),\quad Q := c_1(I_1-3)^2 + c_2(I_4-1)^2
\label{eq:Strain_energy_MV}
\end{equation}
where $c_i,i=0,1,2$ are material parameters defined as: $c_0 = 86.1$, $c_1 = 0.0059$ and $c_2 = 0.031$ (dimensionless). 

\begin{rem}
This model introduces an inherent constitutive coupling between the isotropic and anisotropic material response. In order to avoid non-physical behavior of soft biological tissues, the related strain-energy function must be polyconvex. It can be shown that polyconvexity of a (continuous) strain-energy function implies that the corresponding acoustic tensor is elliptic for all deformations, which means from the physical point of view that only real wave speeds occur; then the material is said to be stable. There exists a vast literature on polyconvexity, a term introduced by Ball~\cite{Ball1976}. In~\eqref{eq:Strain_energy_MV}, the anisotropic term $ c_2(I_4-1)^2$ is activated only when $I_4 \geq 1$ (the actual fiber stretches are greater than unity).

Moreover, as discussed in~\cite{May1998}, the constitutive description based on~\eqref{eq:Strain_energy_MV} is limited to deformations in which the in-plane strains are positive, or tensile and is not able to incorporate the behavior of the structure in compression. Due to the membrane-like geometry of the structure, it is unlikely to support compressive strains without buckling. This limitation extends to the issue of bending stiffness, which is neglected in this model. 
\end{rem}

\subsubsection*{Snapshot matrices and error norms}

In what follows, the training points  corresponding to the fiber orientation angle $\theta$ will be noted with parameter $\lambda$ for convenience with the previous sections. FEM simulations are performed in Abaqus/Standard for the points $\lambda_i \in \Lambda_s = \{0,45,50,60,85,90\}$.
Note that the spherical balloon changes from a pumpkin (Figure~\ref{fig:Ballon_inflation_modes}(a)) to rugby shaped  (Figure~\ref{fig:Ballon_inflation_modes}(d)) for  $\lambda=0$ and $\lambda=90$, respectively. Observe in Figure~\ref{fig:Fibers_Sphere_Abaqus} that the fiber orientation on the sphere is far from being trivial for $\theta \in ]0;90[$.

The target point for interpolation is set to $\widetilde{\lambda}=75$. Thus, it is natural to constraint the training set to $\Lambda_t =\{50,60,85,90\}$  (see Figure~\ref{fig:Ballon_inflation_modes} for some FEM results). We note that the target point $\tilde{\lambda} = 75$ represents a worst case scenario for assessing the interpolation accuracy since it is spaced nearly at the maximum distance between the adjacent training points $\lambda = 60$ and $\lambda = 85$. Additionally, another reason for this choice is the remarkable shape transition of the spherical balloon inflation in this range of fibration angles as can been seen from Figure~\ref{fig:Ballon_inflation_modes}(b) and Figure~\ref{fig:Ballon_inflation_modes}(c), respectively.  
Hence, this selection gives an upper bound of the interpolation accuracy over the considered parametric range.

For each parametric simulation, a sequence of uniform time snapshots is extracted from the model database. From the discretization of the space-time fields (displacement/rotation),  the snapshot matrices $\bS(\lambda_i)$ of size $(n=1728) \times (N_t=1000)$ are formed. The eigenvalue spectrum of the matrices $\bS(\lambda_i)$ corresponding to training points $\lambda_i \in \Lambda_t$ is shown in a log-log scale in Figure~\ref{fig:Sing_values_magnitude_infl_ballon}. The condition number of the matrices is of the order of $1.0e+10$. Notice that the distance between the first and the second eigenvalue is of two orders of magnitude. 

\begin{figure}[http]
\begin{subfigure}[t]{0.49\textwidth}
  \centering
  \includegraphics[width=\linewidth]{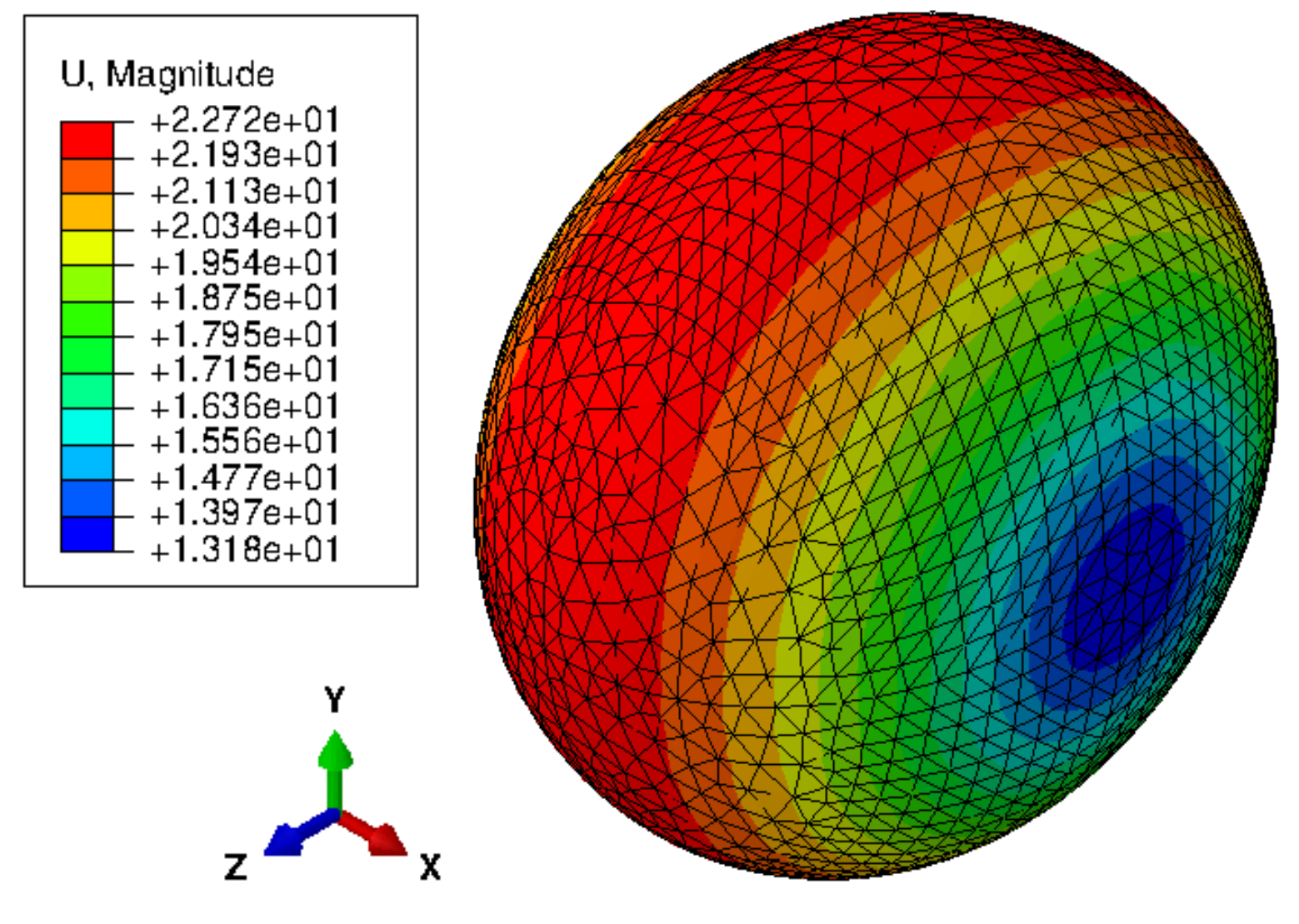}
    \caption{For $\theta = 0^{\circ}$}
  \label{fig:Inflation_0}
\end{subfigure}
\begin{subfigure}[t]{0.49\textwidth}
  \centering
 \includegraphics[width=\linewidth]{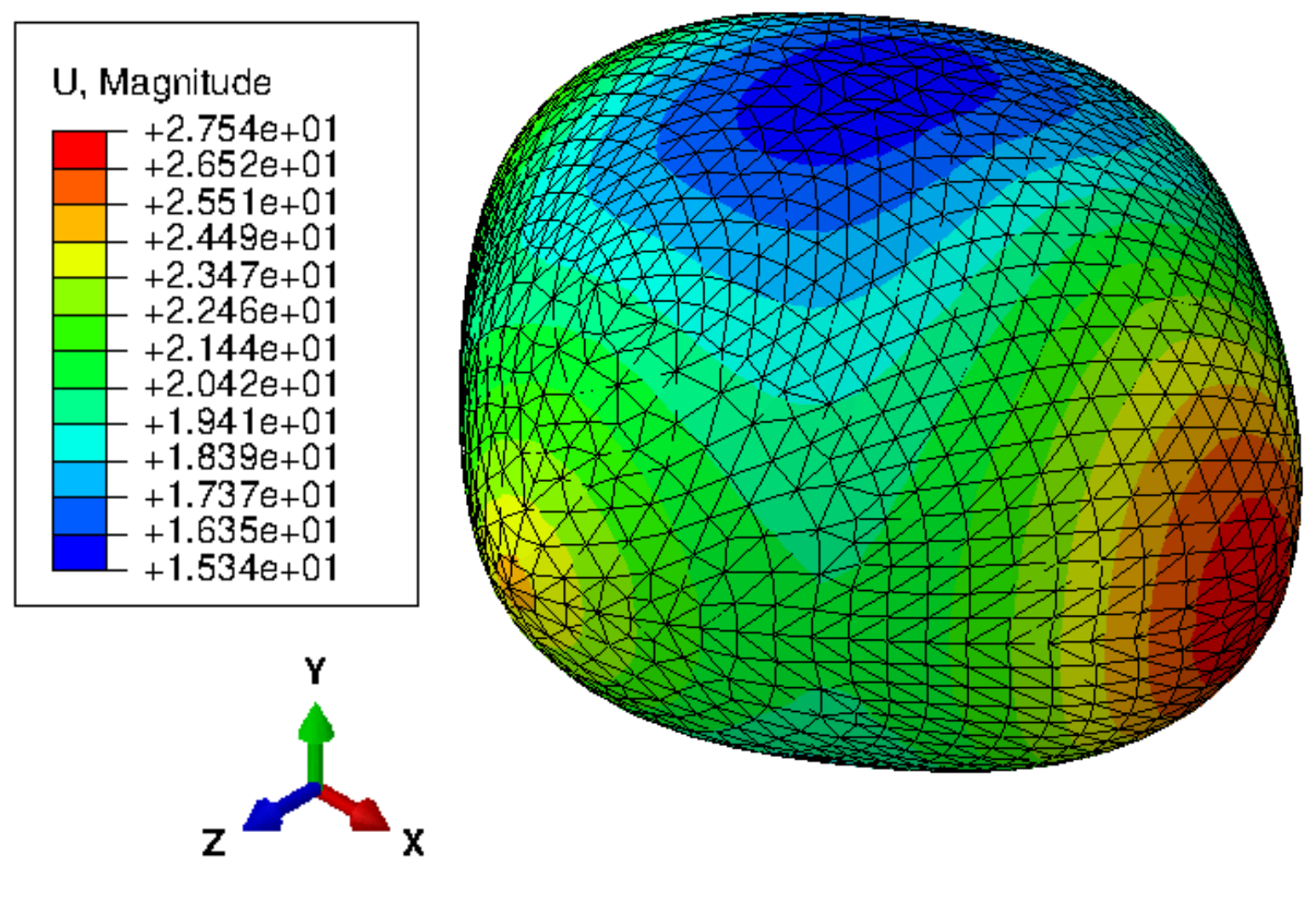}
    \caption{For $\theta = 60^{\circ}$}
  \label{fig:Inflation_60}
\end{subfigure} \\
\begin{subfigure}[t]{0.49\textwidth}
  \centering
 \includegraphics[width=\linewidth]{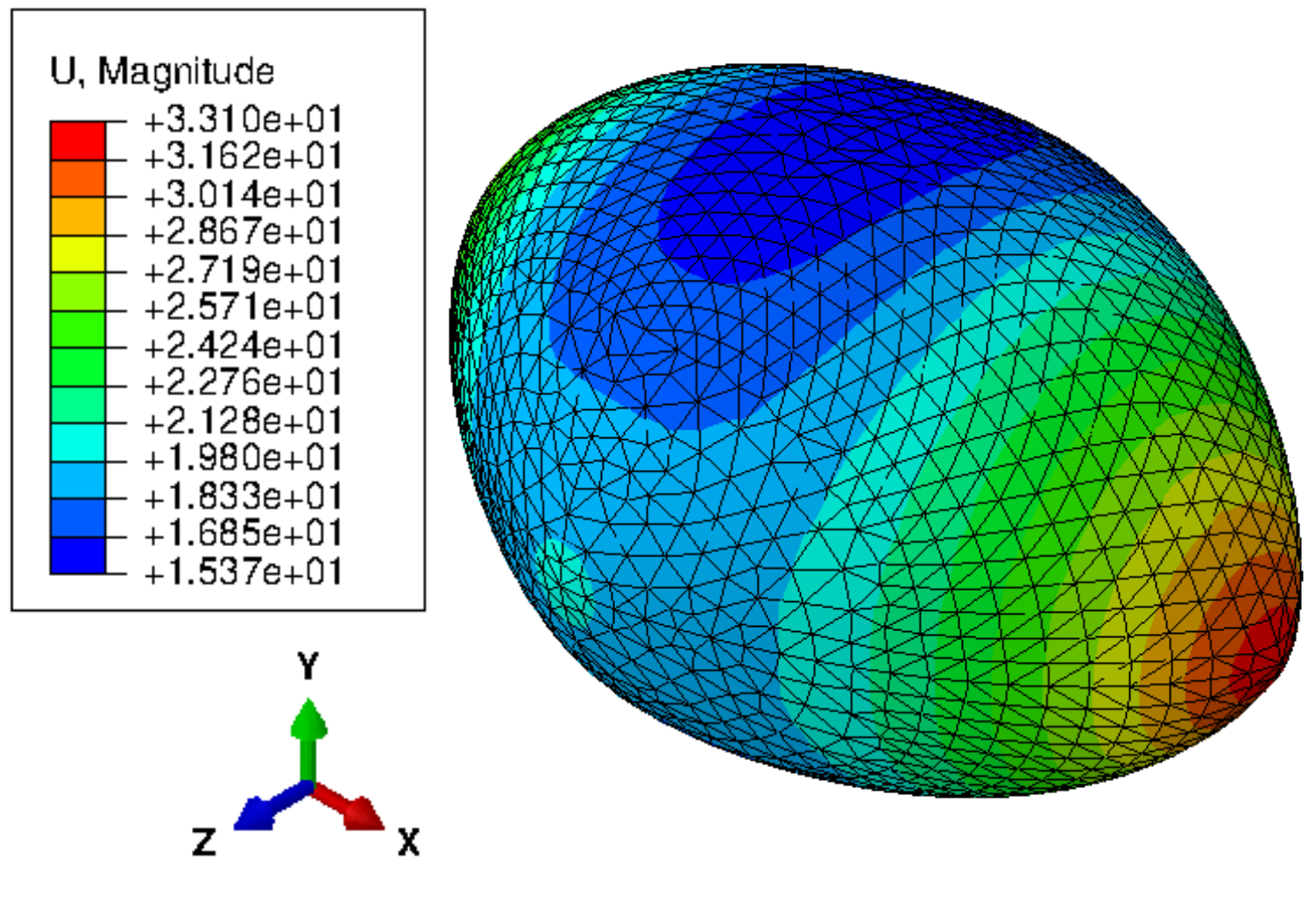}
    \caption{For $\theta = 75^{\circ}$}
  \label{fig:Inflation_75}
\end{subfigure} 
\begin{subfigure}[t]{0.49\textwidth}
  \centering
 \includegraphics[width=\linewidth]{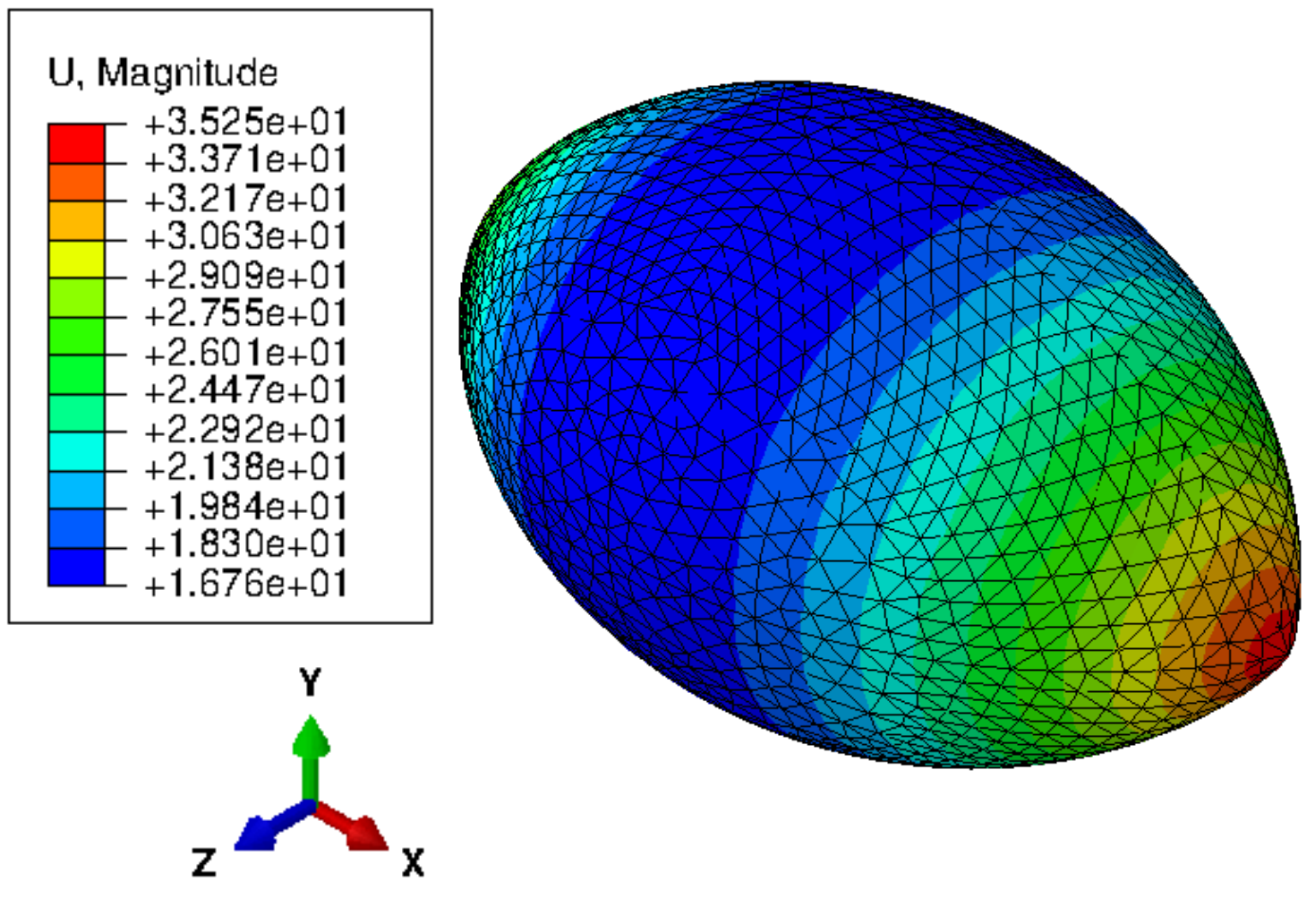}
    \caption{For $\theta = 90^{\circ}$}
  \label{fig:Inflation_90}
\end{subfigure}
\caption{Inflation modes of the benchmark anisotropic spherical balloon after reconstruction of the complete balloon using the plane symmetries conditions at the boundaries of the octant $\mathrm{S_0}$.}
\label{fig:Ballon_inflation_modes}
\end{figure}

\begin{figure}[http]
	\centering
	\includegraphics[width=0.6\columnwidth]{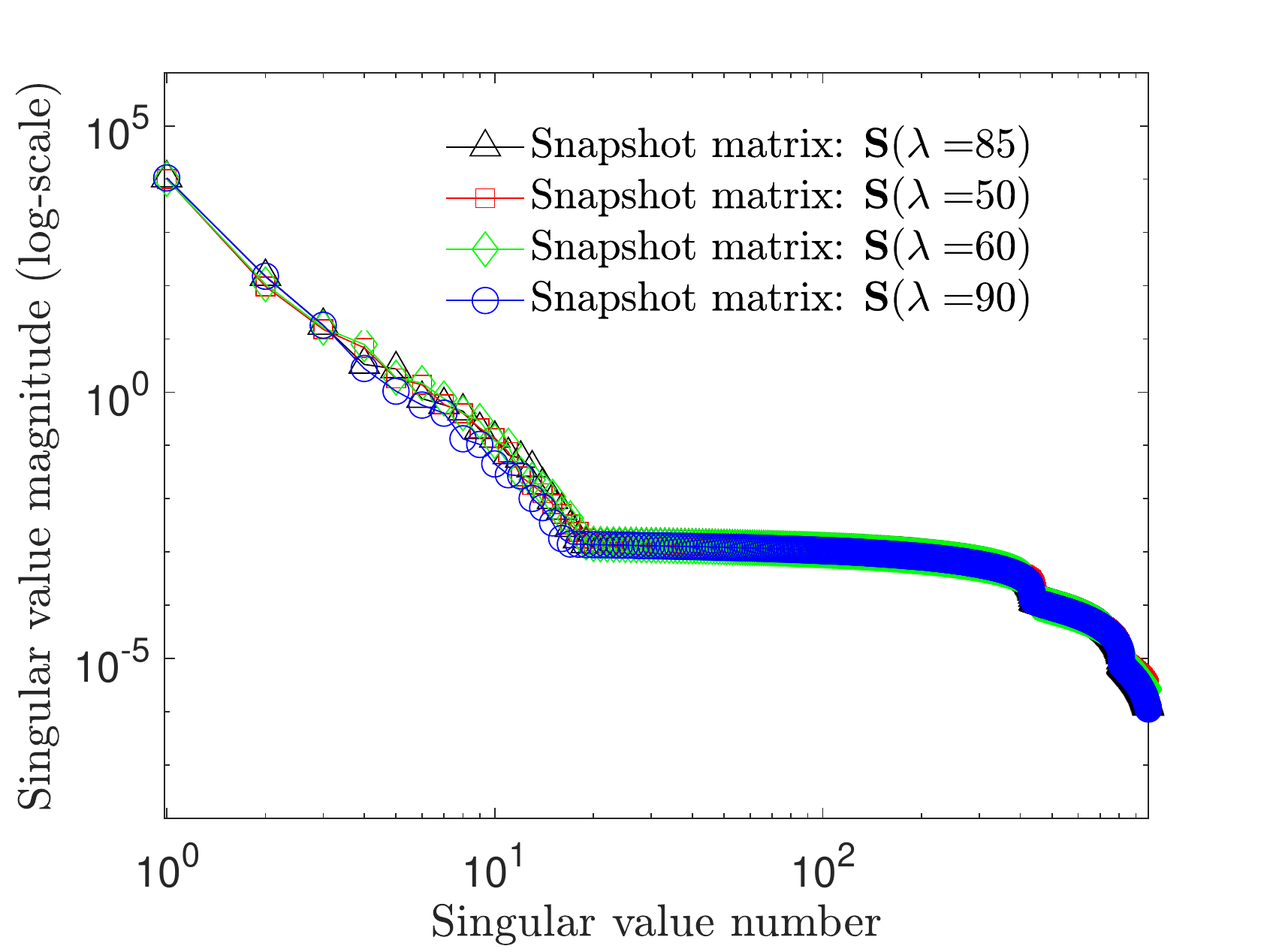}
	\caption{The eigenvalue spectrum of snapshot matrices $\bS_i$ corresponding to training points $\lambda_i=50,60,85,90$.}
	\label{fig:Sing_values_magnitude_infl_ballon}
\end{figure}

To quantify the accuracy of the interpolation, the relative $L_2$-error norm (in time) for a given target point $\tilde{\lambda}$ is evaluated with respect to the high-fidelity FEM solution. Using the interpolated and the HF-FEM snapshot matrices $\tilde{\bS}$ and $\bS^{\text{FEM}}$, respectively, the following error measure is defined at each time snapshot 

\begin{equation}
e_{L_2}(\tilde{\bS})= 
\frac{\Vert \mathbf{\tilde{u}}_i - \mathbf{u}^{\text{FEM}}_i \Vert_{L_2}}{\Vert  \mathbf{u}^{\text{FEM}}_i) \Vert_{L_2}}, \quad
i=1,\dots,N_t
\label{eq:L2_error_norm_HF}
\end{equation} 

In addition, the relative Frobenius error norm represents a global error measure which considers the error in the full time interval of the time steps

\begin{equation}
e_{F}(\tilde{\bS})= \Vert \mathbf{\tilde{S}} - \mathbf{S}^{\text{FEM}} \Vert_{F} / \Vert \mathbf{S}^{\text{FEM}} \Vert_{F}
\label{eq:Frobenious_error_norm}
\end{equation}

Using the linear constraint equations defined in~\eqref{eq:Linear constraint equations_POD}, $p$ reference points (for each POD mode) are created to assign the spatial POD basis representing the interpolated subspace $\tilde{\bom} \in \mathcal{G}(p,n)$ and the unknown time variables. Thus, the total number of equations of the ROM-FEM model is $6 \times p$ while the total number of equations of the corresponding HF-FEM model is $288 \times 6 = 1728$.

\subsubsection*{Stability conditions (C1) and (C2)}

First we need to know if the interpolation is (C1) and (C2) stable.

\
\textbf{Stability (C1).}
All points $\bom_1,\dotsc,\bom_N \in \mathcal{G}(p,n)$ lie in $\mathrm{U}_{\bom_0}$, given by~\eqref{eq:Def_Open_Set_Log}. We need to check that for all $i=1,\dots,N$, the matrix $\bY_0^{T}\bY_i$ is non singular. (C1) condition is satisfied for all $i=1,\dots,N$ and $p = 1,2,5,10,20$ POD modes considered.  

\

\textbf{Stability (C2).}
We need to know if all velocity vectors $\widetilde{v}(\lambda)$ belong to the subset $\mathrm{V}_{\bom_0}$ given by~\eqref{eq:Def_Vm_Angle}, for the parametric range $\lambda \in [\lambda_{1},\lambda_{N} ]$. Thus, we have to check that the first (maximum) singular value $\theta_1$ of a horizontal lift $\tilde{\mathbf{Z}}(\lambda)$ of the velocity vector $\widetilde{v}(\lambda)$ is such that $\theta_{1} < \pi/2$, for all $\lambda \in [\lambda_{1},\lambda_{N} ]$. We proceed by uniformly sampling 401 points over the parametric range $[50;90]$.
Figure~\ref{fig:Stability_condition_1_pbm_1} shows the maximum eigenvalue $\theta_1$ of the horizontal lift $\tilde{\mathbf{Z}}(\lambda)$ for all samples using $\bom_0(\lambda=85)$ as a reference point on the Grassmann manifold.
These curves provide all important information for the (C2) stability of interpolation by detecting the exact intervals of the loss of injectivity of the Exponential mapping for various POD modes $p=1,2,5,10,20$. Observe the loss of injectivity in a specific interval of the parameter range for  modes $p=10,20$. A remarkable result is the loss of injectivity inside the parameter range and not at the boundaries where the Exponential map becomes again injective. Note also that by increasing the dimension $p$, the curves shift more rapidly closer to $\pi/2$.  Figure~\ref{fig:Stability_condition_1_pbm_1} reveals that interpolation is (C2) stable for the target point $\tilde{\lambda} = 75$ for all POD modes $p$.         

\begin{figure}[H]
	\centering
	\includegraphics[width=0.7\columnwidth]{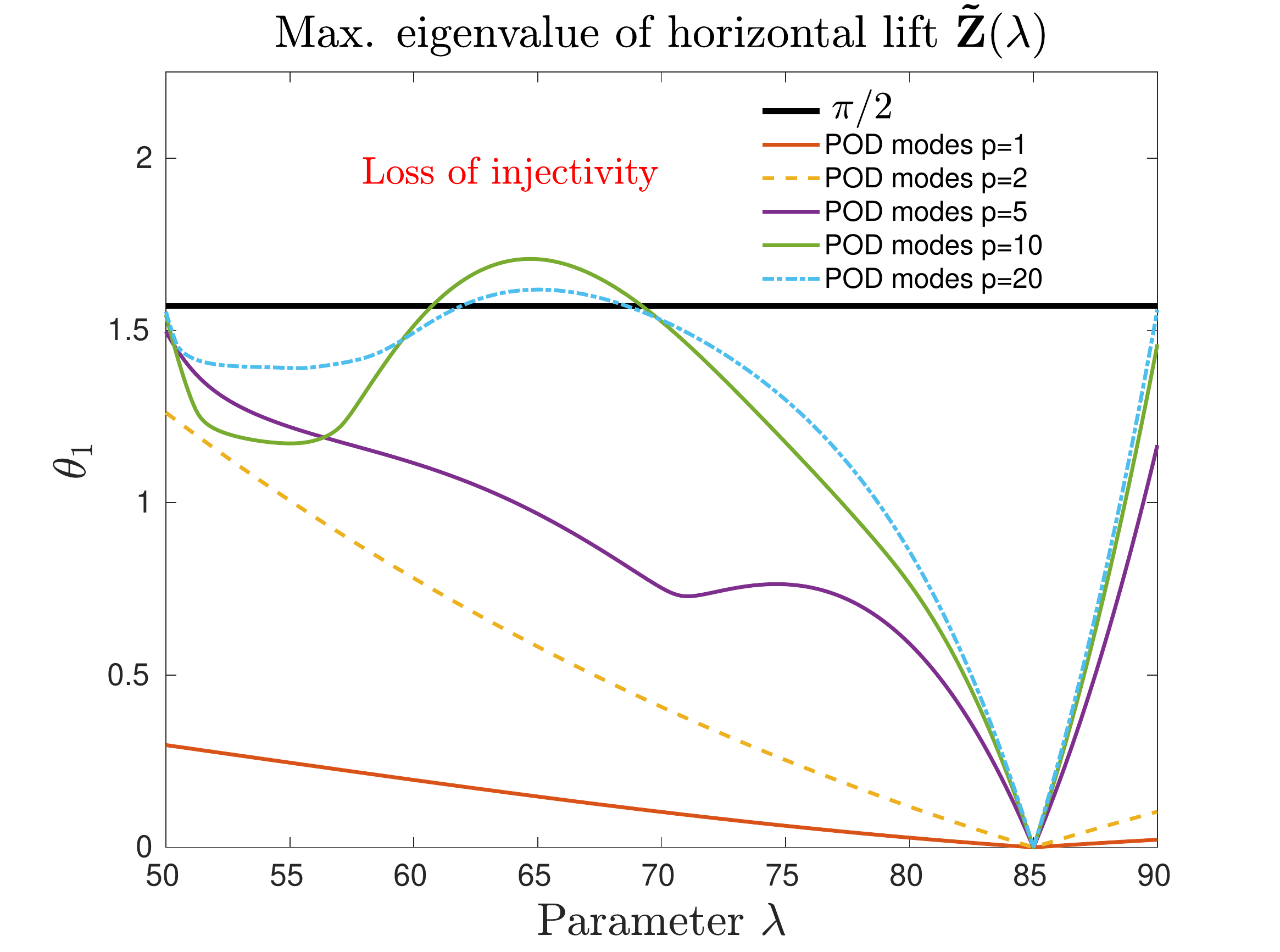}
	\caption{Stability (C2); Computation of the maximum eigenvalue $\theta_1$ of the horizontal lift $\tilde{\mathbf{Z}}(\lambda)$ over the parametric range $\lambda \in [50;90]$. Observe the loss of injectivity in a specific interval in the parametric range for POD modes $p=10,20$. Reference point on Grassmann manifold $\bom_0(\lambda=85)$.}
	\label{fig:Stability_condition_1_pbm_1}
\end{figure}

\subsubsection*{Accuracy and Stability condition (C3)}

Figure~\ref{fig:Standard_POD_vector_number_L2_error_norm} and Figure~\ref{fig:Standard_POD_vector_number_Frobenious_error_norm} show the relative $L_2$-error norm $e_{L_2}(\tilde{\bS})$ and the Frobenius error norm $e_{F}(\tilde{\bS})$ for the target point $\tilde{\lambda}$ of the ROM-FEM  solution constructed from the interpolated $p$-dimensional spatial modes. Additionally, Table~\ref{table:Grassmannian_dim_pbl_1} shows the Grassmannian dimension for the different number of POD modes.

\begin{figure}[H]
	\centering
	\includegraphics[width=0.6\columnwidth]{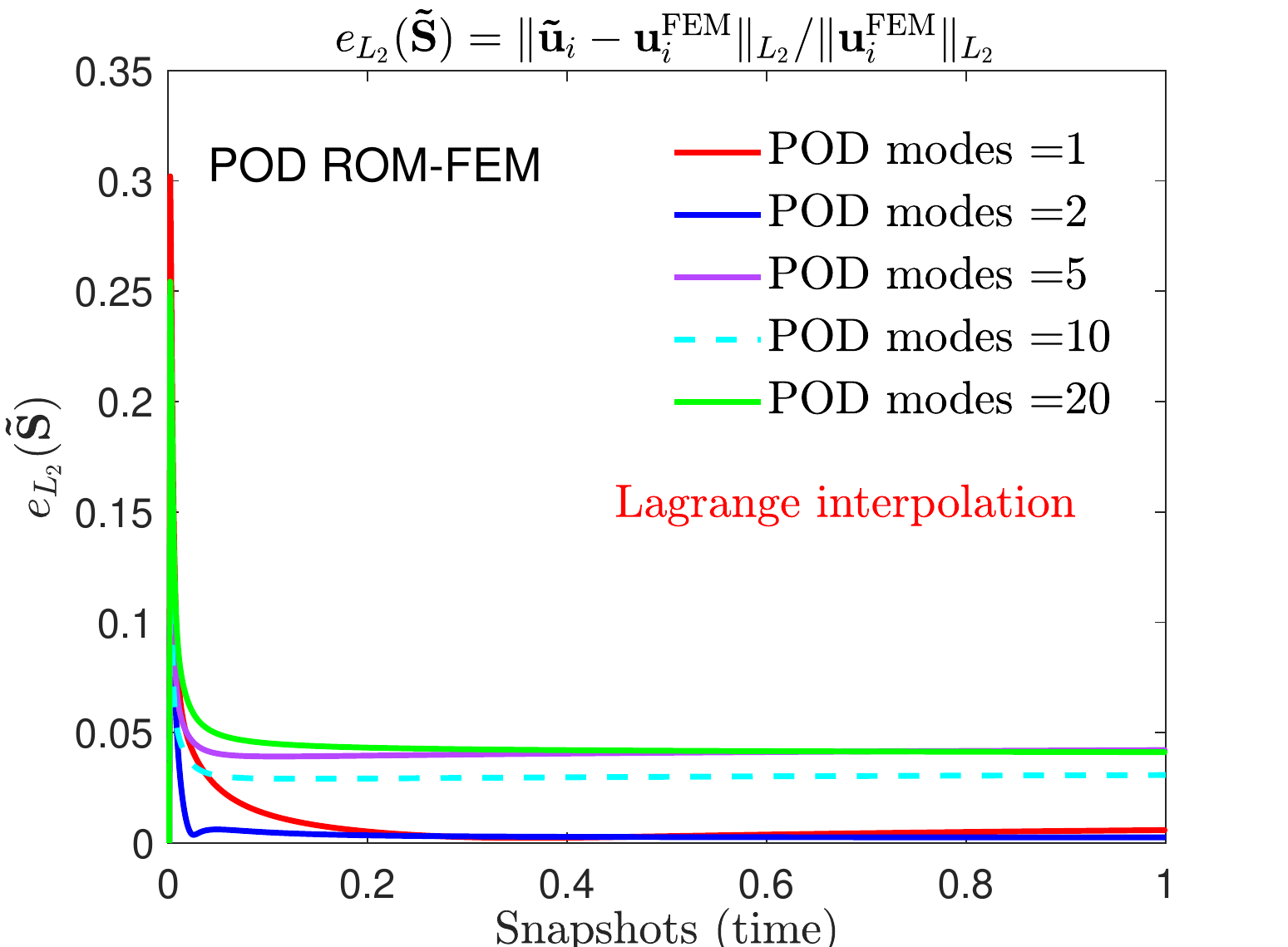}
	\caption{Relative $L_2$-error norm $e_{L_2}(\tilde{\bS})$ against the number of POD vectors for the POD ROM-FEM; target point: $ \tilde{\bom}(\lambda=75)$.}
	\label{fig:Standard_POD_vector_number_L2_error_norm}
\end{figure}

\begin{figure}[H]
	\centering
	\includegraphics[width=0.65\columnwidth]{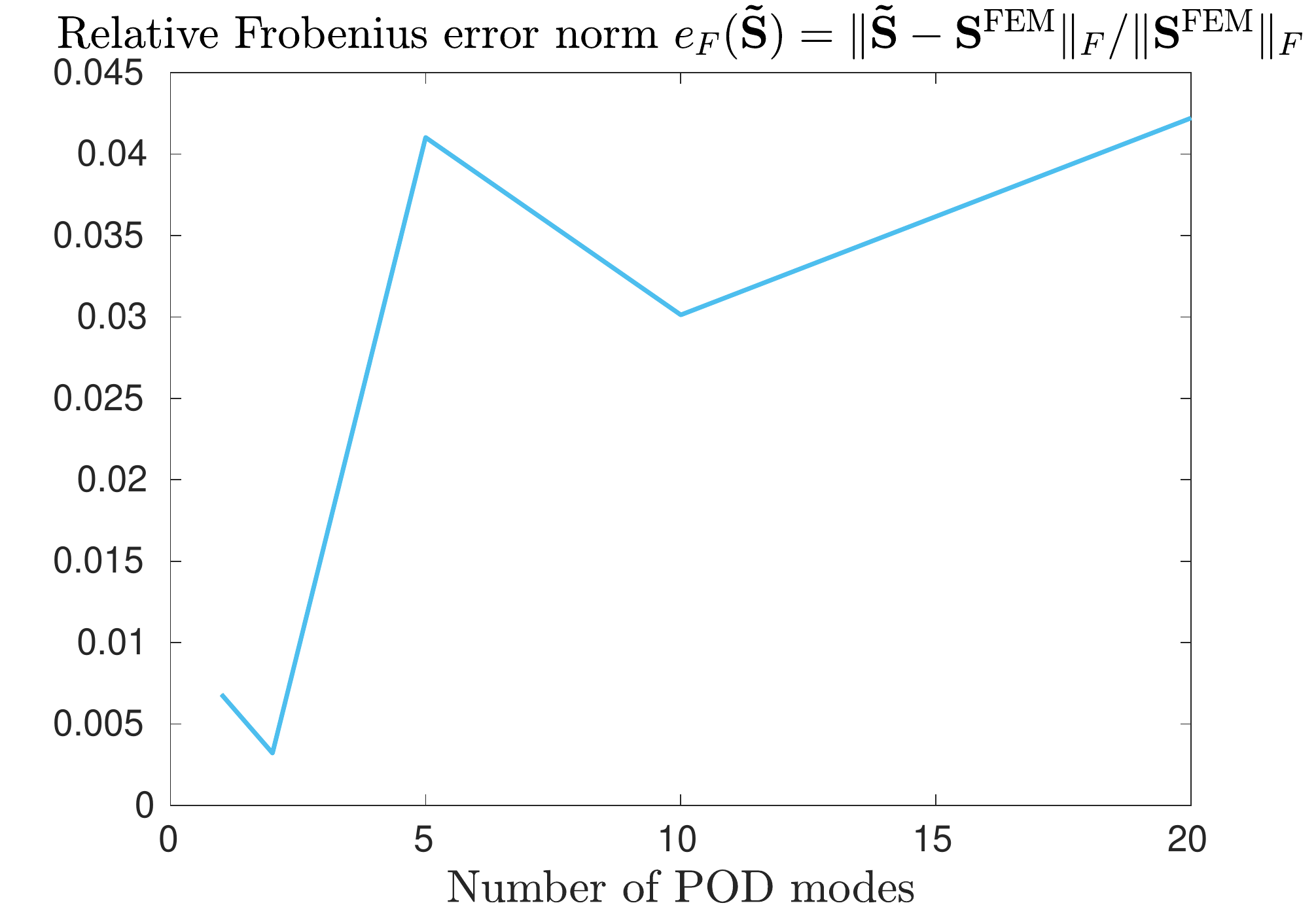}
	\caption{Relative Frobenius error norm against the number of POD vectors for the POD ROM-FEM; target point: $ \tilde{\bom}(\lambda=75)$.}
	\label{fig:Standard_POD_vector_number_Frobenious_error_norm}
\end{figure}

\begin{table}[ht]
\caption{Dimension of the Grassmann manifold $\mathcal{G}(p,n)$}
\centering
\begin{tabular}{ c c c c c c}
\hline 
\textbf{Number of modes} & \textbf{$p=1$} & \textbf{$p=2$} & \textbf{$p=5$} & \textbf{$p=10$} & \textbf{$p=20$}      \\ \hline
\textbf{Dimension: $p(n-p)$ }   &   1727  &  3452 &  8615 & 17180 & 34160 \\ \hline
\end{tabular}
\label{table:Grassmannian_dim_pbl_1}
\end{table}

\textbf{Stability (C3).}
We need to check if the interpolated subspaces $\widetilde{\mathcal{V}}$ and $\widetilde{\mathcal{V}}'$ respectively associated to the matrices $\widetilde{\bY}$ and $\widetilde{\bY}'$ correspond  to mode $p$ and $p'>p$ interpolation are such that $\widetilde{\mathcal{V}}\subset \widetilde{\mathcal{V}}'$.
Before examine if the interpolation is (C3) stable, observe from Figure~\ref{fig:Standard_POD_vector_number_L2_error_norm} and Figure~\ref{fig:Standard_POD_vector_number_Frobenious_error_norm} of the relative error norms \eqref{eq:L2_error_norm_HF} and \eqref{eq:Frobenious_error_norm}, respectively, that the error is minimum for $p=2$ POD modes and increases by introducing additional modes which at first glance contradicts the `expected' improvement of the solution by increasing the number of modes. In this case, the non-monotonous error decrease and the random oscillations follows from the non-inclusion defect between subspaces $\mathcal{V}$ and $\mathcal{V}'$ obtained by using different POD modes. To prove that fact, we compute the non-inclusion defect considering the geometric distance $\delta(\mathcal{V},\mathcal{V}')$ using the principal angles defined in \eqref{eq:Geom_Distance}.
We assume a set of POD modes $p \in \mathscr{P}_m = \{1,2,5,10,20\}$ and a threshold value $T_V =100$.
Figure~\ref{fig:Distance_subspaces_different_dimensions_pbl_1} lists the distances of the obtained POD basis of various dimensions $p \in \mathscr{P}_m$ in a symmetric table form. Observe that i) $\delta(\widetilde{\bY},\widetilde{\bY}') \neq 0$ for all $p \neq p'$ and ii) $\delta(\widetilde{\bY},\widetilde{\bY}')$ increase rapidly for $p>2$. Thus, this table explains why the relative error norms (Figure~\ref{fig:Standard_POD_vector_number_L2_error_norm} and Figure~\ref{fig:Standard_POD_vector_number_Frobenious_error_norm}) have a minimum at $p=2$ modes. 
Since the relative error $\epsilon$ given by~\eqref{eq:Threshold_Stab_C3} is here $\epsilon=554.03 > T_V $, we can conclude that the interpolation is not (C3) stable.
The results make clear and prove the non-inclusion defect of different subspaces which in turn give rise to the oscillatory behavior of the error norms as described above.

\begin{figure}[H]
	\centering
	\includegraphics[width=0.6\columnwidth]{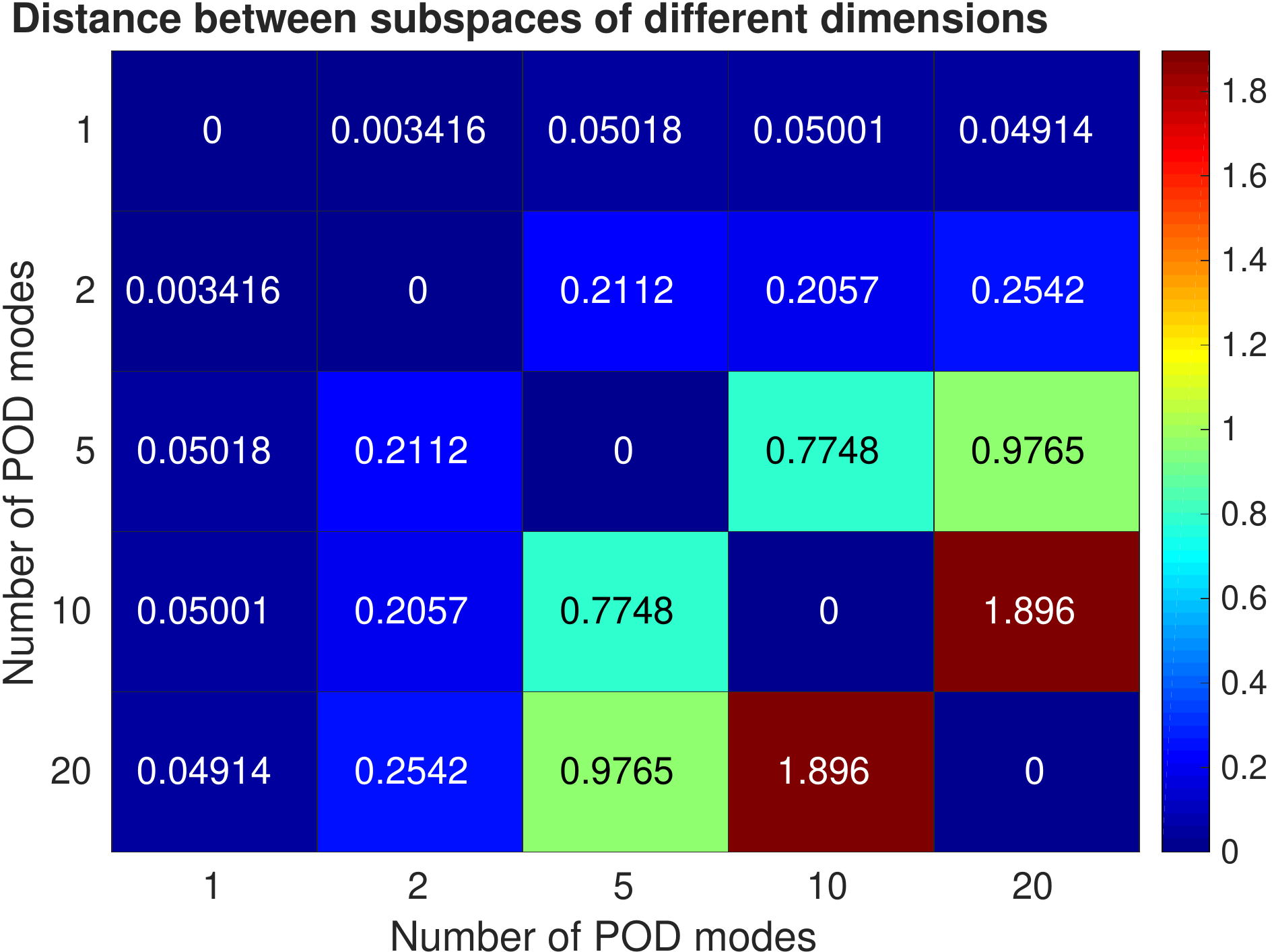}
	\caption{Stability (C3); Geometric distance $\delta(\bY,\bY')$ between interpolated subspaces of different dimensions.}
	\label{fig:Distance_subspaces_different_dimensions_pbl_1}
\end{figure}

Moreover, the  interpolation accuracy  is assessed using the relative displacement error $e_{\mathbf{u}} =\Vert \mathbf{\tilde{u}}(t)-\mathbf{u}^{FEM}(t) \Vert_{L_2} / \Vert \mathbf{u}^{FEM}(t) \Vert_{L_2}$ at the nodal points computed for $p=$1,2,5 and 10 POD modes. Figure~\ref{fig:Snapshot_L2_norm_balloon_modes_state_12} and Figure~\ref{fig:Snapshot_L2_norm_balloon_modes_state_1} present the local error at the increment state $t=0.002$ and at the final increment state  $t=1$ displayed at the position vector $\mathbf{x}^{FEM}(t)$ of the high-fidelity FEM model, respectively. In general, different patterns of the spatial error distribution can be observed with respect to the number of POD modes. In the majority of cases, the maximum error is located at the boundary points of the octant $\mathrm{S}_0$ of the initially spherical balloon where plane symmetries are imposed and at points of maximum displacement. Again, observe that the error is not decreasing by using more POD modes as Figure~\ref{fig:Snapshot_L2_norm_balloon_modes_state_1} shows.

\begin{figure}[H]
\begin{minipage}[b]{0.5\linewidth}
\centering
\includegraphics[width=\textwidth]{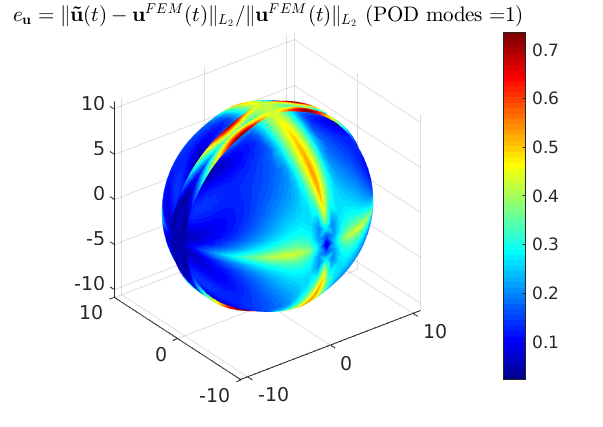}
\end{minipage}
\begin{minipage}[b]{0.5\linewidth}
\centering
\includegraphics[width=\textwidth]{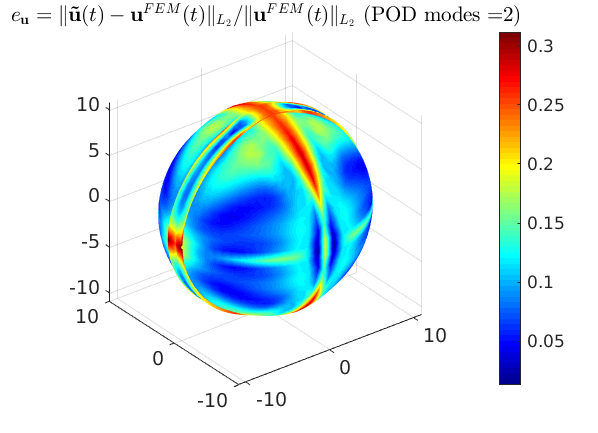}
\end{minipage}

\begin{minipage}[b]{0.5\linewidth}
\centering
\includegraphics[width=\textwidth]{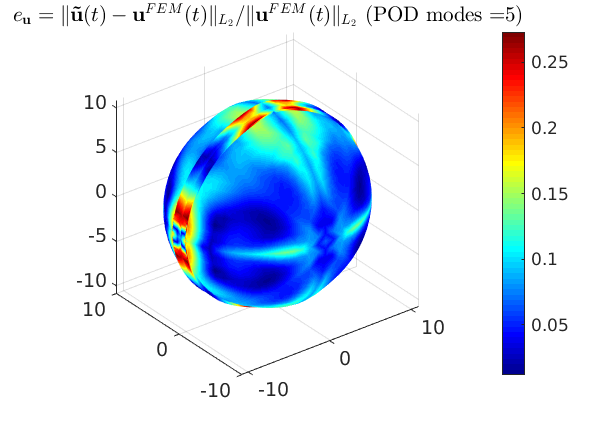}
\end{minipage}
\begin{minipage}[b]{0.5\linewidth}
\centering
\includegraphics[width=\textwidth]{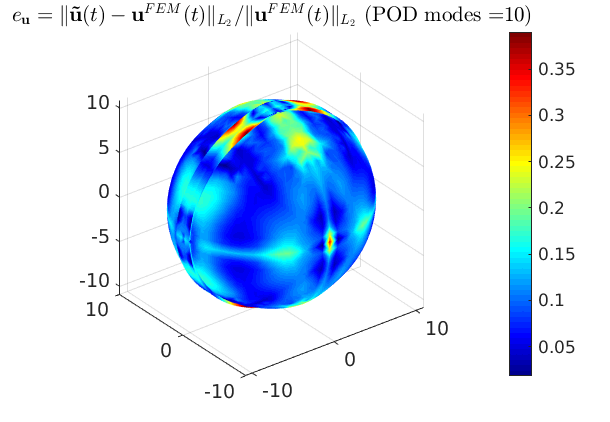}
\end{minipage}

\caption{Relative displacement error $e_{\mathbf{u}} =\Vert \mathbf{\tilde{u}}(t)-\mathbf{u}^{FEM}(t) \Vert_{L_2} / \Vert \mathbf{u}^{FEM}(t) \Vert_{L_2}$  at the nodal points at state $t=0.002$ for POD modes $p= \{1,2,5,10\}$ displayed at the position vector $\mathbf{x}^{FEM}(t)$ of the high-fidelity FEM model; target point: $ \tilde{\bom}(\lambda=75)$.}
\label{fig:Snapshot_L2_norm_balloon_modes_state_12}
\end{figure}

\begin{figure}[H]
\begin{minipage}[b]{0.5\linewidth}
\centering
\includegraphics[width=\textwidth]{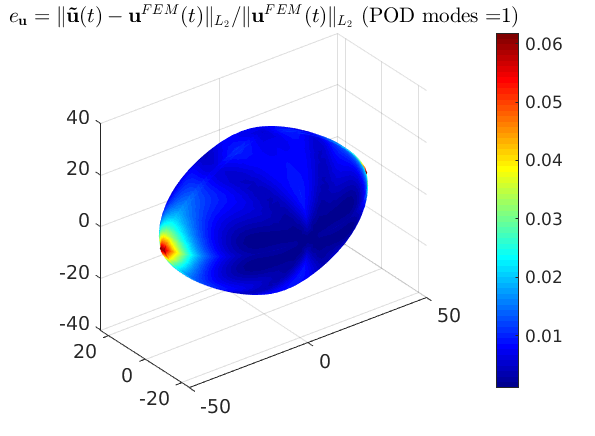}
\end{minipage}
\begin{minipage}[b]{0.5\linewidth}
\centering
\includegraphics[width=\textwidth]{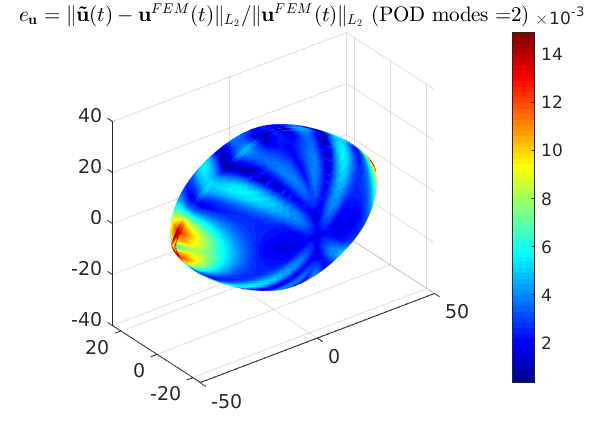}
\end{minipage}

\begin{minipage}[b]{0.5\linewidth}
\centering
\includegraphics[width=\textwidth]{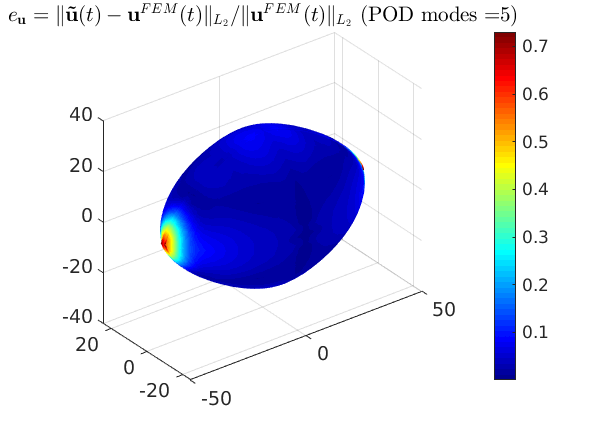}
\end{minipage}
\begin{minipage}[b]{0.5\linewidth}
\centering
\includegraphics[width=\textwidth]{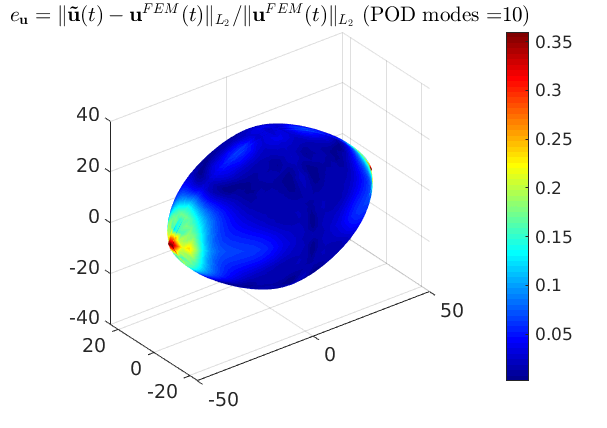}
\end{minipage}
\caption{Relative displacement error $e_{\mathbf{u}} =\Vert \mathbf{\tilde{u}}(t)-\mathbf{u}^{FEM}(t) \Vert_{L_2} / \Vert \mathbf{u}^{FEM}(t) \Vert_{L_2}$  at the nodal points at state $t=1$ for POD modes $p= \{1,2,5,10\}$ displayed at the position vector $\mathbf{x}^{FEM}(t)$ of the high-fidelity FEM model; target point: $ \tilde{\bom}(\lambda=75)$.}
\label{fig:Snapshot_L2_norm_balloon_modes_state_1}
\end{figure}

Finally, Figure~\ref{fig:Total_displacement_radial_points_ROM_FEM_balloon} shows the time-displacement histories for the radial points A, B and C on the initially spherical balloon for the POD ROM-FEM model compared against its high fidelity counterpart solution using POD mode $p=1$. It can be observed that the interpolated ROM-FEM solution delivers good accuracy and is accurate enough to predict the anisotropic balloon inflation at the target parameter.

\begin{figure}[H]
	\centering
	\includegraphics[width=0.6\columnwidth]{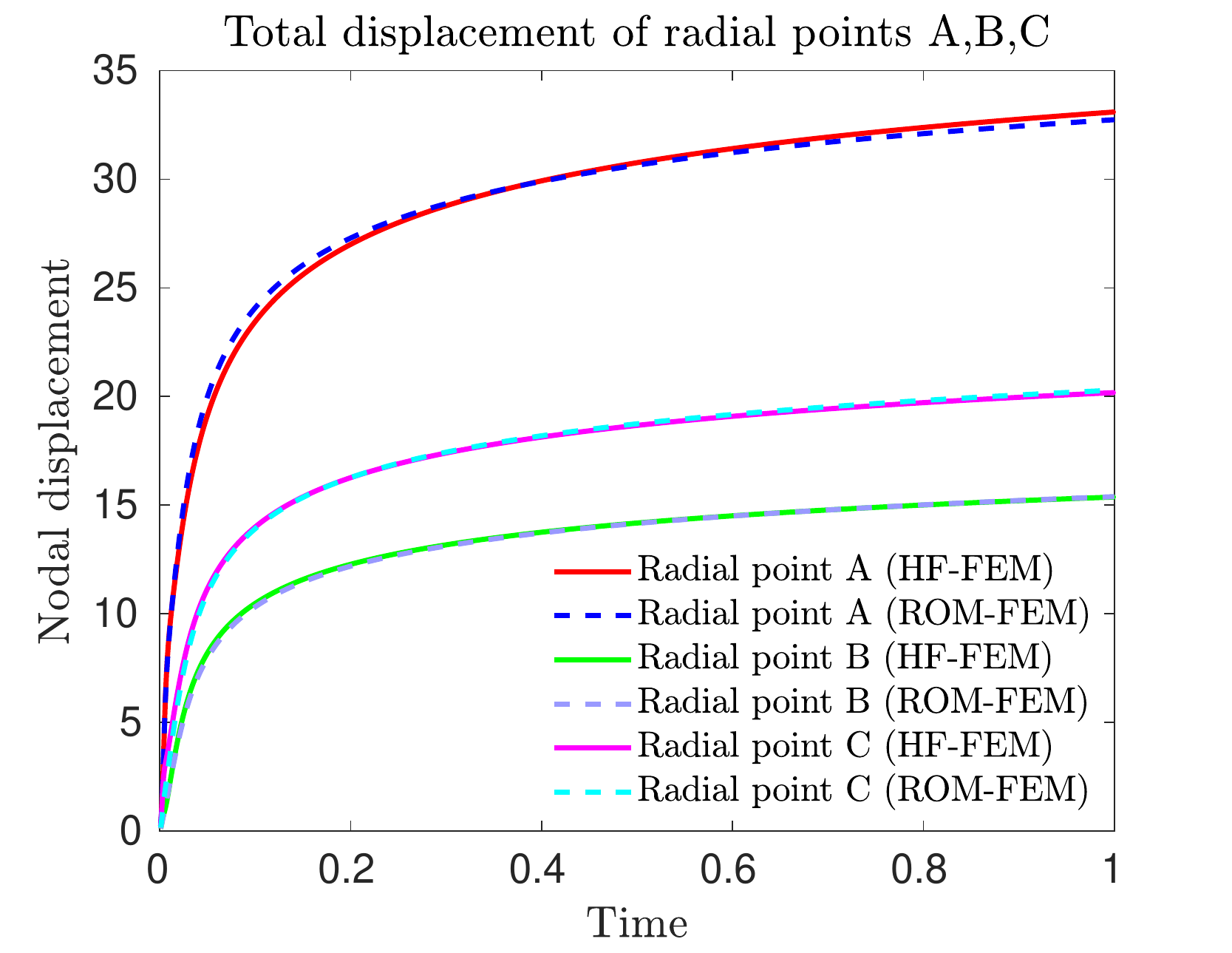}
	\caption{POD ROM-FEM model using Lagrange interpolation; comparison of the displacement of radial points A,B and C against the high-fidelity FEM solution; training points: $\bom_0(\lambda=85)$ (reference point); $\bom_1(\lambda=50)$; $\bom_2(\lambda=60)$; $\bom_3(\lambda=90)$; target point: $ \tilde{\bom}(\lambda=75)$; POD modes $p=1$.}
	\label{fig:Total_displacement_radial_points_ROM_FEM_balloon}
\end{figure}

\subsection{Hyperelastic structure with multiple components}\label{subsec:Benchmark_structural}

In what follows, pMOR is investigated for a hyperelastic structure considering the material stiffness as a parameter. The model consists of two basic components: a plane shell section which is connected with six truss elements (non-symmetrically) (see Figure~\ref{fig:Simple_MV_model}). The plane section has dimensions $20 \times 20$ (mm), a constant thickness of 0.5 mm and is meshed with rectangular shells (S4). The hyperelastic model defined in~\eqref{eq:Strain_energy_MV} (UMAT) is assigned to the plane section in which the fiber orientations are aligned with the x-axis. The following parameters are used: $c_0 = 0.0520$ (kPa),  $c_1 = 4.63$ and  $c_2 = 22.6$. The truss elements are of type T3D2 with a cross-section area of 1 mm$^2$. For these elements, an isotropic incompressible hyperelastic material model is implemented into  Abaqus/Standard subroutine UHYPER~\cite{ABAQUS2014}. The material model is derived from the following strain-energy function

\begin{equation}
U = \alpha_1 (\text{exp}[\alpha_2 (I_1-3)] - 1)
\label{eq:Strain_energy_CT}
\end{equation}

where $\alpha_1$ and $\alpha_2$ are material parameters defined as: $\alpha_1 = 0.0565$ kPa and $\alpha_2$ is used for the parametric analysis. At the boundary  of the plane section ($x=0$) and at the foundations of the truss elements all degrees of freedom are set to zero. A constant hydrostatic pressure of 120 mmHg (0.016 MPa) is applied at the bottom side of the plane section.

\begin{figure}[htbp]
	\begin{minipage}[b]{0.5\linewidth}
		\centering
		\includegraphics[width=\textwidth]{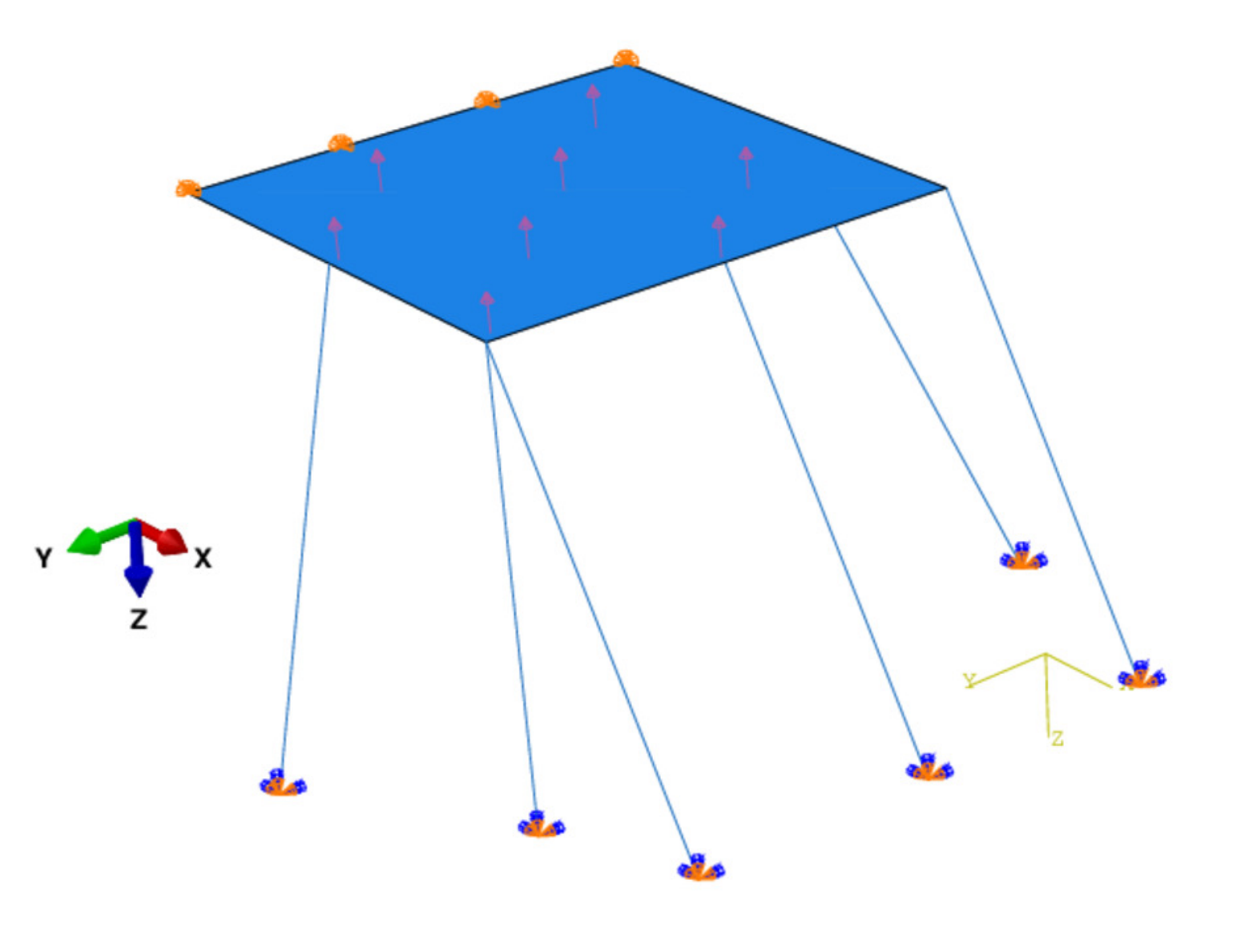}
	\end{minipage}
	\begin{minipage}[b]{0.5\linewidth}
		\centering
		\includegraphics[width=\textwidth]{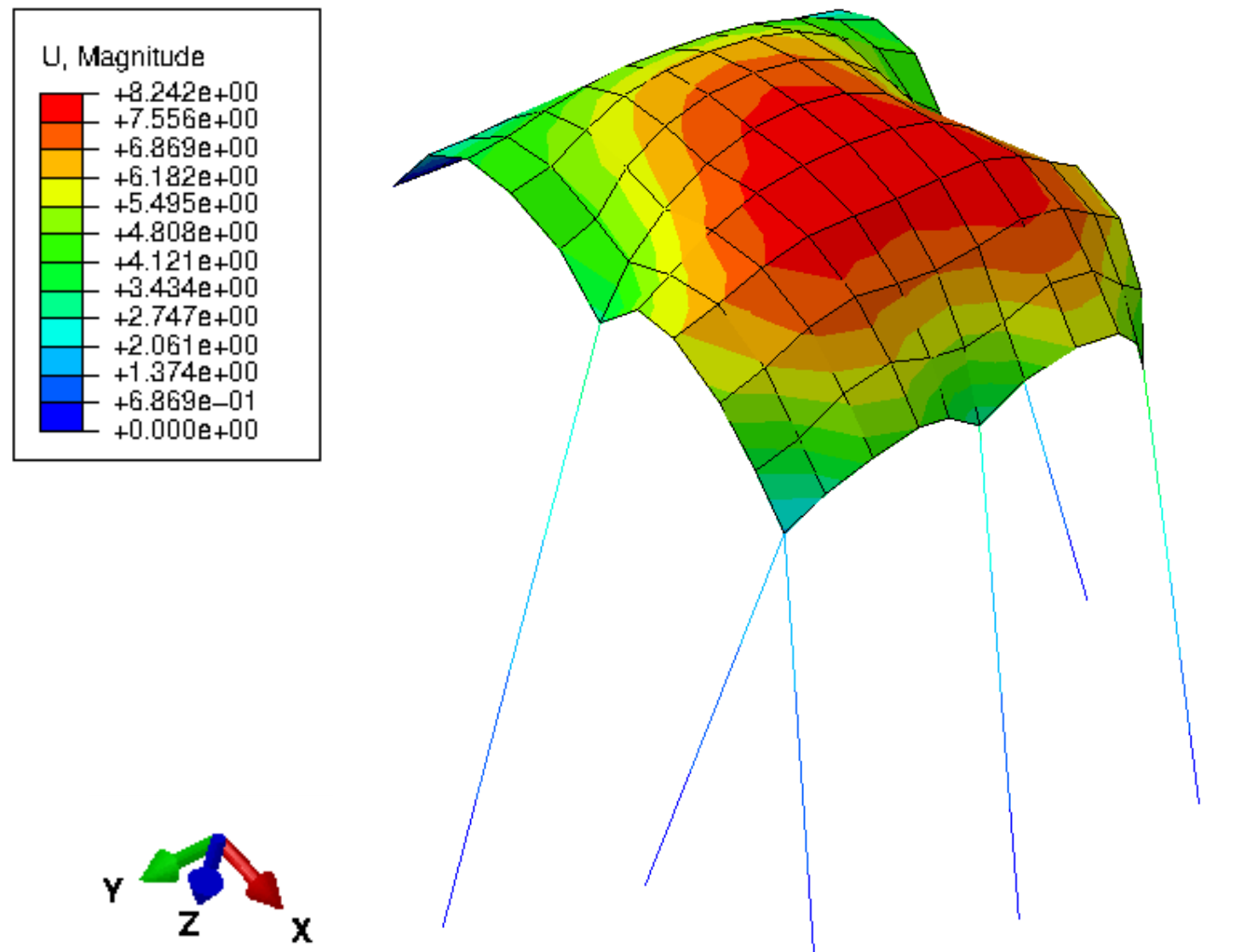}
	\end{minipage}
	\caption{Geometry, boundary conditions and total displacement of the structural multi-component model subjected to hydrostatic pressure, comprised of an anisotropic hyperelastic plane shell section which is non-symmetrically supported by a set of  hyperelastic truss elements.}
	\label{fig:Simple_MV_model}
\end{figure}

\subsubsection*{Snapshot matrices for pMOR}

The FEM simulations are performed using  Abaqus/Standard (Implicit) software.  For the exponential parameter $\alpha_2$, the following set of training points are chosen where for convenience with the previous sections we changed the notation to $\lambda \in \{5, 10, 15, 20, 25, 30 \}$. Figure~\ref{fig:2PKS_stretch} shows the second Piola-Kirchhoff stress-stretch curves for the corresponding parameter values which reveals a wide spectrum of stress values. For each parametric simulation, a sequence of snapshots uniformly distributed over time using an increment of $\Delta t =0.001$ is extracted for all nodes of the plane structure from the model database. The space-time snapshot matrices $\bS(\lambda_i) \in \mathbb{R}^{n \times N_t} $ of size $(n=726) \times (N_t=1000)$ are associated to nodal displacement and rotational fields.
The following training points $\lambda_i \in \Lambda_t = \{ 15,20,25,30 \}$ are chosen for estimating the target point
$\tilde{\lambda}=17.5$. After construction of the set of low-dimensional POD basis for the training points $\lambda_i$, a POD basis for the target point $\tilde{\lambda}$ is interpolated on a Grassmann manifold using Lagrange interpolation. Then, the interpolated POD spatial basis is introduced in Abaqus software using the linear constraint equations (Section~\ref{sec:multi_point_constraint_equations}) to construct a ROM for FEM analysis associated to the target parameter point. For each ROM FEM model of $p$ POD modes,  the same number of reference points are created to assign the interpolated spatial POD modes and the unknown `time' variables that need to be determined.

The eigenvalue spectrum of snapshot matrices $\bS_{i}$ corresponding to training points $\lambda_i \in \Lambda_t$ is shown in a log-log scale in Figure~\ref{fig:Singular_value_magnitude}. It is evident that the distance between the first three eigenvalues is of one order of magnitude each. In our experiments we perform interpolation using $p=1,2,5,10,20$ POD modes since they capture the most important characteristics of the system.

\begin{figure}[H]
	\centering
	\includegraphics[width=0.6\columnwidth]{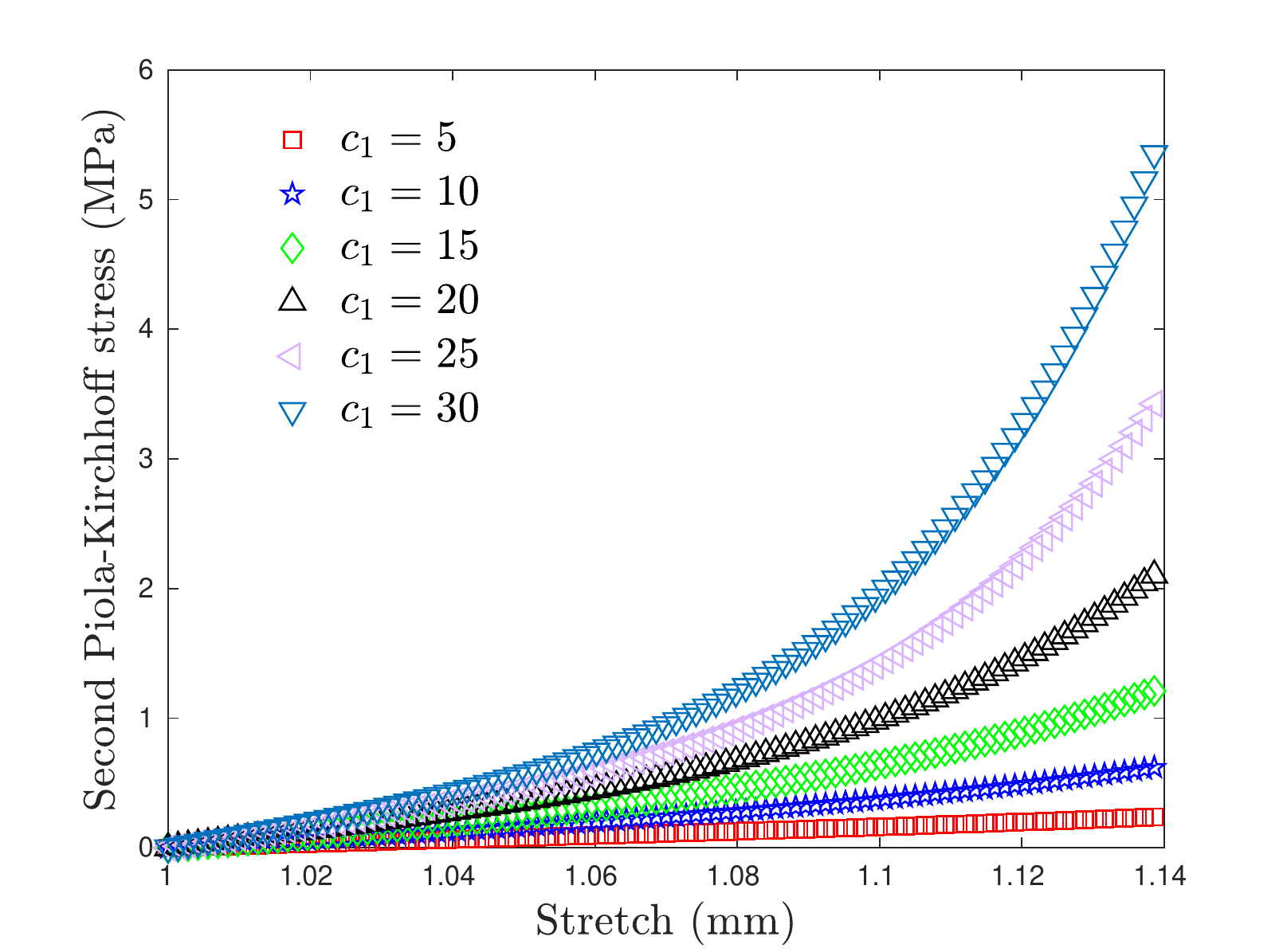}
	\caption{Second Piola-Kirchhoff stress vs stretch for the examined parameter range.}
	\label{fig:2PKS_stretch}
\end{figure}

\begin{figure}[H]
	\centering
	\includegraphics[width=0.6\columnwidth]{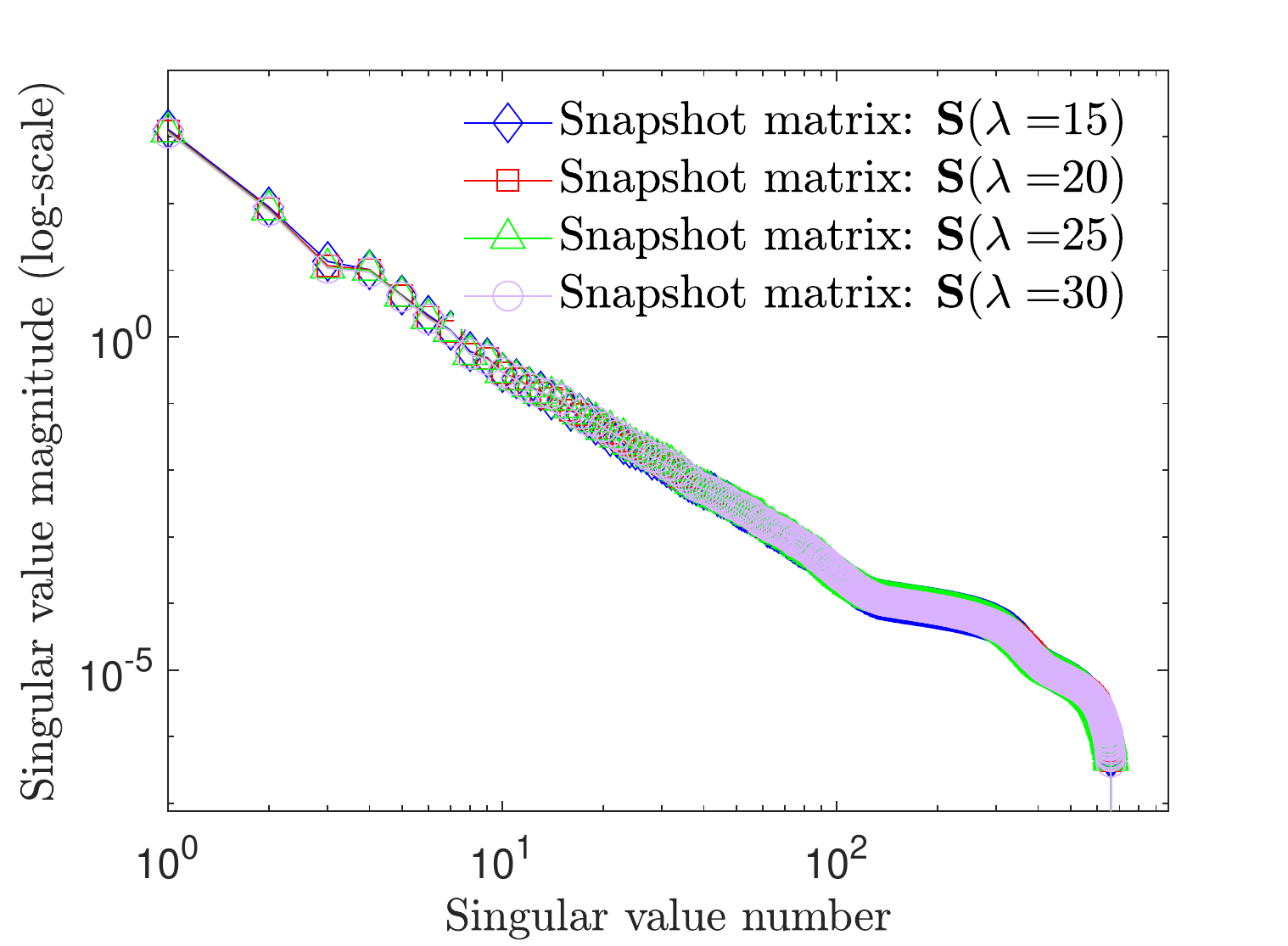}
	\caption{The eigenvalue spectrum of snapshot matrices $\bS_i$ corresponding to training points $\lambda_i \in \Lambda_t = \{ 15,20,25,30 \}$.}
	\label{fig:Singular_value_magnitude}
\end{figure}

\subsubsection*{Stability conditions C1 and C2}

First we need to know if the interpolation is well-defined by evaluating the (C1) and (C2) stability conditions.

\textbf{Stability (C1).}
This condition requires that all points $\bom_1,\dotsc,\bom_N \in \mathcal{G}(p,n)$ lie in $\mathrm{U}_{\bom_0}$ given by~\eqref{eq:Def_Open_Set_Log}. Thus, we need to check if the matrix $\bY_0^{T}\bY_i$ is non-singular for all $i=1,\dots,N$. Since this condition is satisfied for all $i=1,\dots,N$ and $p = 1,2,5,10,20$ POD modes considered in this example, the interpolation is (C1) stable.         
\

\textbf{Stability (C2).}
We need to know if all velocity vectors $\widetilde{v}(\lambda)$ belong to the subset $\mathrm{V}_{\bom_0}$ given by~\eqref{eq:Def_Vm_Angle}, for the parametric range $\lambda \in [\lambda_{1},\lambda_{N} ]$. Thus we have to check that the first (maximum) singular value $\theta_1$ of an horizontal lift $\tilde{\mathbf{Z}}(\lambda)$ of the velocity vector $\widetilde{v}(\lambda)$ is such that $\theta_{1} < \pi/2$, for all $\lambda \in [\lambda_{1},\lambda_{N} ]$.  We proceed by uniformly sampling 151 points over the parametric range $[15;30]$. Figure~\ref{fig:Stability_condition_1_pbm_2} shows the maximum eigenvalue $\theta_1$ of the horizontal lift $\tilde{\mathbf{Z}}(\lambda)$ for all samples using $\bom_0(\lambda=15)$ as a reference point on the Grassmann manifold.
From these curves we are able to assess the (C2) stability of interpolation by detecting the exact  intervals of the loss of injectivity of the Exponential mapping over the parametric range for different number of POD modes $p$. It is clear that for $p \leq 10$ the interpolation is stable over the entire parametric range. Observe  
the loss of injectivity in a specific interval of parameter $\lambda$ for $p=20$ modes. Again, as in the previous example, note that by increasing the dimension $p$, the curves progressively tend to shift closer to $\pi/2$.    Figure~\ref{fig:Stability_condition_1_pbm_2} reveals that interpolation is (C2) stable for the target point $\tilde{\lambda} = 17.5$ for all POD modes p.

\begin{figure}[H]
	\centering
	\includegraphics[width=0.7\columnwidth]{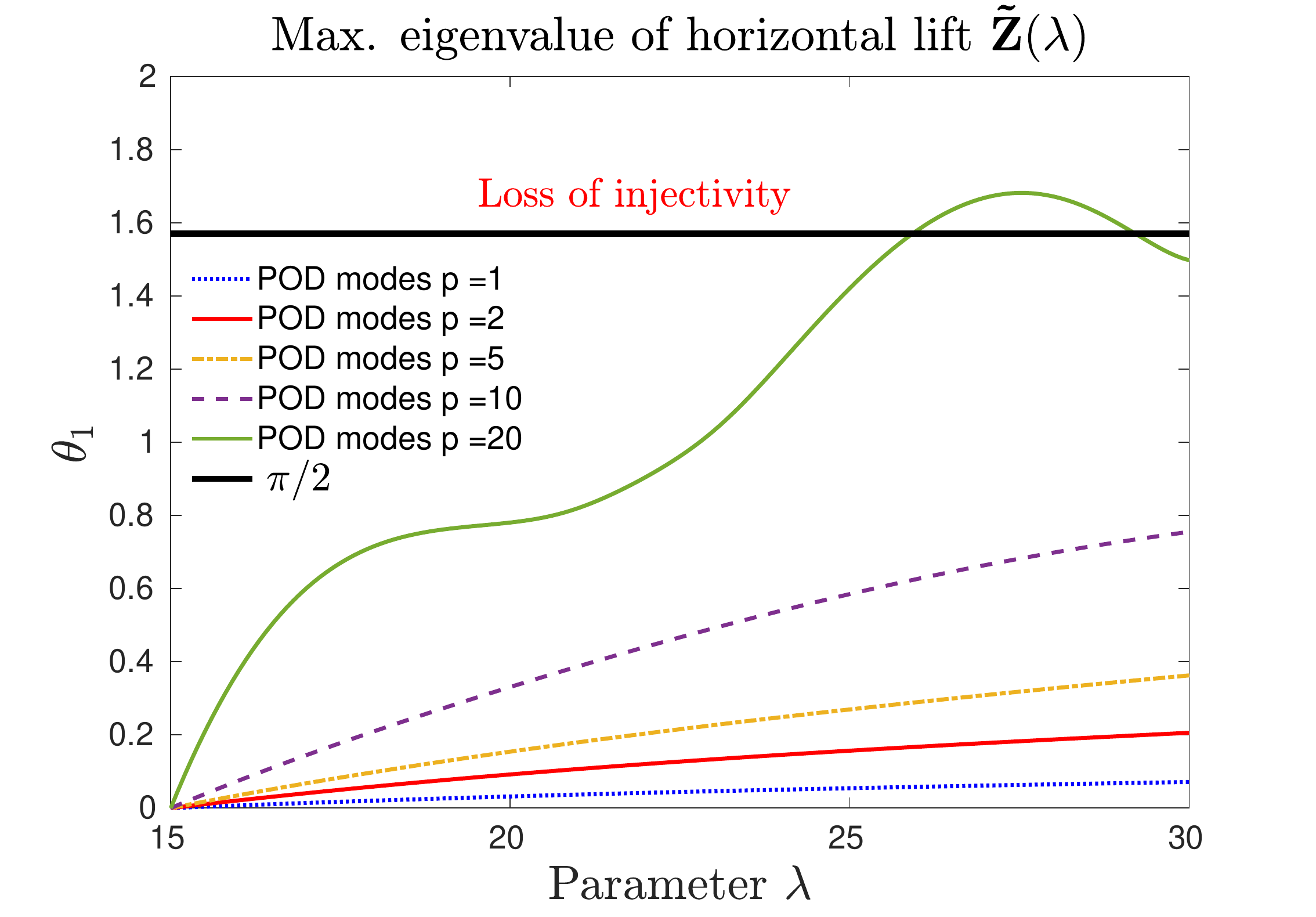}
	\caption{Stability (C2); Computation of the maximum eigenvalue $\theta_1$ of the horizontal lift $\tilde{\mathbf{Z}}(\lambda)$ over the parametric range $[15;30]$. Observe the loss of injectivity in a specific interval of parameters for POD modes $p=20$. Reference point on Grassmann manifold $\bom_0(\lambda=15)$.}
	\label{fig:Stability_condition_1_pbm_2}
\end{figure}

\subsubsection*{Interpolation accuracy and Stability condition (C3)}

The accuracy of interpolation is assessed by comparing the relative $L_2$-error norm $e_{L_2}(\tilde{\bS})$ and the relative Frobenius error norm $e_{F}(\tilde{\bS})$ defined by the ROM FEM model and its  high-fidelity counterpart solution against the number of POD modes $p$, as shown in Figure~\ref{fig:Standard_POD_L2_error_norm} and Figure~\ref{fig:Standard_POD_vector_number_Frobenious_error_norm_pbl_2}, respectively. Additionally, Table~\ref{table:Grassmannian_dim_pbl_2} illustrates the Grassmannian dimension for the corresponding number of POD modes $p$.

\textbf{Stability (C3).}
We need to check if the interpolated subspaces $\widetilde{\mathcal{V}}$ and $\widetilde{\mathcal{V}}'$ respectively associated to matrices $\widetilde{\bY}$ and $\widetilde{\bY}'$ correspond to mode $p$ and $p'>p$ interpolation, are such that $\widetilde{\mathcal{V}}\subset \widetilde{\mathcal{V}}'$.
Before performing this stability test, observe the monotonic decrease of the  relative error norms \eqref{eq:L2_error_norm_HF} and \eqref{eq:Frobenious_error_norm} by increasing mode $p$, as depicted in Figure~\ref{fig:Standard_POD_L2_error_norm} and Figure~\ref{fig:Standard_POD_vector_number_Frobenious_error_norm_pbl_2}, respectively. We are now ready to see how the geometric distance $\delta(\mathcal{V},\mathcal{V}')$ using the principal angles defined in \eqref{eq:Geom_Distance} relate to the error norm behavior.
Again, we assume a set of POD modes $p \in \mathscr{P}_m = \{1,2,5,10,20\}$ and a threshold value $T_V =100$.
To this end, we compute the distances $\delta(\widetilde{\bY},\widetilde{\bY}')$ of the interpolated POD basis on Grassmann manifolds $\mathcal{G}(p,n)$ of various dimensions $p \in \mathscr{P}_m$, plotted in a symmetric table form, as Figure~\ref{fig:Distance_subspaces_different_dimensions_pbl_2} shows. Again, the results prove the non-connectivity of different subspaces of various dimensions $p$. What is remarkable to observe in this case, is that the geometric distance $\delta(\widetilde{\bY},\widetilde{\bY}') \approx 0$ for all $p \neq p'$.  
Moreover, the relative error $\epsilon$ given by~\eqref{eq:Threshold_Stab_C3} is here $\epsilon=73.60 < T_V $, which is sufficiently small to assure a (C3) stable interpolation.

Finally, Figure~\ref{fig:Time_displacement_histories_ROM_FEM_POD_20} shows a comparison of the predicted time histories of selected nodal total displacements for the ROM FEM model using $p = 20$ POD modes against the high fidelity FEM solution. It is evident that all nodal time-histories are nearly identical.

\begin{figure}[http]
	\centering
	\includegraphics[width=0.7\columnwidth]{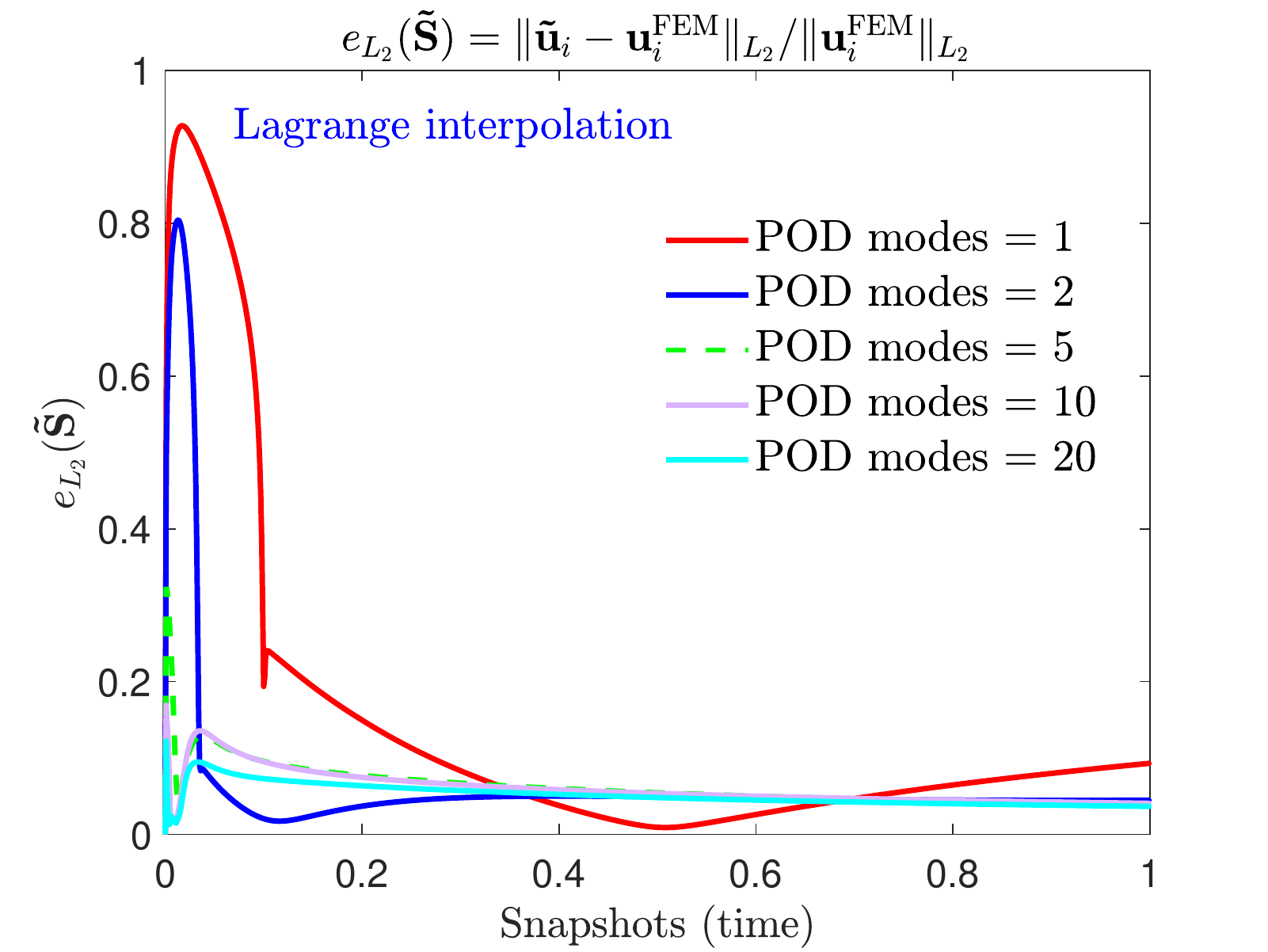}
	\caption{POD ROM-FEM; relative $L_2$-error norm $e_{L_2}(\tilde{\bS})$ against the number of POD vectors; target point: $ \tilde{\bom}(\lambda=17.5)$.}
	\label{fig:Standard_POD_L2_error_norm}
\end{figure}

\begin{figure}[H]
	\centering
	\includegraphics[width=0.65\columnwidth]{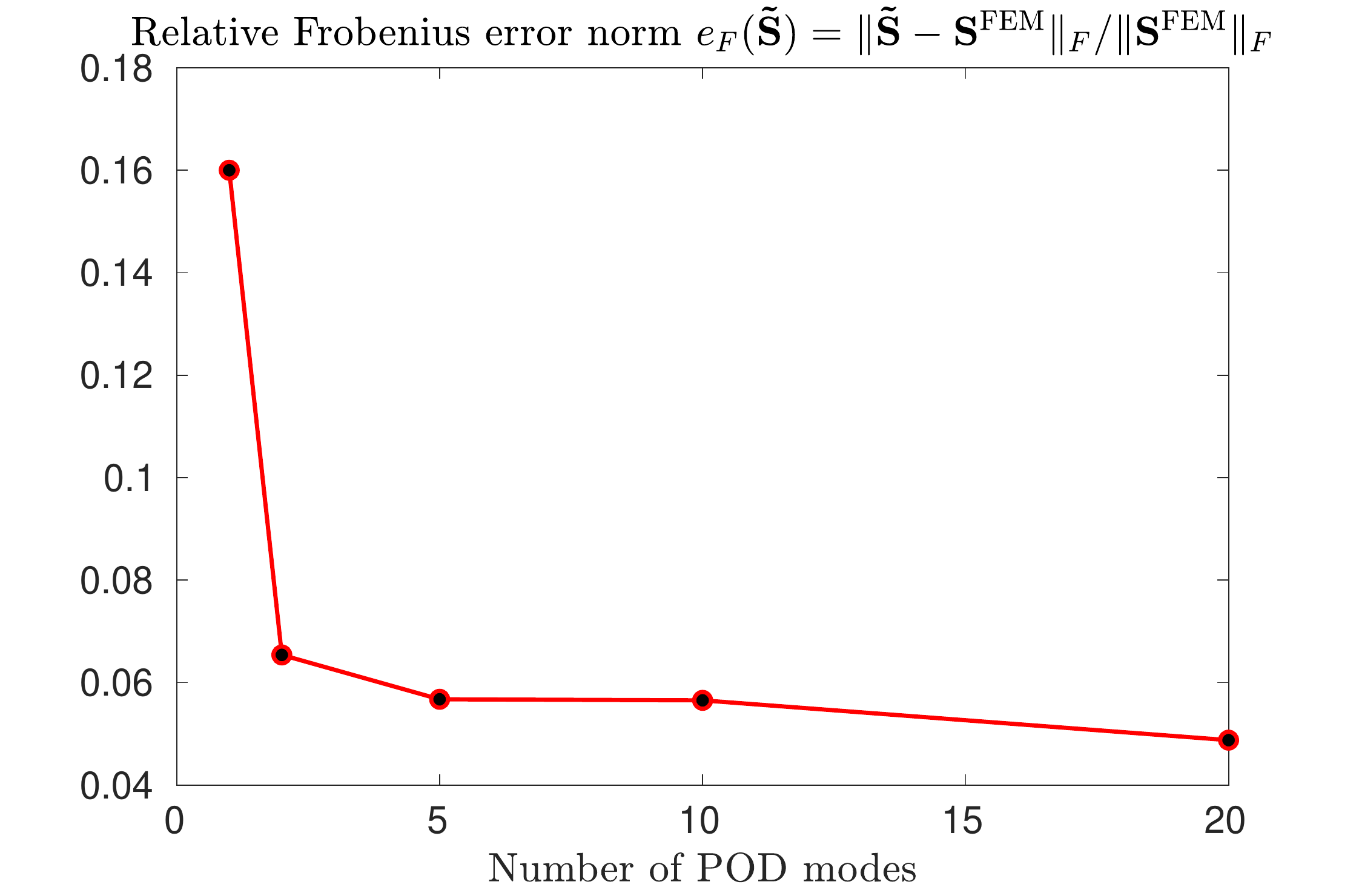}
	\caption{Relative Frobenius error norm against the number of POD vectors for the POD ROM-FEM; target point: $ \tilde{\bom}(\lambda=17.5)$.}
	\label{fig:Standard_POD_vector_number_Frobenious_error_norm_pbl_2}
\end{figure}

\begin{table}[ht]
	\caption{Dimension of the Grassmann manifold $\mathcal{G}(p,n)$}
	\centering
	\begin{tabular}{ c c c c c c}
		\hline 
		\textbf{Number of modes} & \textbf{$p=1$} & \textbf{$p=2$} & \textbf{$p=5$} & \textbf{$p=10$} & \textbf{$p=20$}      \\ \hline
		\textbf{Dimension: $p(n-p)$ } & 725 & 1448 & 3605 & 7160 & 14120   \\ \hline
	\end{tabular}
	\label{table:Grassmannian_dim_pbl_2}
\end{table}

\begin{figure}[H]
	\centering
	\includegraphics[width=0.6\columnwidth]{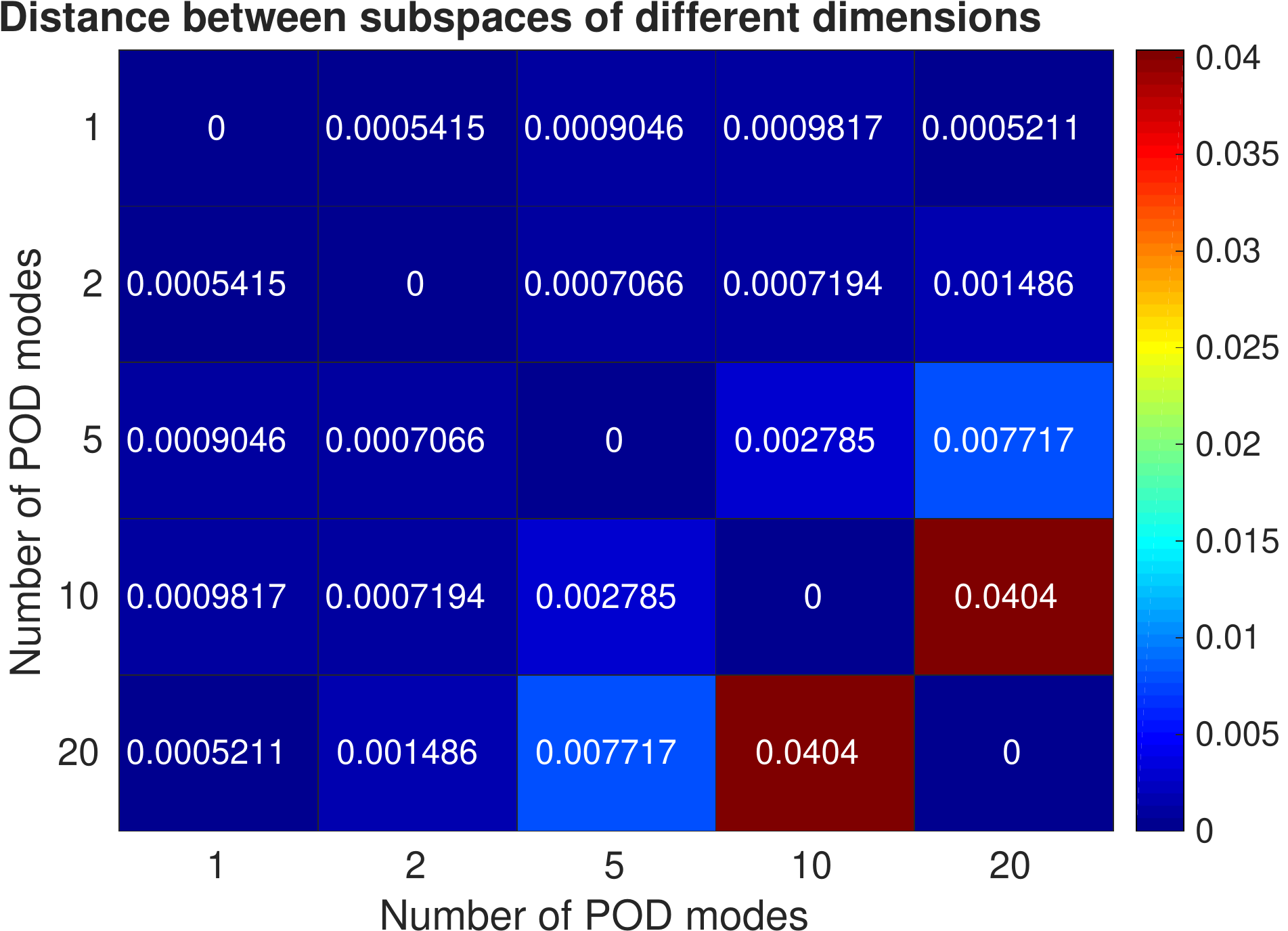}
	\caption{Stability (C3); Geometric distance $\delta(\bY,\bY')$ between interpolated subspaces of different dimensions $p$.}
	\label{fig:Distance_subspaces_different_dimensions_pbl_2}
\end{figure}

\begin{figure}[http]
	\centering
	\includegraphics[width=0.7\columnwidth]{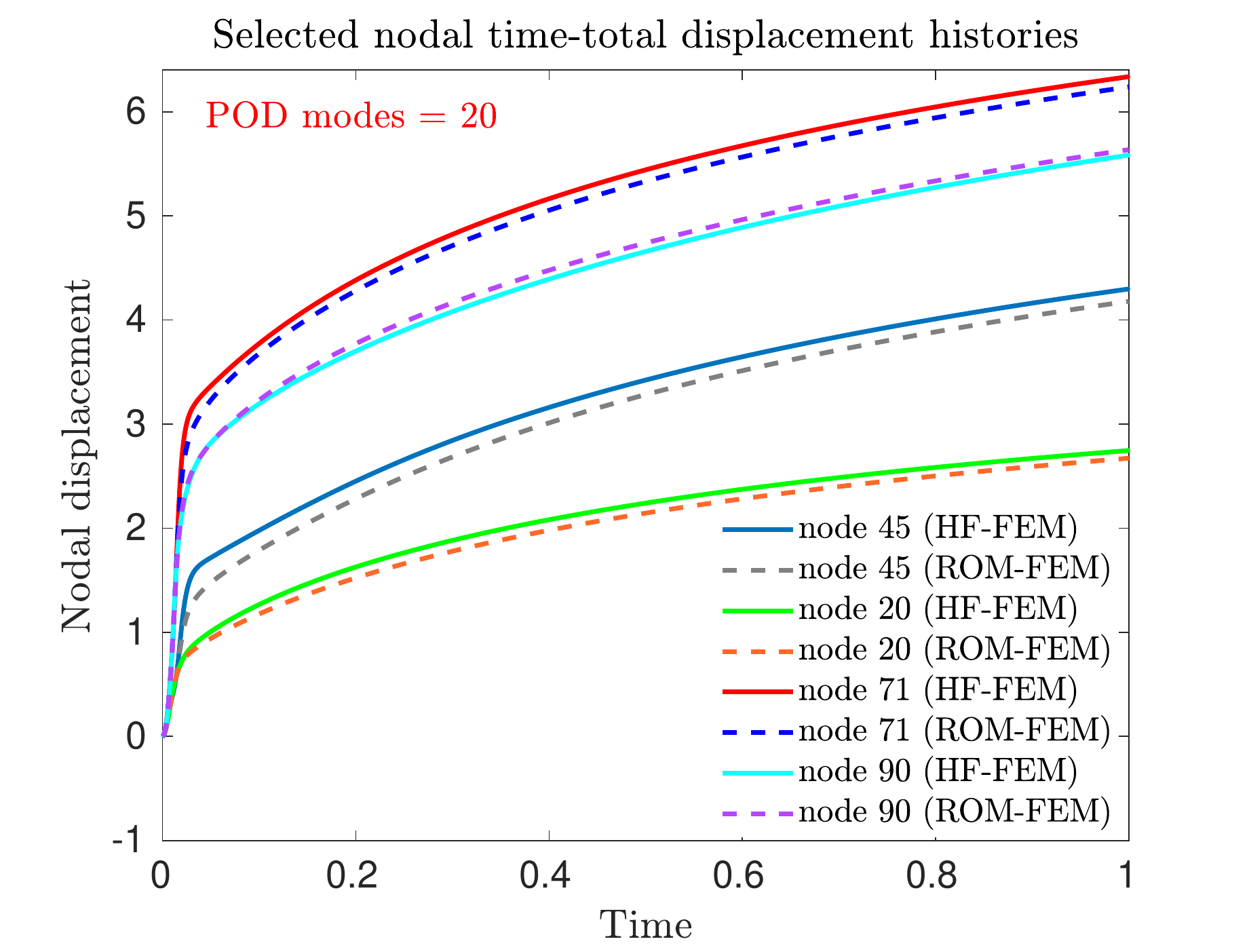}
	\caption{POD ROM-FEM; comparison of selected nodal time-displacement histories against the high-fidelity FEM solution; training points: $\bom_0(\lambda=15)$; $\bom_1(\lambda=20)$; $\bom_2(\lambda=25)$; $\bom_3(\lambda=30)$; target point: $ \tilde{\bom}(\lambda=17.5)$; POD modes = 20.}
	\label{fig:Time_displacement_histories_ROM_FEM_POD_20}
\end{figure}

\section{Conclusions}
\label{sec:Conclusions}

Effective mathematical definitions for stability conditions of POD basis interpolation on Grassmann manifolds for pMOR in hyperelasticity are given. Special attention has been paid on the definition of local maps on Grassmann manifolds considering the Logarithm and Exponential maps. In this context, the notion of cut--locus is introduced since it optimally captures the loss of injectivity of the exponential map. 
The formulae for the  Grassmannian cut--locus to establish a stable interpolation is mathematically proved. 
Another intrinsic stability condition is defined by computing the geometric distance of the interpolated POD basis of different mode. This enables us to explain intrinsic oscillations of the error norm with increasing mode, and on the contrary, solutions with  monotonic behavior. The
pMOR benchmark examples revealed important aspects of stability.   

\section{Acknowledgements}
This work has been founded by DGA (``direction générale pour l’armement", French ministry of defense) under the RAPID contract called ``Innvivotech Tissus Mous" in partnership with BIOMODEX.


\appendix

\renewcommand{\thethm}{\Alph{section}.\arabic{thm}}

\section{Riemannian geometry of Grassmann Manifolds}
\label{sec:Grassmann_Manifolds}

The purpose of this appendix is to recall main results about Grassmann manifolds, as well as new ones about the \emph{cut--locus} and injectivity condition for the exponential map.
As far as we know, the normal coordinates are classically defined using the exponential map restricted on an open disk deduced from the injectivity radius~\cite{Kozlov2000,Mosquera2019b}. In fact, it will be possible to go beyond such an injectivity radius, using an open set deduced from the \emph{cut--locus} of the Grassmann manifold, all this being detailed in~\ref{subsec:Cut_Locus}.

Note that some results recalled here are classical, either given in their matrix forms~\cite{Absil2004,Amsallem2009,Edelman1998g,Mosquera2018,Mosquera2019b,Oulghelou2018}, or given in a more abstract one~\cite{Kozlov1997,Kozlov2000}, but it was necessary to write them back for our proofs to be clearly established. Note also that all details about general differential Riemannian geometry can be found in~\cite{Lee2013,Boothby1986,Gallot1990}.

From now on, let us consider two integers $p,n$ such that $p\leq n$ and take $\mathcal{G}(p,n)$ to be the Grassmann manifold of $p$ dimensional subspaces of $\RR^n$. A first way to obtain a point $\bom\in \mathcal{G}(p,n)$ is to consider a basis $\by_1,\dots,\by_p$ of the associated subspace
\begin{equation*}
\bom=\text{Vect}(\by_1,\dots,\by_p).
\end{equation*}
Without loss of generality, it is possible to restrict oneself in the case of \emph{orthonormal basis}, so $\bom$ can be represented by a matrix
\begin{equation*}
\bY:=[\by_1,\dotsc,\by_p]\in \VecMat{n}{p},\quad \bY^T\bY=\bI_p.
\end{equation*}
Such matrix $\bY$ is not unique, as any matrix in the set
\begin{equation*}
\left\{\bY\bP,\quad \bP\in \mathrm{O}(p)\right\},\quad \mathrm{O}(p):=\left\{\bP\in \VecMat{p}{p},\quad \bP^T\bP=\bI_p\right\},
\end{equation*}
can represent the same point $\bom$.

From this, the Grassmann manifold $\mathcal{G}(p,n)$ is obtained as a \emph{quotient space}~\cite[Chapter 21]{Lee2013} of the (compact) space of $p$ ordered orthonormal vectors of $\RR^n$. More specifically~\cite[Appendix C.2]{Helmke2012}, first define the compact Stiefel manifold $\mathcal{S}t^{c}(p,n)$ to be the set of $p$ orthonormal vectors $\{\by_1,\dotsc,\by_p\}$ of $\RR^n$. Taking any basis of $\RR^n$, such a set can be represented by a rank $p$ matrix
\begin{equation*}
\bY:=[\by_1,\dotsc,\by_p]\in \VecMat{n}{p},\quad \bY^T\bY=\bI_p.
\end{equation*}
This led to define a \emph{fiber bundle}~\cite{Kobayashi1996,Ferrer1994}, which is also a \emph{submersion}~\cite{Lee2013}:
\begin{equation}\label{eq:Sub_Stiefel_NC}
\pi \: : \: \bY \in \mathcal{S}t^{c}(p,n)\mapsto \pi(\bY)=\bom:=\{ \bY\bP,\quad \bP\in \mathrm{O}(p)\} \in \mathcal{G}(p,n)
\end{equation}
Informally speaking, it means that any point $\bom$ of the Grassmann manifold $\mathcal{G}(p,n)$ can be represented by any point $\bY$ of the \emph{fiber} $\pi^{-1}(\bom)$ (Figure~\ref{fig:Fiber_bundle}).

\begin{figure}[http]
	\centering
	\includegraphics[scale=0.65]{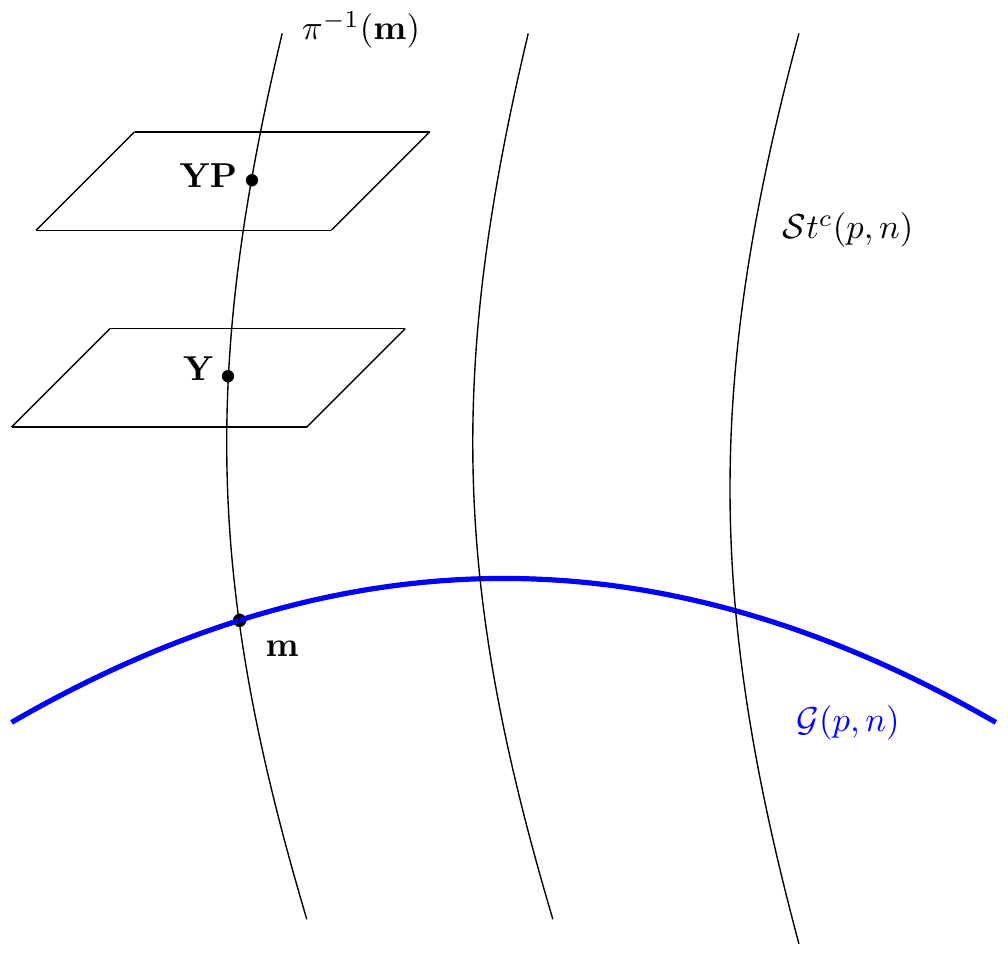}
	\caption{Schematic of a fiber bundle.}  
	\label{fig:Fiber_bundle}
\end{figure}

\subsection{The Grassmann Manifold and its Riemannian metric}

From the submersion $\pi$ given by~\eqref{eq:Sub_Stiefel_NC}, the Grassmann manifold $\mathcal{G}(p,n)$ can inherit the geometry of the Stiefel manifold $\mathcal{S}t^{c}(p,n)$ and its Riemannian structure~\cite{Gallot1990}.

First, the Stiefel manifold $\mathcal{S}t^{c}(p,n)\subset\VecMat{n}{p}$, is naturally endowed with an inner product
given by
\begin{equation*}
\langle \mathbf{Z}_1,\mathbf{Z_2}\rangle:=\tr(\mathbf{Z}_1^{T}\mathbf{Z_2}),\quad \bZ_1,\bZ_2\in \VecMat{n}{p}.
\end{equation*}
Now, we need to attach, to each $\bom\in \mathcal{G}(p,n)$ a \emph{tangent space} $T_{\bom}\mathcal{G}(p,n)$, which is a vector space isomorphic to $\RR^{p\times (n-p)}$, equipped with a scalar product (depending smoothly on $\bom$), so that $\mathcal{G}(p,n)$ becomes a Riemannian manifold.

In fact, there is no canonical way to get a representation of a velocity vector $v\in T_{\bom}\mathcal{G}(p,n)$, as it depends on the choice of a matrix $\bY\in \mathcal{S}t^{c}(p,n)$ defining $\bom$ (see Figure~\ref{fig:Fiber_bundle}): for any $\bY\in \pi^{-1}(\bom)$, we define indeed its associated \emph{horizontal space} by:
\begin{equation}\label{eq:Def_Hor_Lift}
\text{Hor}_{\bY}:=\{ \bZ \in \VecMat{n}{p},\quad \bZ^{T}\bY=\mathbf{0} \}.
\end{equation}
Finally:
\begin{enumerate}
	\item The tangent space $T_{\bom}\mathcal{G}(p,n)$ is isomorphic to any $\text{Hor}_{\bY}$ with $\bY$ such that $\pi(\bY)=\bom$. An isomorphism is given by
	\begin{equation*}
	\text{d}\pi_{\bY \mid\text{Hor}_{\bY}}\: : \: \text{Hor}_{\bY}\longmapsto T_{\bom}\mathcal{G}(p,n).
	\end{equation*}
	\item For any $v\in T_{\bom}\mathcal{G}(p,n)$, the unique $\bZ\in \text{Hor}_{\bY}$ such that
	\begin{equation}\label{eq:Hor_Lift}
	\text{d}\pi_{\bY}\cdot \bZ=v
	\end{equation}
	is called a \emph{horizontal lift} of $v$.
	\item For any $\bP\in \mathrm{O}(p)$, then $\bZ\bP$ is another horizontal lift of $v$ (but belonging to the vector space $\text{Hor}_{\bY\bP}$) and
	\begin{equation*}
	\text{d}\pi_{\bY\bP}\cdot (\bZ\bP)=v.
	\end{equation*}
\end{enumerate}
The Riemannian metric on the Grassmannian $\mathcal{G}(p,n)$ is then defined by
\begin{equation*}
\langle v_1,v_2\rangle_{\bom}:=\langle \bZ_1,\bZ_2\rangle_{\bY},\quad
\end{equation*}
with $\pi(\bY)=\bom$ and $\bZ_{1}$ (resp. $\bZ_{2}$) a horizontal lift of $v_1$ (resp. $v_2$) in $\text{Hor}_{\bY}$.

For the proofs of the following subsections, an interesting geometric approach, due to Zhou~\cite{Zhou1998}, is given by:

\begin{lem}\label{lem:Hor_Lift_Intrisic}
	Let $\bom\in \mathcal{G}(p,n)$ and $v\in T_{\bom}\mathcal{G}(p,n)$, with $2p\leq n$. Then there exists an orthonormal basis $\by_1,\dots,\by_n$ of $\RR^n$ such that
	\begin{multline*}
	\bY=[\by_1,\dots,\by_p]\in \pi^{-1}(\bom),\quad \bZ=[\theta_1\by_{p+1},\cdots,\theta_{p}\by_{2p}]\in \text{\emph{Hor}}_{\bY},\\
	\theta_{1}\geq \cdots \geq \theta_{p}\geq 0.
	\end{multline*}
\end{lem}

\begin{proof}
	Let us consider any $\bY\in \pi^{-1}(\bom)$ and a horizontal lift $\bZ$ of $v$ such that $\bZ^{T}\bY=\mathbf{0}$. We define a thin singular value decomposition of $\bZ$, so we can find orthonormal vectors $\uu_1,\dots,\uu_p$ in $\RR^n$ and $\vv_1,\dots,\vv_p$ in $\RR^p$ such that
	\begin{equation*}
	\bZ=\sum \theta_i \uu_i \vv_i^{T},\quad \theta_{1}\geq \cdots \geq \theta_p\geq 0.
	\end{equation*}
	From the condition $\bZ^T\bY=\mathbf{0}$ we thus deduce that $\by_1,\dots,\by_p,\uu_1,\dots,\uu_p$ is a family of orthonormal vectors. Taking now the matrix $\bP:=[\vv_1,\dots,\vv_p]\in \mathrm{O}(p)$ and $\bY':=\bY\bP\in \pi^{-1}(\bom)$, we obtain
	\begin{equation*}
	\bZ':=[\theta_1\by_{p+1},\cdots,\theta_{p}\by_{2p}]\in \text{Hor}_{\bY\bP},\quad \by_{p+i}:=\uu_i,
	\end{equation*}
	so we can conclude.
\end{proof}

\begin{rem}
	In the case when $2p>n$, that is $p>n-p$, then we can only write a horizontal lift as
	\begin{equation*}
	\bZ=[\theta_{1}\by_{p+1},\cdots,\theta_{n-p}\by_{n},\underbrace{\mathbf{0},\dots,\mathbf{0}}_{2p-n \text{ times }}],\quad \theta_{1}\geq \cdots \geq \theta_{n-p}\geq 0.
	\end{equation*}
\end{rem}

\subsection{Geodesics and distance on Grassmann manifolds}

The Grassmann manifold $\mathcal{G}(p,n)$ being equipped with a Riemannian metric, it is possible to define the length  of any curve $c:[0;1]\rightarrow \mathcal{G}(p,n)$:
\begin{equation}\label{eq:Curve_Lenght}
L(c)=\int_{0}^{1} \langle \dot{c}(t),\dot{c}(t)\rangle_{c(t)}\text{d}t
\end{equation}
and so the associated \emph{Riemannian distance}  
\begin{equation}\label{eq:Riem_Distance}
d_{r}(\bom,\bom'):=\Inf \{L(c),\quad c(0)=\bom,c(1)=\bom'\}.
\end{equation}
To obtain an explicit computation of such a distance, one can use the \emph{geodesics} obtained from the Riemannian metric and its associated \emph{Levi-Civita connection}~\cite{Gallot1990,Lee2013} (see also~\cite[III.6]{Kobayashi1996}). First recall that for Grassmann manifold, geodesics are obtained explicitly~\cite{Kozlov1997,Absil2004}:
\begin{thm}\label{thm:Eq_Geo_Grass}
	Let $\bom\in \mathcal{G}(p,n)$ and $v\in T_{\bom}\mathcal{G}(p,n)$ with horizontal lift given by $\bZ\in \text{Hor}_{\bY}$, where $\pi(\bY)=\bom$ and $\bY^{T}\bY=\bI_p$. Let $\bZ=\bU\bThe\bV^{T}$ be a thin singular value decomposition of $\bZ$. Then
	\begin{equation}\label{eq:geo_v}
	\alpha_{v} : \: t\in \RR\mapsto \pi\left( \bY\bV\cos(t\bThe)+\bU\sin(t\bThe)\right) \in \mathcal{G}(p,n)
	\end{equation}
	is the unique maximal geodesic such that $\alpha_{v}(0)=\bom$ and $\dot{\alpha_{v}}(0)=v$, maximality meaning here that such curve is defined on all $\RR$.
\end{thm}

\begin{rem}\label{rem:Geodesics_Intrinsic}
	There is another approach proposed in~\cite{Zhou1998} which produces a more intrinsic formula for the geodesics. Indeed, let us consider $2p\leq n$ and take back the result from Lemma~\ref{lem:Hor_Lift_Intrisic}. Then one horizontal lift of $v$ can writes
	\begin{equation*}
	\bZ=[\theta_1\by_{p+1},\cdots,\theta_{p}\by_{2p}],\quad \theta_{1}\geq \cdots \geq \theta_{p}\geq 0.
	\end{equation*}
	where $\bY=[\by_1,\dots,\by_p]\in \pi^{-1}(\bom)$ and $\by_1,\dots,\by_{2p}$ is an orthonormal family. The unique geodesic obtained from velocity vector $v$ is then defined by $\pi(\bY(t))$, with
	\begin{equation*}
	\bY(t)=\left[\cos(\theta_{1}t)\by_1+\sin(\theta_{1}t)\by_{p+1},\dots,\cos(\theta_{p}t)\by_p+\sin(\theta_{p}t)\by_{2p}\right].
	\end{equation*}
	We observe that the norm of the velocity vector is given by
	\begin{equation*}
	\|v\|=\sqrt{\sum \theta_{i}^2}.
	\end{equation*}
\end{rem}

In fact, all matrices given by~\eqref{eq:geo_v} are lying in $\mathcal{S}t^c(p,n)$:
\begin{lem}\label{lem:Geod_Inside_Stiefel_Compact}
	Let $\bom\in \mathcal{G}(p,n)$ and $v\in T_{\bom}\mathcal{G}(p,n)$. Take $\bY\in \pi^{-1}(\bom)$ and $\bZ\in \text{Hor}_{\bY}$ like in statement of Theorem~\ref{thm:Eq_Geo_Grass}. Then for any $t\in \RR$ we have
	\begin{equation*}
	\bY(t):=\bY\bV\cos(t\bThe)+\bU\sin(t\bThe)\in \mathcal{S}t^c(p,n) \text{, meaning that } \bY(t)^{T}\bY(t)=\bI_p.
	\end{equation*}
\end{lem}

\begin{proof}
	By direct computation we have:
	\begin{align*}
	\bY(t)^{T}\bY(t)&=\cos^2(t\bThe)+\sin^{2}(t\bThe)+\mathbf{X}+\mathbf{X}^{T},
	\quad \mathbf{X}:=\sin(t\bThe)\bU^{T}\bY\bV\cos(t\bThe) \\
	&=\bI_p+\mathbf{X}+\mathbf{X}^{T}.
	\end{align*}
	As we have $\bZ^{T}=\bU\bThe\bV^{T}$ and $\bZ^{T}\bY=0$ we deduce that
	\begin{equation*}
	\bV\bThe\bU^{T}\bY=0,\quad \bV\in \mathrm{O}(p)\implies \bThe\bU^{T}\bY=0.
	\end{equation*}
	and thus $\sin(t\bThe)\bU^{T}\bY=0$ for all $t$, which conclude the proof.
\end{proof}

As a consequence of Hopf-Rinow Theorem~\cite[Theorem 2.103]{Gallot1990}, any two points of the Grassmann manifold can be joined by a length minimizing geodesic. An explicit expression of such a geodesic is given by (see also~\cite{Kozlov2000}):
\begin{thm}\label{thm:Min_Geodeiscs_between_Two_Points}
	Let $\bom,\bom'\in \mathcal{G}(p,n)$ be any two points on the Grassmann manifold $\mathcal{G}(p,n)$. Then, for $2p\leq n$:
	\begin{enumerate}[(1)]
		\item There exists an orthonormal family $\by_1,\dots,\by_n$ of $\RR^n$ such that
		\begin{align*}
		\bY'&=[\cos(\theta_1)\by_1+\sin(\theta_1)\by_{p+1},\dots,\cos(\theta_p)\by_p+\sin(\theta_p)\by_{2p}]\in \pi^{-1}(\bom'),\\
		\bY&=[\by_1,\dots,\by_p]\in \pi^{-1}(\bom),
		\end{align*}
		with $\theta_i\in \left[0,\pi/2\right]$ are the \emph{Jordan's principal angles} between $\bY$ and $\bY'$, meaning that $\theta_i=\arccos(\sigma_{p-i+1})$, where $0\leq \sigma_p\leq \dots\leq \sigma_1$ are the singular values of $\bY^{T}\bY'$.
		\item A length minimizing geodesic from $\bom$ to $\bom'$ is given by $t\in [0,1] \mapsto \pi(\bY(t))$ with
		\begin{equation*}
		\bY(t):=[\cos(t\theta_1)\by_1+\sin(t\theta_1)\by_{p+1},\dots,\cos(t\theta_p)\by_p+\sin(t\theta_p)\by_{2p}].
		\end{equation*}
		Furthermore, such length minimizing geodesic is unique if and only if $\theta_{1}<\pi/2$.
	\end{enumerate}
	In the case $2p>n$, the same result holds using
	\begin{multline*}
	\bY'=[\cos(\theta_1)\by_1+\sin(\theta_1)\by_{p+1},\dots,\cos(\theta_{n-p})\by_{n-p}+\sin(\theta_{n-p})\by_{n-p},\\
	\by_{n-p+1},\dots,\by_p]\in \pi^{-1}(\bom).
	\end{multline*}
\end{thm}

\begin{proof}
	Take any $\bY\in \pi^{-1}(\bom)$ and $\bY'\in \pi^{-1}(\bom')$. Let now consider a \emph{reordered} SVD of the square matrix $\bY^{T}\bY'$:
	\begin{equation*}
		\bY^{T}\bY'=\bU\bSig \bV^{T},\quad 
		\bSig=\begin{pmatrix}
		\sigma_p & \dots & 0 \\
		\vdots & \ddots & \vdots \\
		0 & \dots & \sigma_{1}
		\end{pmatrix},\quad \bU,\bV\in \mathrm{O}(p),
	\end{equation*}
	with singular values $0\leq \sigma_p\leq \dots \leq \sigma_{1}$. Define 
	\begin{equation*}
		\widehat{\bY}:=\bY\bU\in \pi^{-1}(\bom),\quad	\widehat{\bY}':=\bY'\bV\in \pi^{-1}(\bom')
	\end{equation*}
	and write 
	\begin{equation*}
		\widehat{\bY}=[\by_1,\dots,\by_p]\in \VecMat{n}{p},\quad \widehat{\bY}'=[\bx_1,\dots,\bx_p]\in \VecMat{n}{p}
	\end{equation*}
	so we can deduce from $\widehat{\bY}^{T}\widehat{\bY}'=\bSig$ the inner products 
	\begin{equation*}
		\langle \by_i,\bx_j\rangle=\sigma_{p-i+1}\delta_{ij},\quad \sigma_{p-i+1}\in [0,1].
	\end{equation*} 
	Using a direct induction on $i$, we obtain a family of orthonormal vectors $\by_{p+1},\dots,\by_{2p}$ such that
	\begin{equation*}
		\bx_i=\cos(\theta_i)\by_i+\sin(\theta_{i})\by_{p+i},\quad \theta_i:=\arccos(\sigma_{p-i+1}),\quad
		\langle \by_i,\by_{p+j}\rangle=0
	\end{equation*}
	which conclude the proof of $(1)$. 
	
	Now, from Remark~\ref{rem:Geodesics_Intrinsic}, any other geodesic from $\bom$ to $\bom'$ reads $t\mapsto \pi(\bY(t))$ with 
	\begin{equation*}
		\bY(t)=\left[\cos(\alpha_{1}t)\by_1+\sin(\alpha_{1}t)\by_{p+1},\dots,\cos(\alpha_{p}t)\by_p+\sin(\alpha_{p}t)\by_{2p}\right],\quad \pi(\bY(1))=\bom'
	\end{equation*}
	so that $\cos(\alpha_{i})=\cos(\theta_i)$ and $\alpha_{i}=\theta_i+k_i\pi$, with $k_i\in \mathbb{Z}$. We deduce that the length of this geodesic is given by
	\begin{equation*}
		\left( \sum_{i=1}^{p} (\theta_i+k_i\pi)^2\right)^{1/2}.
	\end{equation*}
	As $(\theta+k\pi)^2\geq \theta^2$ for all $k\in \mathbb{Z}$ and $\theta\in [0,\pi/2]$, we deduce length minimization for $k_i=0$. Non unicity can only occur if and only if there is non-zero $k_i\in \mathbb{Z}$ such that $\theta_{i}+k_i\pi=-\theta_i$, so that 
	\begin{equation*}
		k_i=\frac{-2\theta_{i}}{\pi}\in \mathbb{Z}-\{0\}
	\end{equation*}
	which translate into $\theta_i=\theta_{i-1}=\dots=\theta_1=\pi/2$, which conclude the proof.
\end{proof}

As a consequence of Theorem~\ref{thm:Min_Geodeiscs_between_Two_Points}, for any two points $\bom$ and $\bom'$ of $\mathcal{G}(p,n)$ the Riemannian distance is given by
\begin{equation}\label{eq:Geodesic_distance_Grassmann}
d_{r}(\bom,\bom')=\bigg(\sum_{i=1}^{p} \theta^2_i \bigg)^{1/2} 
\end{equation}
with $\theta_{i}$ the Jordan's principal angles as defined in the statement of the theorem. Finally, the \emph{diameter} of $\mathcal{G}(p,n)$ (the maximum distance between two points) is given by
\begin{equation}\label{eq:Diam_Grass}
\text{diam}=\sqrt{r}\frac{\pi}{2},\quad r=\min(p,n-p).
\end{equation}

\subsection{Exponential and logarithm map on Grassmann manifolds}

By exploiting geodesics of a Riemannian manifold, it is possible to establish local maps using \emph{normal coordinates}~\cite{Gallot1990} defined from the exponential map. 

In the case of Grassmann manifolds, the exponential map is obtained from the exact formulation of the geodesics (see Theorem~\ref{thm:Eq_Geo_Grass}). 

\begin{defn}[Exponential map]
	For any point $\bom\in \mathcal{G}(p,n)$, let consider the tangent plane $T_{\bom}\mathcal{G}(p,n)\simeq \RR^d$, with $d=p(n-p)$ the dimension of $\mathcal{G}(p,n)$. Then the exponential map is defined by
	\begin{equation}\label{eq:Def_Exp_Map}
		\Exp_{\bom}\,:\, v\in T_{\bom}\mathcal{G}(p,n) \mapsto \pi\left( \bY\bV\cos\bThe+\bU\sin\bThe\right) \in \mathcal{G}(p,n)
	\end{equation}
	where $\bY\in \pi^{-1}(\bom)$ and $\bZ=\bU\bThe\bV$ is a thin SVD of a horizontal lift $\bZ\in \text{Hor}_{\bY}$ of $v$.
\end{defn}

Such a map is only a diffeomorphism \emph{locally}, meaning that there exists some open set $\mathrm{W}\subset T_{\bom}\mathcal{G}(p,n)$ containing $0$ such that $\left(\Exp_{\bom}\right)_{\mid \mathrm{W}}$ is a diffeomorphism, which thus makes it possible to define local coordinates on $\mathrm{W}$. A first way to do so is to consider the \emph{injectivity radius} and thus the open disk:
\begin{equation}\label{eq:Disk_Inj}
\mathrm{D}_{\bom}:=\left\{v\in T_{\bom}\mathcal{G}(p,n),\quad \|v\|<\pi/2\right\},
\end{equation}
where $\pi/2$ is the injectivity radius for Grassmann manifolds~\cite{Kozlov2000}. We obtain here a local map
\begin{equation*}
	\left(\Exp_{\bom}\right)_{\mid_{\mathrm{D}_{\bom}}}\: : \: \mathrm{D}_{\bom} \longrightarrow \Exp_{\bom}\left(\mathrm{D}_{\bom}\right).
\end{equation*}
It turns out that in our case, it is possible to go beyond this injectivity radius. To do so, a logarithm map is directly define at each point of the Grassmann manifold.

First, for any point $\bom\in \mathcal{G}(p,n)$, let us define the open set
\begin{equation}\label{eq:Def_Open_Set_Log}
\mathrm{U}_{\bom}:=\{ \bom'\in \mathcal{G}(p,n),\quad \bY^{T}\bY' \text{ is invertible},\quad \pi(\bY)=\bom,\quad \pi(\bY')=\bom\}.
\end{equation}
A more geometric insight of such an open set is given by a lemma directly deduced from Jordan's principal angles (see Theorem~\ref{thm:Min_Geodeiscs_between_Two_Points}):
\begin{lem}\label{lem:COnd_Angle_Open_set}
	For any $\bom,\bom'\in \mathcal{G}(p,n)$, take $0\leq \theta_p\leq \dots\leq \theta_1\leq\pi/2$ to be their corresponding Jordan's principal angles. Then $\bom'\in\mathrm{U}_{\bom}$ if and only if $\theta_1<\pi/2$. 
\end{lem}
From now on, let us suppose that $2p\leq n$, while the case $2p>n$ is straightforward. 

Following Theorem~\ref{thm:Min_Geodeiscs_between_Two_Points}, we can find an orthonormal family $\by_1,\dots,\by_n$ of $\RR^n$ such that 
\begin{align}\label{eq:Intrinsic_Jordan_Angles}
\bY'&=[\cos(\theta_1)\by_1+\sin(\theta_1)\by_{p+1},\dots,\cos(\theta_p)\by_p+\sin(\theta_p)\by_{2p}]\in \pi^{-1}(\bom'),\\
\bY&=[\by_1,\dots,\by_p]\in \pi^{-1}(\bom), \notag
\end{align}
and then $\bY^T\bY'=\cos\bThe$. The classical definition of the logarithm map~\cite{Mosquera2019b} makes use of a thin SVD of
\begin{equation}\label{eq:SVD_for_Log_Map}
	\bY'\left(\bY^{T}\bY'\right)^{-1}-\bY=[\tan(\theta_{1})\by_{p+1},\dots,\tan(\theta_p)\by_{2p}]
\end{equation}
where singular values are well-defined (as a consequence of Lemma~\ref{lem:COnd_Angle_Open_set}). From all this, it is possible to have the following definition, using the $\arctan$ function:
\begin{defn}[Logarithm map in Grassmann manifolds]\label{def:Log_map_Grass}
	For any $\bom\in \mathcal{G}(p,n)$, take the open set $\mathrm{U}_{\bom}$ defined by~\eqref{eq:Def_Open_Set_Log}. Then the logarithm map at $\bom$ is given by
	\begin{equation*}
		\Log_{\bom}\,: \, \bom'\in \mathrm{U}_{\bom}\mapsto \Log_{\bom}(\bom')\in T_{\bom}\mathcal{G}(p,n)
	\end{equation*}
	where an horizontal lift $\bZ$ of $\Log_{\bom}(\bom')$ is defined using a thin SVD
	\begin{equation*}
		\bY'\left(\bY^{T}\bY'\right)^{-1} -\bY=\bU\bSig \bV^{T},\quad \bY'\in \pi^{-1}(\bom'),
	\end{equation*}
	so that
	\begin{equation*}
		\bZ:=\bU \arctan(\bSig) \bV^{T}.
	\end{equation*}
\end{defn}

As a direct consequence of~\eqref{eq:Intrinsic_Jordan_Angles} and~\eqref{eq:SVD_for_Log_Map}, the horizontal lift $\bZ$ of $v=\Log_{\bom}(\bom')$ encodes the Jordan's principal angles between $\bom$ and $\bom'$, as we can write in the orthonormal basis $\by_1,\dots,\by_n$ of $\RR^n$:
\begin{equation*}
	\bZ=[\theta_{1}\by_{p+1},\dots,\theta_p\by_{2p}].
\end{equation*}
From Remark~\ref{rem:Geodesics_Intrinsic}, we deduce that we have $\Exp_{\bom}(v)=\bom'$, leading to:
\begin{lem}\label{lem:Exp_and_Log}
	For any $\bom\in \mathcal{G}(p,n)$, the map $\Log_{\bom}$ is a diffeomorphism from $\mathrm{U}_{\bom}$ onto $\Log_{\bom}\left(\mathrm{U}_{\bom}\right)$, with inverse map given by the exponential map at $\bom$:
	\begin{equation*}
	\Exp_{\bom}\circ \Log_{\bom}=\emph{\text{id}}_{\mathrm{U}_{\bom}}.
	\end{equation*}
\end{lem}

As a conclusion of this subsection, we obtain here normal coordinates on all the open set $\mathrm{U}_{\bom}$, which is in fact an improvement compare to the open set deduced from the injectivity radius disk, thanks to the lemma:

\begin{lem}\label{lem:Cut_Better_than_Radius}
	For any $\bom\in \mathcal{G}(p,n)$ and $n,p$ such that $\min(p,n-p)\geq 2$, the open set $\mathrm{U}_{\bom}$ given by~\eqref{eq:Def_Open_Set_Log} strictly contains $\Exp_{\bom}\left(\mathrm{D}_{\bom}\right)$, with $\mathrm{D}_{\bom}$ given by~\eqref{eq:Disk_Inj}:
	\begin{equation*}
	\Exp_{\bom}\left(\mathrm{D}_{\bom}\right)\varsubsetneq \mathrm{U}_{\bom}.
	\end{equation*}
\end{lem}

\begin{proof}
	The inclusion follows from Theorem~\ref{thm:Cut_Locus_Grass} as any $v\in \mathrm{D}_{\bom}$ is such that
	\begin{equation*}
	\|v\|<\frac{\pi}{2}.
	\end{equation*}
	To obtain a strict inclusion we follow Remark~\ref{rem:Geodesics_Intrinsic} in the case $2p\leq n$. Let us consider an orthonormal basis $\by_1,\dots,\by_n$ and $v$ with horizontal lift given by
	\begin{equation*}
	\bZ=[\theta_{1}\by_{p+1},\dotsc,\theta_{p}\by_{2p}].
	\end{equation*}
	Then we can find $\theta_1,\dots,\theta_p$ such that
	\begin{equation*}
	\|v\|=\left(\sum \theta_{i}^2\right)^{1/2}\geq \pi/2 \text{ and } \theta_{1}<\pi/2
	\end{equation*}
	using for instance
	\begin{equation*}
	\theta_{i}:=\alpha<\frac{\pi}{2} \text{ with } \frac{\pi}{2} \leq \sqrt{p}\alpha.
	\end{equation*}
\end{proof}

\subsection{Cut--locus and exponential map injectivity on Grassmann manifolds}\label{subsec:Cut_Locus}

In this final subsection, it is proposed to establish the link between the open set $\mathrm{U}_{\bom}$ defined by~\eqref{eq:Def_Open_Set_Log} and the cut--locus of Grassmann manifolds. Such a notion of cut--locus is particularly related to the loss of injectivity of the exponential map. As far as we know, such a result about the cut--locus was suggested in~\cite{Wong1967}, but without any clear proof nor statement.

Let us take back here the geodesic $t\in \RR\mapsto \alpha_{v}(t)$ from~\eqref{eq:geo_v}, with non-zero initial velocity $v\in T_{\bom}\mathcal{G}(p,n)$. Define now
\begin{equation*}
I_{v}:=\{t\in \RR,\quad (\alpha_v)_{\mid_{[0,t]}} \text{ is length minimal}\}=[0,\rho(v)],
\end{equation*}
where $\rho(v)$ is some bounded real number (see~\cite[Section 2.C.7]{Gallot1990}). A first result is given by~\cite[Theorem 3.77]{Gallot1990}:
\begin{thm}\label{thm:Exp_Diff_Cut_Locus}
	Let $\bom\in \mathcal{G}(p,n)$ and
	\begin{equation}\label{eq:Open_set_Cut_locus}
	\mathrm{V}_{\bom}:=\left\{v\in T_{\bom}\mathcal{G}(p,n),\quad \rho(v)>1 \right\}\cup \{0\}.
	\end{equation}
	Then $\mathrm{V}_{\bom}$ is an open neighborhood of $0\in T_{\bom}\mathcal{G}(p,n)$ and the map
	\begin{equation*}
	\left(\Exp_{\bom}\right)_{\mid_{\mathrm{V}_\bom}} \: : \: \mathrm{V}_{\bom} \longrightarrow \Exp_{\bom}(\mathrm{V}_\bom)
	\end{equation*}
	is a diffeomorphism.
\end{thm}
The image of the boundary $\partial\mathrm{V}_{\bom}$ then define the cut-locus:
\begin{defn}[Cut-locus]
	For any point $\bom\in\mathcal{G}(p,n)$, the cut-locus of $\bom$ is given by
	\begin{equation*}
	\text{Cut}(\bom):=\left\{\Exp_{\bom}(\rho(v)v),\quad \|v\|=1 \right\}.
	\end{equation*}
\end{defn}

In the specific case of Grassmann manifolds, there is a way to explicitly obtain the bound $\rho(v)$, while the main ideas are directly taken from~\cite[Theorem 12.5]{Kozlov2000}:
\begin{lem}\label{lem:rho_v_Grass}
	Let $\bom\in \mathcal{G}(p,n)$ and $v\in T_{\bom}\mathcal{G}(p,n)$, with horizontal lift given by some $\bZ\in \VecMat{n}{p}$. Then we have
	\begin{equation*}
	\rho(v)=\frac{\pi}{2\theta_{1}},
	\end{equation*}
	where $\theta_1$ is the maximal singular value of $\bZ$ and thus, taking back the open set $\mathrm{V}_{\bom}$ defined by~\eqref{eq:Open_set_Cut_locus} we have
	\begin{equation}\label{eq:Def_Vm_Angle}
	\mathrm{V}_{\bom}=\left\{v\in T_{\bom}\mathcal{G}(p,n),\quad \theta_{1}<\frac{\pi}{2} \right\}\cup \{0\}.
	\end{equation}
\end{lem}
\begin{proof}
	From Lemma~\ref{lem:Hor_Lift_Intrisic}, we can consider an orthonormal basis $\by_1,\dots,\by_n$ of $\RR^n$ such that $\bY\in \pi^{-1}(\bom)$ and a horizontal lift $\bZ$ of $v$ are given by (for $2p\leq n$):
	\begin{align*}    
	\bY=[\by_1,\dots,\by_{p}],\quad
	\bZ=[\theta_{1}\by_{p+1},\dotsc,\theta_{p}\by_{2p}],
	\end{align*}
	where $0\leq \theta_{p}\leq \dots \leq \theta_{1}$ are the singular values of any horizontal lift of $v$.  
	
	Now, from Theorem~\ref{thm:Min_Geodeiscs_between_Two_Points} the geodesic $\alpha(t)=\pi(\bY(t))$ with
	\begin{equation*}
	\bY(t)=\left[\cos(\theta_{1}t)\by_1+\sin(\theta_{1}t)\by_{p+1},\dots,\cos(\theta_{p}t)\by_p+\sin(\theta_{p}t)\by_{2p}\right]
	\end{equation*}
	is minimal for all $t\leq \pi/(2\theta_1)$, and is not unique anymore for $t=\pi/(2\theta_1)$. From ~\cite[Corollary 2.111]{Gallot1990}, $\alpha$ is no longer minimal on $[0,\pi/(2\theta_1)+\varepsilon]$ for all $\varepsilon>0$, so we can conclude (the proof being the same for $2p>n$). The last equation~\eqref{eq:Def_Vm_Angle} is straightforward.
\end{proof}
Our main result is now:
\begin{thm}\label{thm:Cut_Locus_Grass}
	For any $\bom\in \mathcal{G}(p,n)$ we have
	\begin{equation*}
	\Exp_{\bom}(\mathrm{V}_\bom)=\mathrm{U}_{\bom}
	\end{equation*}
	with $\mathrm{U}_{\bom}$ and $\mathrm{V}_{\bom}$ respectively defined by~\eqref{eq:Def_Open_Set_Log} and~\eqref{eq:Open_set_Cut_locus}. Furthermore the cut-locus at $\bom$ is given by:
	\begin{equation*}
	\emph{\text{Cut}}(\bom)=\left\{ \bom',\quad \bY^{T}\bY' \text{ is singular},\quad \pi(\bY)=\bom,\quad \pi(\bY')=\bom' \right\}.
	\end{equation*}
\end{thm}

\begin{proof}
	Taking back Lemma~\ref{lem:rho_v_Grass} recall that
	\begin{equation*}
	\mathrm{V}_{\bom}=\left\{v\in T_{\bom}\mathcal{G}(p,n),\quad \theta_{1}<\frac{\pi}{2} \right\}\cup \{0\}
	\end{equation*}
	where $\theta_1$ is the maximal singular value of any horizontal lift $\bZ\in \VecMat{n}{p}$ of $v$. Take now any $v\in \mathrm{V}_{\bom}$ and define an orthonormal basis $\by_1,\dots,\by_n$ of $\RR^n$ like in Lemma~\ref{lem:Hor_Lift_Intrisic}, so that for $2p\leq n$
	\begin{equation*}
	\Exp_{\bom}(v)=\pi\left([\cos(\theta_{1})\by_1+\sin(\theta_{1})\by_{p+1},\dots,\cos(\theta_{p})\by_p+\sin(\theta_{p})\by_{2p}] \right),\quad \theta_1<\pi/2.
	\end{equation*}
	From Lemma~\ref{lem:COnd_Angle_Open_set} we deduce that $\Exp_{\bom}(v)\in \mathrm{U}_{\bom}$ and thus $\Exp_{\bom}(\mathrm{V}_{\bom})\subset \mathrm{U}_{\bom}$.
	
	The converse is a direct consequence of Theorem~\ref{thm:Min_Geodeiscs_between_Two_Points} and Lemma~\ref{lem:COnd_Angle_Open_set}, all proof being the same for $2p>n$.   
	
	Finally, the statement for $\text{Cut}(\bom)$ follows in the same way, so we can conclude.  
\end{proof}

\end{document}